\documentclass[12pt]{article}
\usepackage{amssymb}
\usepackage{amsmath}
\usepackage{amsthm}
\usepackage{comment}
\usepackage{enumitem}
\usepackage{graphicx}
\usepackage{mathtools}		
\renewcommand{\labelenumi}{\roman{enumi})}
\setitemize{itemsep=0pt}

\DeclareFontFamily{OT1}{pzc}{}
\DeclareFontShape{OT1}{pzc}{m}{it}{<-> s * [1.10] pzcmi7t}{}
\DeclareMathAlphabet{\mathpzc}{OT1}{pzc}{m}{it}

\begin{document}

\def\cF{\mathcal{F}}
\def\cG{\mathcal{G}}
\def\cN{\mathcal{N}}
\def\co{\mathpzc{o} \hspace{-.5mm}}
\def\cO{\mathcal{O}}
\def\cP{\mathcal{P}}
\def\cS{\mathcal{S}}
\def\cT{\mathcal{T}}
\def\cU{\mathcal{U}}
\def\cW{\mathcal{W}}
\def\cX{\mathcal{X}}
\def\cY{\mathcal{Y}}
\def\defeq{\vcentcolon=}
\def\ee{\varepsilon}

\newcommand{\removableFootnote}[1]{}

\newtheorem{theorem}{Theorem}[section]
\newtheorem{corollary}[theorem]{Corollary}
\newtheorem{lemma}[theorem]{Lemma}
\newtheorem{proposition}[theorem]{Proposition}

\theoremstyle{definition}
\newtheorem{definition}{Definition}[section]
\newtheorem{example}[definition]{Example}

\theoremstyle{remark}
\newtheorem{remark}{Remark}[section]





\title{The structure of mode-locking regions of piecewise-linear continuous maps}
\author{
D.J.W.~Simpson\\\\
Institute of Fundamental Sciences\\
Massey University\\
Palmerston North\\
New Zealand
}
\maketitle

\begin{abstract}

The mode-locking regions of a dynamical system are the subsets of the parameter space of the system
within which there exists an attracting periodic solution.
For piecewise-linear continuous maps, these regions have a curious
chain structure with points of zero width called shrinking points.
In this paper we perform a local analysis about an arbitrary shrinking point.
This is achieved by studying the symbolic itineraries
of periodic solutions in nearby mode-locking regions
and performing an asymptotic analysis on one-dimensional slow manifolds
in order to build a comprehensive theoretical framework for the local dynamics.
We obtain leading-order quantitative descriptions for the shape of nearby mode-locking regions,
the location of nearby shrinking points, and the key properties of these shrinking points.
We apply the results to the three-dimensional border-collision normal form,
nonsmooth Neimark-Sacker-like bifurcations,
and grazing-sliding bifurcations in a model of a dry friction oscillator.

\end{abstract}

\section{Introduction}
\label{sec:intro}
\setcounter{equation}{0}

This paper concerns piecewise-linear continuous maps of the form\removableFootnote{
The dependence of $\xi$ on $B$ can usually be ignored for the purposes of classifying dynamics
by transforming to the observer canonical form.
}
\begin{equation}
x_{i+1} = f(x_i;\mu,\xi) \defeq
\begin{cases}
A_L(\xi) x_i + B(\xi) \mu \;, & s_i \le 0 \\
A_R(\xi) x_i + B(\xi) \mu \;, & s_i \ge 0
\end{cases} \;,
\label{eq:f}
\end{equation}
where $x_i \in \mathbb{R}^N$ ($N \ge 2$) 
and $s_i$ denotes the first component of $x_i$, i.e.,
\begin{equation}
s_i \defeq e_1^{\sf T} x_i \;.
\label{eq:s}
\end{equation}
In (\ref{eq:f}), $A_L$ and $A_R$ are real-valued $N \times N$ matrices and $B \in \mathbb{R}^N$.
$A_L$, $A_R$ and $B$ are $C^K$ functions of a parameter $\xi \in \mathbb{R}^M$,
and $\mu \in \mathbb{R}$ is another parameter.

The assumption that (\ref{eq:f}) is continuous on the switching manifold $s=0$ implies that
$A_L$ and $A_R$ differ in only their first columns, i.e.,\removableFootnote{
I'm not sure that I ever actually use the symbol $C$,
but I certainly use (\ref{eq:continuityCondition}) in various places below.
}
\begin{equation}
A_R = A_L + C e_1^{\sf T} \;,
\label{eq:continuityCondition}
\end{equation}
for some $C \in \mathbb{R}^N$.
Furthermore, (\ref{eq:f}) satisfies the linear scaling property
\begin{equation}
f(\gamma x;\gamma \mu,\xi) \equiv \gamma f(x;\mu,\xi) \;,
\label{eq:fscaling}
\end{equation}
for any $\gamma > 0$.
For this reason, the structure of the dynamics of (\ref{eq:f})
is independent of the magnitude of $\mu$,
and the size of any bounded invariant set of (\ref{eq:f}) is proportional to $|\mu|$.

Maps of the form (\ref{eq:f}) arise in diverse contexts.
The tent map and the Lozi map, one and two-dimensional examples
of (\ref{eq:f}), are instructive prototypical maps exhibiting chaos \cite{Lo78,Mi80}.
Maps that can be put in the form (\ref{eq:f}) through a change of variables
have been used to model phenomena involving a switch or abrupt event,
particularly in social sciences \cite{PuSu06}.
Most importantly,
maps of the form (\ref{eq:f}) describe the dynamics near border-collision bifurcations.

A border-collision bifurcation occurs when a fixed point of a piecewise-smooth map
collides with a switching manifold under certain regularity conditions \cite{NuYo92,DiBu08,Si15}.
Except in special cases, (\ref{eq:f}) has a border-collision bifurcation at $\mu = 0$.
The dynamics of (\ref{eq:f}) for $\mu < 0$ and $\mu > 0$ represent the dynamics on either side of the bifurcation.
In this context, $\mu$ is the primary bifurcation parameter. 
By varying $\xi$ in a continuous fashion,
we can investigate how the dynamics created in the border-collision bifurcation
changes with respect to other parameters.

This paper concerns mode-locking regions of (\ref{eq:f}).
A mode-locking region of a map is a region of parameter space within which
the map has an attracting periodic solution of a given period.
We consider two-dimensional cross-sections of parameter space
as these are simple to visualise and informative.
For (\ref{eq:f}) we always consider cross-sections with fixed $\mu \ne 0$,
so that we avoid the border-collision bifurcation at $\mu = 0$ and 
degeneracies due to the scaling property (\ref{eq:fscaling}).

Fig.~\ref{fig:modeLockEx20} shows an example using
\begin{equation}
A_L = \begin{bmatrix}
\tau_L & 1 & 0 \\
-\sigma_L & 0 & 1 \\
\delta_L & 0 & 0
\end{bmatrix} \;, \qquad
A_R = \begin{bmatrix}
\tau_R & 1 & 0 \\
-\sigma_R & 0 & 1 \\
\delta_R & 0 & 0
\end{bmatrix} \;, \qquad
B = \begin{bmatrix}
1 \\ 0 \\ 0
\end{bmatrix} \;,
\label{eq:ALAREx20}
\end{equation}
where $\tau_L,\sigma_L,\delta_L,\tau_R,\sigma_R,\delta_R \in \mathbb{R}$.
The map (\ref{eq:f}) with (\ref{eq:ALAREx20}) is the
border-collision normal form in three dimensions \cite{DiBu08,Si15}.
Here parameter space (not including $\mu$) is six-dimensional
(we could write $\xi = (\tau_L,\sigma_L,\delta_L,\tau_R,\sigma_R,\delta_R)$).
The two-dimensional cross-section of parameter space used in Fig.~\ref{fig:modeLockEx20}
is defined by the restriction
\begin{equation}
\begin{gathered}
\tau_L = 0 \;, \qquad
\sigma_L = -1 \;, \qquad
\sigma_R = 0 \;, \qquad
\delta_R = 2 \;.
\end{gathered}
\label{eq:paramEx20}
\end{equation}
The mode-locking regions are coloured by the period, $n$,
of the corresponding stable periodic solution.
Only mode-locking regions up to $n = 50$ are shown.
The periodic solutions are ``rotational'', in a symbolic sense defined in \S\ref{sub:rss},
and can be assigned a rotation number, $\frac{m}{n}$.
As we move from left to right across Fig.~\ref{fig:modeLockEx20},
the rotation number decreases roughly monotonically.
For the parameters of Fig.~\ref{fig:modeLockEx20}
there also exist mode-locking regions corresponding to non-rotational periodic solutions,
but these regions are small, relative to the rotational ones,
and not shown in Fig.~\ref{fig:modeLockEx20} or studied in this paper.

\begin{figure}[t!]
\begin{center}
\setlength{\unitlength}{1cm}
\begin{picture}(15,5.2)
\put(0,0){\includegraphics[height=5cm]{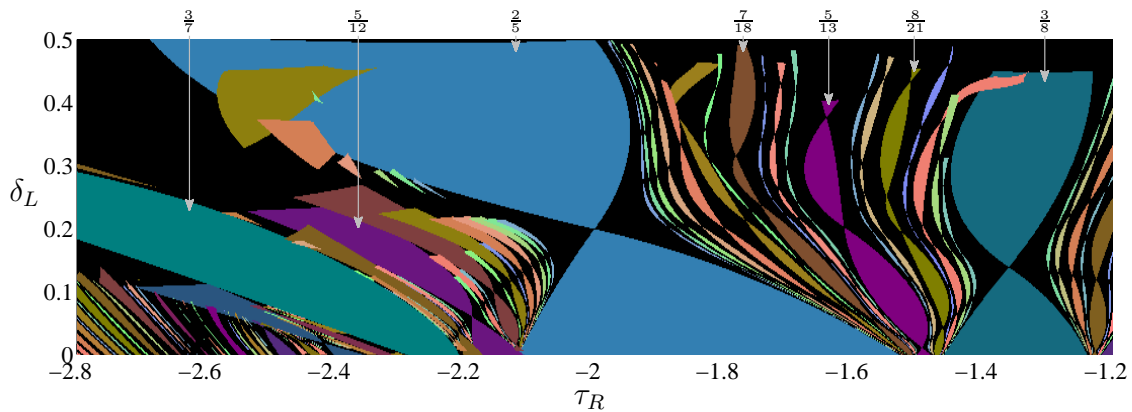}}
\put(7.5,0){$\tau_R$}
\put(0,2.7){$\delta_L$}
\put(2.28,5.04){\tiny $\frac{3}{7}$}
\put(4.47,5.04){\tiny $\frac{5}{12}$}
\put(6.63,5.04){\tiny $\frac{2}{5}$}
\put(9.59,5.04){\tiny $\frac{7}{18}$}
\put(10.72,5.04){\tiny $\frac{5}{13}$}
\put(11.88,5.04){\tiny $\frac{8}{21}$}
\put(13.66,5.04){\tiny $\frac{3}{8}$}
\end{picture}
\caption{
Mode-locking regions of (\ref{eq:f}) with (\ref{eq:ALAREx20})-(\ref{eq:paramEx20}) and $\mu > 0$
corresponding to rotational periodic solutions (defined in \S\ref{sub:definitions}) of period $n \le 50$.
Selected mode-locking regions are labelled by their rotation number, $\frac{m}{n}$.
This figure was computed by numerically checking the admissibility and stability of
rotational periodic solutions on a $1024 \times 256$ grid of $\tau_R$ and $\delta_L$ values.
At grid points where multiple stable periodic solutions exist,
the periodic solution with the highest period (less than $50$) is indicated.
\label{fig:modeLockEx20}
}
\end{center}
\end{figure}

The majority of the boundaries of the mode-locking regions shown in Fig.~\ref{fig:modeLockEx20}
are where one point of the corresponding periodic solution lies on the switching manifold ($s=0$).
These boundaries may be viewed as curves of border-collision bifurcations for
the $n^{\rm th}$-iterate of (\ref{eq:f}) \cite{Si15}.
The boundaries intersect at so-called shrinking points
where (\ref{eq:f}) has a periodic solution with two points on the switching manifold.

Many of the mode-locking regions of Fig.~\ref{fig:modeLockEx20}
have several shrinking points and an overall structure
that loosely resembles a string of sausages.
This structure was first identified in piecewise-linear circle maps \cite{YaHa87,CaGa96},
and subsequently described in piecewise-smooth models of a DC-DC power converter \cite{ZhSo02}
and a trade cycle \cite{GaGa03}.
The structure has also been identified in an integrate-and-fire neuron model \cite{Ti02},
a model of an oscillator subject to dry friction \cite{SzOs09},
and a model of the synchronisation of breathing to heart rate \cite{McHo04}.

A rigorous study of shrinking points of (\ref{eq:f}) was performed in \cite{SiMe09}.
Here it was shown that the four curves of border-collision bifurcations
that bound the mode-locking region near a shrinking point
correspond to where four different points of a periodic solution lie on the switching manifold.
These four points admit a simple characterisation
in terms of the symbolic itinerary of the periodic solution.
Furthermore, a local analysis reveals that
as the boundaries emanate from the shrinking point,
they curve in a manner that favours the boundaries reintersecting at other shrinking points.

However, the results of \cite{SiMe09}
do not provide us with any information about the dynamics of (\ref{eq:f})
outside the associated mode-locking region.
In any neighbourhood of a typical shrinking point
there are infinitely many mode-locking regions corresponding to higher periods.
These are not readily apparent in Fig.~\ref{fig:modeLockEx20}, as only periods up to $50$ are shown,
but about shrinking points corresponding to relatively low periods,
we can see the beginnings of sequences of mode-locking regions converging to the shrinking points.

The purpose of this paper is to describe nearby mode-locking regions and their associated shrinking points.
This is achieved by combining Farey addition, used to identify the symbolic itineraries of periodic solutions
in nearby mode-locking regions, with calculations on one-dimensional slow manifolds
in order to develop a comprehensive theoretical framework that explains the dynamics.
To benefit the reader we begin in \S\ref{sec:results} by stating the main results.
This requires introducing some notation and definitions.
This is done as quickly as possible,
skipping over derivations and explanations which are provided in later sections.
In \S\ref{sec:results} we also demonstrate the scope of the results with a variety of examples.

In the subsequent three sections we provide an array of small mathematical results, many of which are new,
that can be viewed as building blocks used to obtain the main results.
Section \ref{sec:symbol} concerns the symbolic itineraries.
Each shrinking point is associated with a rotational symbol sequence.
By partitioning this sequence into blocks, and repeating the blocks appropriately,
we can construct the symbol sequences of periodic solutions in nearby mode-locking regions.
These sequences are also obtained via Farey addition.
In \S\ref{sec:periodicSolns} we combine symbolic representations with matrix algebra
in order to compute periodic solutions.
In \S\ref{sec:shrPoints} we describe fundamental properties of shrinking points
and review the basic unfolding of shrinking points that was derived in \cite{SiMe09}.

In \S\ref{sec:locations} we identify the locations of nearby shrinking points,
and in \S\ref{sec:properties} determine properties of these shrinking points.
The key geometric tool used to obtain the results is that of a slow manifold.
At a shrinking point, periodic dynamics are associated with $N-1$ stable directions
and one neutral direction (corresponding to a unit eigenvalue).
Near a shrinking point there are consequently
one-dimensional slow manifolds (each point of the periodic solution
has an associated slow manifold).
By working on these slow manifolds our calculations reduce from $N$ dimensions to one dimension.

Finally, in \S\ref{sec:conc} we provide concluding comments.
Some proofs are given in Appendix \ref{app:proofs}.

\section{Main results}
\label{sec:results}
\setcounter{equation}{0}

In this section we present the main results.
We begin in \S\ref{sub:definitions} by briefly introducing the essential symbolic concepts,
notation used for shrinking points,
and the shrinking point unfolding theorem of \cite{SiMe09}.
In \S\ref{sub:theorems} we state the main results as four theorems,
and in \S\ref{sub:examples} we illustrate the theorems with a variety of examples.

\subsection{Basic definitions and properties of shrinking points}
\label{sub:definitions}

Let
\begin{equation}
f^L(x;\mu,\xi) \defeq A_L(\xi) x + B(\xi) \mu \;, \qquad
f^R(x;\mu,\xi) \defeq A_R(\xi) x + B(\xi) \mu \;,
\label{eq:fJ}
\end{equation}
denote the two affine half-maps of (\ref{eq:f}).
As in \cite{Si15,Si10,SiMe10}, we work with symbol sequences, $\cS : \mathbb{Z} \to \{ L,R \}$,
and match periodic solutions of (\ref{eq:f}) to periodic symbol sequences.
This is made precise by the following definition.
(In this definition, and throughout this paper, $\cS_i$ denotes the $i^{\rm th}$ element of $\cS$.)

\begin{definition}
Let $\cS$ be a periodic symbol sequence of period $n$.
We refer to an $n$-tuple, $\{ x^{\cS}_i \}_{i=0}^{n-1}$,
satisfying $x^{\cS}_{(i+1) {\rm \,mod\,} n} = f^{\cS_i} \left( x^{\cS}_i \right)$,
for all $i = 0,\ldots,n-1$,
as an {\em $\cS$-cycle}.
\label{df:Scycle}
\end{definition}

As explained in \S\ref{sub:periodicSolns},
if each $x^{\cS}_i$ lies on the ``correct'' side of the switching manifold (or on the switching manifold),
then the $\cS$-cycle is a periodic solution of (\ref{eq:f}) and said to be admissible.
Given a periodic symbol sequence $\cS$ of period $n$, let
\begin{equation}
f^{\cS} \defeq f^{\cS_{n-1}} \circ \cdots \circ f^{\cS_0} \;.
\label{eq:fS}
\end{equation}
A straight-forward expansion leads to
\begin{equation}
f^{\cS}(x) = M_{\cS} x + P_{\cS} B \mu \;,
\label{eq:fS2}
\end{equation}
where
\begin{align}
M_{\cS} &\defeq A_{\cS_{n-1}} \cdots A_{\cS_0} \;, \label{eq:MS} \\
P_{\cS} &\defeq I + A_{\cS_{n-1}} + A_{\cS_{n-1}} A_{\cS_{n-2}} + \cdots +
A_{\cS_{n-1}} \cdots A_{\cS_1} \;. \label{eq:PS}
\end{align}
Each $x^{\cS}_i$ is a fixed point of $f^{\cS^{(i)}}$,
where we use $\cS^{(i)}$ to denote the $i^{\rm th}$ left shift permutation of $\cS$.
If $\det \left( I - M_{\cS} \right) \ne 0$,
then the $\cS$-cycle is unique and
\begin{equation}
s^{\cS}_i = \frac{\det \left( P_{\cS^{(i)}} \right) \varrho^{\sf T} B \mu}
{\det \left( I - M_{\cS} \right)} \;,
\label{eq:sFormulaEarly}
\end{equation}
where $\varrho^{\sf T} \defeq e_1^{\sf T} {\rm adj} \left( I - A_L \right)$.
In view of (\ref{eq:sFormulaEarly}),
it is useful to treat a mode-locking region boundary on which $s^{\cS}_i = 0$
as a curve on which $\det \left( P_{\cS^{(i)}} \right) = 0$ \cite{Si15,Si10}.

\begin{definition}
Given $\ell,m,n \in \mathbb{Z}^+$,
with $\ell < n$, $m < n$ and ${\rm gcd}(m,n) = 1$,
we define a symbol sequence $\cF[\ell,m,n] : \mathbb{Z} \to \{ L,R \}$ by
\begin{equation}
\cF[\ell,m,n]_i \defeq
\begin{cases}
L \;, & i m {\rm ~mod~} n < \ell \\
R \;, & i m {\rm ~mod~} n \ge \ell
\end{cases} \;.
\label{eq:rss}
\end{equation}
We say that $\cF[\ell,m,n]$, and any shift permutation of $\cF[\ell,m,n]$,
is {\em rotational}.
\label{df:rss}
\end{definition}

We refer to $\frac{m}{n}$ as the rotation number of $\cF[\ell,m,n]$.
The requirement ${\rm gcd}(m,n) = 1$ ensures that $\frac{m}{n}$ is
an irreducible fraction and that each $\cF[\ell,m,n]$ is of period $n$ \cite{Si10}.
Throughout this paper we let $d \in \{ 1,\ldots,n-1 \}$
denote the multiplicative inverse of $m$ modulo $n$
(i.e.~$m d {\rm ~mod~} n = 1$). 
For brevity we omit ``${\rm mod~} n$'' in subscripts
where it is clear that modulo arithmetic is being used.

Next we provide a definition of a shrinking point of (\ref{eq:f}).
Each shrinking point corresponds to a particular rotational symbol sequence, $\cS = \cF[\ell,m,n]$,
and is referred to as an $\cS$-shrinking point.
As explained in \S\ref{sub:shrPoints},
at an $\cS$-shrinking point there are infinitely many $\cS$-cycles. 
For this reason our definition of an $\cS$-shrinking point refers to an $\cS^{\overline{0}}$-cycle
(which is unique),
where we use $\cS^{\overline{i}}$ to denote the symbol sequence
that differs from $\cS$ in only the indices $i + j n$, for all $j \in \mathbb{Z}$.
It is important to note that $\cS^{\overline{0}}$ is a shift permutation of $\cF[\ell-1,m,n]$,
and $\cS^{\overline{\ell d}} = \cF[\ell+1,m,n]$, see \S\ref{sub:rss}.

\begin{definition}
Consider (\ref{eq:f}) for some $\xi \in \mathbb{R}^M$ and $\mu \ne 0$,
and suppose $\varrho^{\sf T} B \ne 0$.
Let $\cS = \cF[\ell,m,n]$ be a rotational symbol sequence with $2 \le \ell \le n-2$.
Suppose $\det \left( I-M_{\cS^{\overline{0}}} \right) \ne 0$
and $\det \left( I-M_{\cS^{\overline{\ell d}}} \right) \ne 0$\removableFootnote{
We do need this assumption,
as there exist codimension-three points where we also have
$\det \left( I-M_{\cS^{\overline{\ell d}}} \right) = 0$.
}.
If the $\cS^{\overline{0}}$-cycle is admissible,
and $s^{\cS^{\overline{0}}}_i = 0$ only for $i = 0$ and $i = \ell d$\removableFootnote{
In \cite{SiMe09,SiMe10} we do not require $s^{\cS^{\overline{0}}}_i \ne 0$
for other values of $i$ in the definition.
We include this requirement here because it is needed
for most of the results in this paper, as well as the basic unfolding theorem.
I have not mentioned this in the text as it is a bit technical.
An alternative is to call this a ``non-degenerate shrinking point''.
},
then we say that $\xi$ is an {\em $\cS$-shrinking point}.
\label{df:shrPoint}
\end{definition}

In Definition \ref{df:shrPoint},
$s^{\cS^{\overline{0}}}_0 = 0$ and $s^{\cS^{\overline{0}}}_{\ell d} = 0$
are the two codimension-$1$ conditions that specify an $\cS$-shrinking point.
The remaining conditions of the definition ensure genericity.
At an $\cS$-shrinking point, we let
\begin{align}
y_i &\defeq x^{\cS^{\overline{0}}}_i \;, &
t_i &\defeq s^{\cS^{\overline{0}}}_i \;, \label{eq:yitiDef} \\
a &\defeq \det \left( I-M_{\cS^{\overline{0}}} \right) \;, &
b &\defeq \det \left( I-M_{\cS^{\overline{\ell d}}} \right) \;. \label{eq:ab}
\end{align}
As shown in \S\ref{sub:furtherProperties}, we always have $a b < 0$.

At an $\cS$-shrinking point, $M_{\cS}$ has a unit eigenvalue of algebraic multiplicity $1$
(see \S\ref{sub:shrPoints}).
It follows that the same is true for each $M_{\cS^{(i)}}$.
For each of the four indices $j = 0$, $(\ell-1)d$, $\ell d$ and $-d$ (taken modulo $n$),
we let $u_j^{\sf T}$ and $v_j$ denote the left and right eigenvectors of $M_{\cS^{(j)}}$
corresponding to the unit eigenvalue and
normalised by $u_j^{\sf T} v_j = 1$ and $e_1^{\sf T} v_j = 1$.
The restriction to the four given indices ensures this normalisation
can always be achieved (see \S\ref{sub:eigenvectors}).
As shown in \S\ref{sub:furtherProperties}, the eigenvectors satisfy
\begin{equation}
\frac{u_0^{\sf T} v_{-d}}{a} +
\frac{u_{\ell d}^{\sf T} v_{(\ell-1)d}}{b} = 
\frac{u_{(\ell-1)d}^{\sf T} v_{\ell d}}{a} +
\frac{u_{-d}^{\sf T} v_0}{b} = \frac{1}{c} \;,
\label{eq:uvIdentity12early}
\end{equation}
where $c$ denotes the product of the nonzero eigenvalues of $I-M_{\cS}$, that is
\begin{equation}
c \defeq \prod_{i=2}^N \lambda_i \;,
\label{eq:c}
\end{equation}
where $\lambda_i$ are the eigenvalues of $I-M_{\cS}$, counting multiplicity, and $\lambda_1 = 0$.

We now consider the properties of (\ref{eq:f}) near an $\cS$-shrinking point.
For simplicity we assume $\xi \in \mathbb{R}^2$ and write $\xi = (\xi_1,\xi_2)$.
Suppose that (\ref{eq:f}) has an $\cS$-shrinking point at some $\xi = \xi^*$.
It follows that there exists a neighbourhood of $\xi^*$ within which
the $\cS^{\overline{0}}$-cycle exists and is unique.
By Definition \ref{df:shrPoint},
$s^{\cS^{\overline{0}}}_0 = s^{\cS^{\overline{0}}}_{\ell d} = 0$ at $\xi = \xi^*$.
We let
\begin{equation}
\eta \defeq s^{\cS^{\overline{0}}}_0(\xi_1,\xi_2) \;, \qquad
\nu \defeq s^{\cS^{\overline{0}}}_{\ell d}(\xi_1,\xi_2) \;,
\label{eq:etanu}
\end{equation}
and assume that the coordinate change $(\xi_1,\xi_2) \to (\eta,\nu)$ is locally invertible,
i.e.~$\det(J) \ne 0$, where
\begin{equation}
J \defeq \left. \begin{bmatrix}
\frac{\partial \eta}{\partial \xi_1} &
\frac{\partial \eta}{\partial \xi_2} \\
\frac{\partial \nu}{\partial \xi_1} &
\frac{\partial \nu}{\partial \xi_2}
\end{bmatrix} \right|_{\xi = \xi^*} \;.
\label{eq:J}
\end{equation}
In $(\eta,\nu)$-coordinates the $\cS$-shrinking point $\xi = \xi^*$
is located at $(\eta,\nu) = (0,0)$.

\begin{figure}[t!]
\begin{center}
\setlength{\unitlength}{1cm}
\begin{picture}(10.33,8)
\put(0,0){\includegraphics[height=8cm]{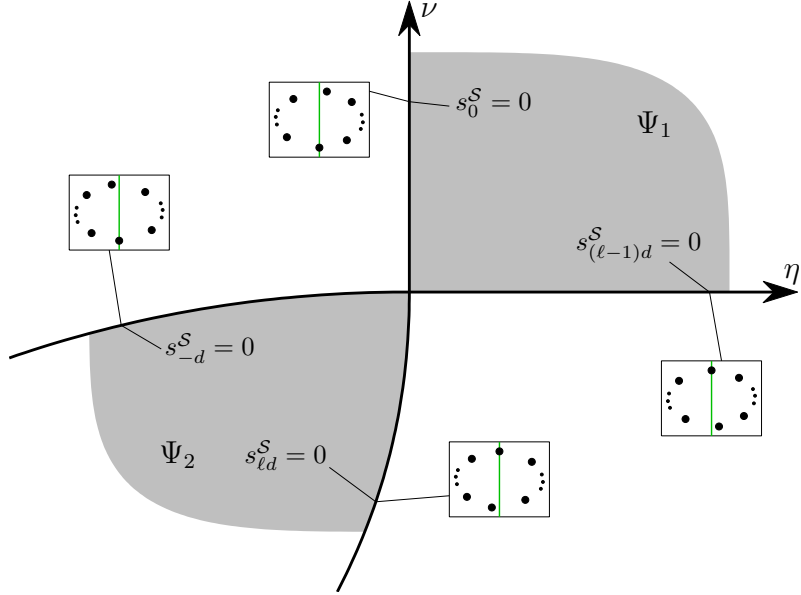}}
\put(10.3,4.2){$\eta$}
\put(5.47,7.67){$\nu$}
\put(8.33,6.13){$\Psi_1$}
\put(2,1.73){$\Psi_2$}
\put(5.93,6.43){\footnotesize $s^{\cS}_0 = 0$}
\put(7.53,4.57){\footnotesize $s^{\cS}_{(\ell-1)d} = 0$}
\put(3.13,1.74){\footnotesize $s^{\cS}_{\ell d} = 0$}
\put(2.07,3.17){\footnotesize $s^{\cS}_{-d} = 0$}
\end{picture}
\caption{
The basic unfolding of a generic $\cS$-shrinking point of (\ref{eq:f}),
as explained in the surrounding text and in more detail in \S\ref{sub:unfolding}.
The $(\eta,\nu)$-coordinates (\ref{eq:etanu})
provide a two-dimensional cross-section of parameter space for which
two of the border-collision bifurcation curves coincide with the coordinate axes.
The insets are schematic phase portraits indicating the location of
the points of the $\cS$-cycle relative to the switching manifold.
\label{fig:shrPointUnfoldingBasic}
}
\end{center}
\end{figure}

As shown in \cite{SiMe09},
four border-collision bifurcation curves emanate from $(\eta,\nu) = (0,0)$.
Each curve corresponds to the existence of an $\cS$-cycle with $s^{\cS}_j = 0$,
for $j = 0$, $(\ell-1)d$, $\ell d$ and $-d$.
The curves are orientated as in Fig.~\ref{fig:shrPointUnfoldingBasic}
and locally define two regions, $\Psi_1$ and $\Psi_2$.
In $\Psi_1$ there exist unique $\cF[\ell,m,n]$ and $\cF[\ell-1,m,n]$-cycles,
and in $\Psi_2$ there exist unique $\cF[\ell,m,n]$ and $\cF[\ell+1,m,n]$-cycles.
If, in both $\Psi_1$ and $\Psi_2$, one of the two periodic solutions is stable,
then (\ref{eq:f}) has a mode-locking region with zero width at $\xi = \xi^*$.
In this case, the $\cF[\ell,m,n]$-cycle is stable on exactly one side of the shrinking point,
as determined by the sign of $a$, Table \ref{tb:stability}.

Finally let $\sigma \ge 0$ denote the maximum of the moduli of the eigenvalues of $M_{\cS}$,
excluding the unit eigenvalue\removableFootnote{
Note: $\sigma$ is a constant defined at the shrinking point
(not a function of $\eta$ and $\nu$).
},
at the $\cS$-shrinking point.
That is,
\begin{equation}
\sigma \defeq \max_{i = 2,\ldots,N} | \rho_i | \;,
\label{eq:sigma}
\end{equation}
where $\rho_i$ are the eigenvalues of $M_{\cS}$, counting multiplicity, and $\rho_1 = 1$.
In general the stability multipliers of an $\cS$-cycle are the eigenvalues of $M_{\cS}$,
\S\ref{sub:periodicSolns}.
Consequently we must have $\sigma < 1$ in order for there to exist a stable $\cS$-cycle
for some parameter values near the $\cS$-shrinking point\removableFootnote{
The assumption $\sigma < 1$ is included in the theorems
not so that periodic solutions are stable,
but rather so that the slow manifolds are attracting
and thus we can expect orbits to regular return to the slow manifolds
and thus we can prove something about the dynamics by looking near the slow manifolds.
}.

\begin{table}[t!]
\begin{center}
\begin{tabular}{c|c|c|}
& $a<0$ & $a>0$ \\
\cline{1-3}
$\cF[\ell,m,n]$-cycle in $\Psi_1$ & stable & unstable \\
$\cF[\ell-1,m,n]$-cycle in $\Psi_1$ & unstable & stable \\
\cline{1-3}
$\cF[\ell,m,n]$-cycle in $\Psi_2$ & unstable & stable \\
$\cF[\ell+1,m,n]$-cycle in $\Psi_2$ & stable & unstable \\
\cline{1-3}
\end{tabular}
\caption{
Cases for the stability of periodic solutions near an $\cS$-shrinking point,
where $\cS = \cF[\ell,m,n]$,
in the scenario that stable periodic solutions exist on both sides of the shrinking point.
\label{tb:stability}
}
\end{center}
\end{table}

\subsection{Theorems for nearby mode-locking regions}
\label{sub:theorems}

Each mode-locking region of Fig.~\ref{fig:modeLockEx20}
corresponds to stable $\cF[\ell,m,n]$-cycles with fixed values of $m$ and $n$,
and values of $\ell$ that change by one each time we cross a shrinking point.
Nearby mode-locking regions have rotation numbers close to $\frac{m}{n}$. 
This motivates the following definition.

\begin{definition}
Given $k \in \mathbb{Z}^+$, $\chi \in \mathbb{Z}$ with $| \chi | < k$,
and a rotational symbol sequence $\cF[\ell,m,n]$, we let
\begin{equation}
\cG^\pm[k,\chi] \defeq
\cF \left[ \ell_k^\pm + \chi, m_k^\pm, n_k^\pm \right] \;,
\label{eq:Gplusminusearly}
\end{equation}
where
\begin{equation}
\ell_k^\pm \defeq k \ell + \ell^\pm \;, \qquad
m_k^\pm \defeq k m + m^\pm \;, \qquad
n_k^\pm \defeq k n + n^\pm \;,
\label{eq:lmnkpm}
\end{equation}
and $\frac{m^-}{n^-}$ and $\frac{m^+}{n^+}$ are the left and right Farey roots of $\frac{m}{n}$,
$\ell^+ \defeq \left\lceil \frac{\ell n^+}{n} \right\rceil$, and
$\ell^- \defeq \left\lfloor \frac{\ell n^-}{n} \right\rfloor$.
We also let $\tilde{\ell} \defeq \ell_k^\pm + \chi$,
and let $d_k^\pm$ denote the multiplicative inverse of $m_k^\pm$ modulo $n_k^\pm$.
\label{df:cG}
\end{definition}

Each $\cG^\pm[k,\chi]$ is a rotational symbol sequence
with rotation number $\frac{m_k^\pm}{n_k^\pm} = \frac{k m + m^\pm}{k n + n^\pm}$.
These rotation numbers limit to $\frac{m}{n}$, as $k \to \infty$,
and are in the first level of complexity relative to $\frac{m}{n}$ \cite{ZhMo06b}.
Other rotational symbol sequences with rotation numbers near $\frac{m}{n}$
have rotation numbers of higher levels of complexity
and are beyond the scope of this paper.

\begin{figure}[t!]
\begin{center}
\setlength{\unitlength}{1cm}
\begin{picture}(12,6)
\put(0,0){\includegraphics[height=6cm]{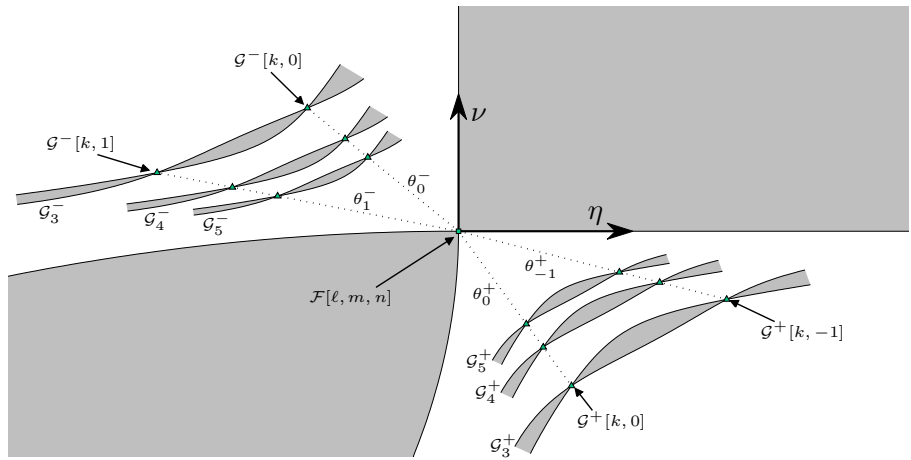}}
\put(7.7,3.17){\small $\eta$}
\put(6.15,4.47){\small $\nu$}
\put(6.17,2.1){\tiny $\theta^+_0$}
\put(6.86,2.47){\tiny $\theta^+_{-1}$}
\put(5.3,3.6){\tiny $\theta^-_0$}
\put(4.58,3.37){\tiny $\theta^-_1$}
\put(7.56,.41){\tiny $\cG^+[k,0]$}
\put(10,1.6){\tiny $\cG^+[k,-1]$}
\put(3,5.22){\tiny $\cG^-[k,0]$}
\put(.5,4.11){\tiny $\cG^-[k,1]$}
\put(6.08,1.22){\tiny $\cG^+_5$}
\put(6.2,.8){\tiny $\cG^+_4$}
\put(6.4,.09){\tiny $\cG^+_3$}
\put(2.58,3.03){\tiny $\cG^-_5$}
\put(1.8,3.1){\tiny $\cG^-_4$}
\put(.4,3.2){\tiny $\cG^-_3$}
\put(4,2.06){\tiny $\cF[\ell,m,n]$}
\end{picture}
\caption{
A sketch of typical $\cG^\pm_k$-mode-locking regions (for $k = 3,4,5$)
near an $\cF[\ell,m,n]$-shrinking point
in $(\eta,\nu)$-coordinates in the case $a < 0$.
Each shrinking point is labelled by its associated symbol sequence.
Formulas for the angles, denoted $\theta^\pm_{\chi}$, 
about which sequences of $\cG^\pm[k,\chi]$-shrinking points
emanate from the $\cF[\ell,m,n]$-shrinking point are given by (\ref{eq:thetaPlus})-(\ref{eq:thetaMinus}).
In the case $a > 0$, the relative location of the $\cG^+_k$ and $\cG^-_k$-mode-locking regions is reversed.
\label{fig:shrPointSchem}
}
\end{center}
\end{figure}

Near a typical shrinking point, 
there exist mode-locking regions corresponding
to $\cG^\pm[k,\chi]$-cycles for several consecutive values of $\chi$, Fig.~\ref{fig:shrPointSchem}.
We refer to these as $\cG^\pm_k$-mode-locking regions.

For any $\chi_{\rm max} \in \mathbb{Z}^+$, we define the collection
\begin{align}
\Xi_{\chi_{\rm max}} &\defeq
\left\{ \cG^+[k,\chi] ~\big|~ k \in \mathbb{Z}^+,
-\chi_{\rm max} \le \chi < \chi_{\rm max}, |\chi| < k \right\} \nonumber \\
&\quad\cup \left\{ \cG^-[k,\chi] ~\big|~ k \in \mathbb{Z}^+,
-\chi_{\rm max} < \chi \le \chi_{\rm max}, |\chi| < k \right\} \;.
\end{align}
We define polar coordinates $(r,\theta)$ by\removableFootnote{
It is interesting that this puts $\det \left( I - M_{\cS} \right)$ on the $135^\circ$ line.
}
\begin{equation}
\eta = \left| \frac{c t_d}{a} \right| r \cos(\theta) \;, \qquad
\nu = \left| \frac{c t_{(\ell-1)d}}{a} \right| r \sin(\theta) \;.
\label{eq:polarCoords}
\end{equation}
We also define a continuous function $\Gamma : \left( 0, \frac{\pi}{2} \right) \to \mathbb{R}$ by
\begin{equation}
\Gamma(\theta) \defeq \begin{cases}
\frac{\ln \left( \cos(\theta) \right) - \ln \left( \sin(\theta) \right)}
{\cos(\theta) - \sin(\theta)} \;, & \theta \in \left( 0, \frac{\pi}{2} \right)
\setminus \left\{ \frac{\pi}{4} \right\} \\
\sqrt{2} \;, & \theta = \frac{\pi}{4} 
\end{cases} \;,
\label{eq:Gamma}
\end{equation}
and extend this definition to all non-integer multiples of $\frac{\pi}{2}$ in a periodic fashion:
\begin{equation}
\Gamma(\theta) \defeq \Gamma \left( \theta {\rm ~mod~} \frac{\pi}{2} \right) \;, \quad
\theta \ne \frac{j \pi}{2} \;, \quad 
j \in \mathbb{Z} \;.
\label{eq:GammaExtended}
\end{equation}
The following theorem provides us with the location of the $\cG^\pm_k$-mode-locking regions
to leading order.

\begin{theorem}
Suppose (\ref{eq:f}) with $K \ge 2$ has an $\cS$-shrinking point
satisfying $\sigma < 1$ and $\det(J) \ne 0$,
and write $\cS = \cF[\ell,m,n]$.
Then for all $\chi_{\rm max} \in \mathbb{Z}^+$,
there exists $k_{\rm min} \in \mathbb{Z}^+$
and a neighbourhood $\cN$ of $(\eta,\nu) = (0,0)$,
such that for all $\cT = \cG^\pm[k,\chi] \in \Xi_{\chi_{\rm max}}$ with $k \ge k_{\rm min}$,
within $\cN$ there exists a unique $C^K$ curve on which $\det \left( P_{\cT} \right) = 0$
and a unique $C^K$ curve on which
$\det \left( P_{\cT^{\left( \left( \tilde{\ell}-1 \right) d_k^\pm \right)}} \right) = 0$,
and both curves lie within $\cO \left( \frac{1}{k^2} \right)$ of
\begin{equation}
r = \frac{1}{k} \,\Gamma(\theta) \;,
\label{eq:nearbyCurve}
\end{equation}
where
\begin{equation}
\begin{gathered}
{\rm if~} \cT = \cG^+[k,\chi] {\rm ~and~} a < 0, {\rm ~or~} \cT = \cG^-[k,\chi] {\rm ~and~} a > 0,
{\rm ~then~} \theta \in {\textstyle \left( \frac{3 \pi}{2}, 2 \pi \right)} \;, \\
{\rm and~if~} \cT = \cG^+[k,\chi] {\rm ~and~} a > 0, {\rm ~or~} \cT = \cG^-[k,\chi] {\rm ~and~} a < 0,
{\rm ~then~} \theta \in {\textstyle \left( \frac{\pi}{2}, \pi \right)} \;.
\end{gathered}
\label{eq:thetaSigns}
\end{equation}
\label{th:main1}
\end{theorem}

\begin{figure}[b!]
\begin{center}
\setlength{\unitlength}{1cm}
\begin{picture}(6,4.5)
\put(0,0){\includegraphics[height=4.5cm]{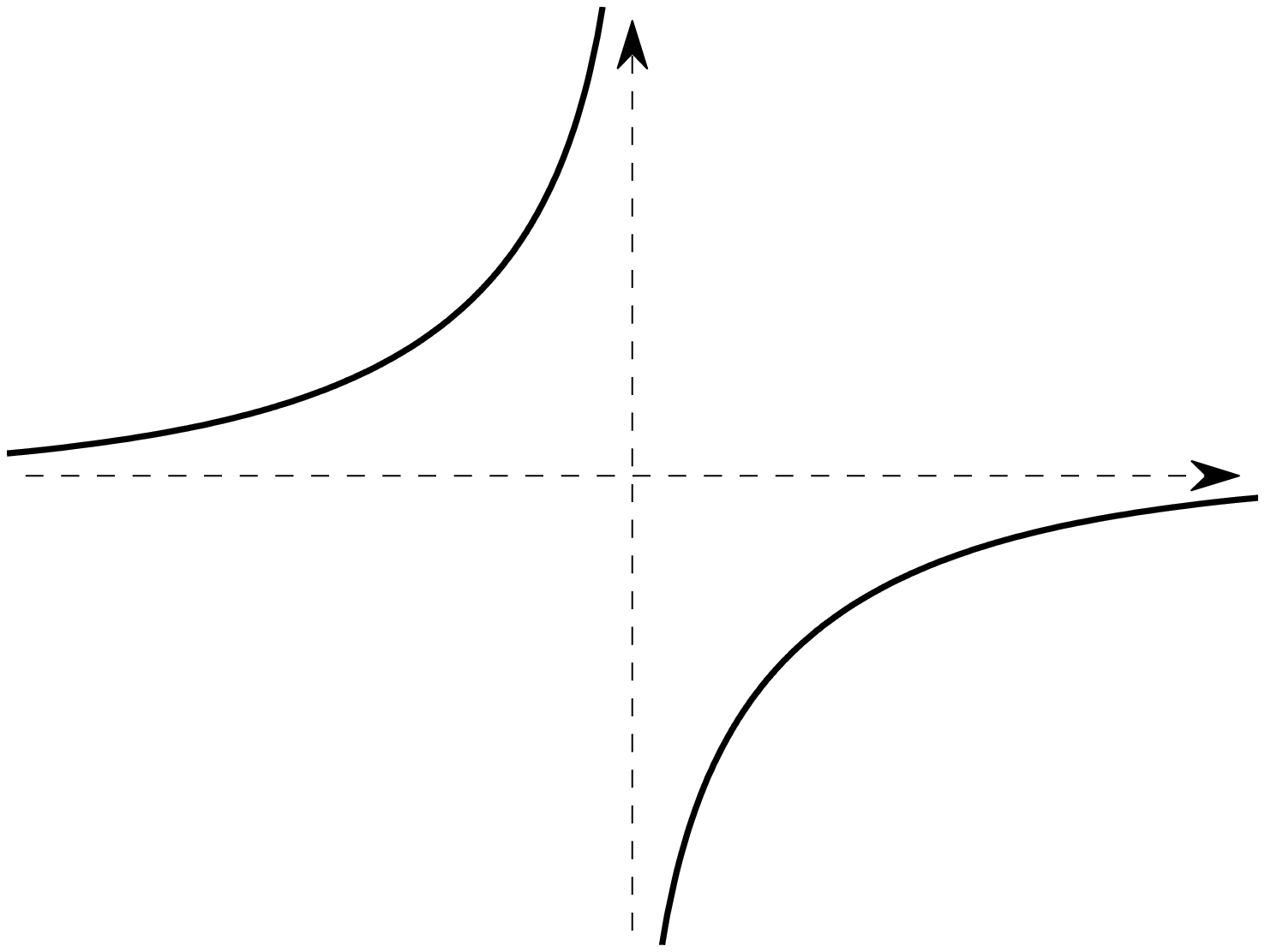}}
\put(5.4,2.4){\small $\eta$}
\put(3.1,4.1){\small $\nu$}
\put(4,2.33){\scriptsize $\theta = 0$}
\put(3.07,3.3){\scriptsize $\theta = \frac{\pi}{2}$}
\put(1.1,2){\scriptsize $\theta = \pi$}
\put(2.02,1){\scriptsize $\theta = \frac{3 \pi}{2}$}
\put(3.88,1.26){\small $r = \frac{1}{k} \,\Gamma(\theta)$}
\put(.48,3.2){\small $r = \frac{1}{k} \,\Gamma(\theta)$}
\end{picture}
\caption{
A sketch of (\ref{eq:nearbyCurve}).
\label{fig:Phi}
}
\end{center}
\end{figure}

Theorem \ref{th:main1} tells us that if there exists a $\cG^\pm_k$-mode-locking region,
it has a width of at most $\cO \left( \frac{1}{k^2} \right)$,
lies approximately on the curve (\ref{eq:nearbyCurve}) (sketched in Fig.~\ref{fig:Phi}),
and is located an $\cO \left( \frac{1}{k} \right)$ distance from the $\cS$-shrinking point.
By (\ref{eq:thetaSigns}), if $a < 0$
then the $\cG^+_k$-mode-locking regions lie in the fourth quadrant of the $(\eta,\nu)$-plane,
and the $\cG^-_k$-mode-locking regions lie in the second quadrant of the $(\eta,\nu)$-plane
(as in Fig.~\ref{fig:shrPointSchem}).
If $a > 0$, then the opposite is true.

Next we investigate shrinking points on the $\cG^\pm_k$-mode-locking regions.
We find that $\cG^\pm[k,\chi]$-shrinking points exist for arbitrarily large values of $k$
only for certain values of $\chi$,
and that these values of $\chi$ depend on the given $\cS$-shrinking point.
To determine the appropriate values of $\chi$,
we define scalar quantities $\kappa^{\pm}_\chi$ by
\begin{align}
\kappa^+_{\chi} &\defeq \begin{cases}
u_{\ell d}^{\sf T} M_{\cS^{\overline{0} (\ell d)}}^{-\chi-1} \big|_{(\eta,\nu)=(0,0)} v_{(\ell-1)d} \;, &
\chi \le -1 \\
u_0^{\sf T} M_{\cS^{\overline{\ell d}}}^{\chi} \big|_{(\eta,\nu)=(0,0)} v_{-d} \;, &
\chi \ge 0
\end{cases} \;, \label{eq:kappaPlus} \\
\kappa^-_{\chi} &\defeq \begin{cases}
u_{-d}^{\sf T} M_{\cS^{\overline{0}}}^{-\chi} \big|_{(\eta,\nu)=(0,0)} v_0 \;, &
\chi \le 0 \\
u_{(\ell-1)d}^{\sf T} M_{\cS^{\overline{\ell d} (\ell d)}}^{\chi - 1} \big|_{(\eta,\nu)=(0,0)} v_{\ell d} \;, &
\chi \ge 1
\end{cases} \;. \label{eq:kappaMinus}
\end{align}
To clarify these expressions, the matrix $M_{\cS^{\overline{0} (\ell d)}}^{-\chi-1}$, for example,
is the $(-\chi-1)^{\rm th}$ power of $M_{\cS^{\overline{0} (\ell d)}}$,
where $M_{\cS^{\overline{0} (\ell d)}}$ is given by (\ref{eq:MS})
using $\cS^{\overline{0} (\ell d)}$ -- the $(\ell d)^{\rm th}$
left shift permutation of $\cS^{\overline{0}}$.
In (\ref{eq:kappaPlus}) this matrix is evaluated at the $\cS$-shrinking point.


We also define\removableFootnote{
Note, the $\theta^\pm_{\chi}$ do not denote the exact values of the intersection points.
}
\begin{align}
\theta^+_{\chi} &\defeq \begin{cases}
\tan^{-1} \left( \frac{t_{(\ell+1)d}}
{t_{(\ell-1)d} \left| \kappa^+_{\chi} \right|} \right) \;, &
\chi \le -1 \\
\tan^{-1} \left( \frac{t_d}
{t_{-d} \left| \kappa^+_{\chi} \right|} \right) \;, &
\chi \ge 0
\end{cases} \;, \label{eq:thetaPlus} \\
\theta^-_{\chi} &\defeq \begin{cases}
\tan^{-1} \left( \frac{t_d \left| \kappa^-_{\chi} \right|}
{t_{-d}} \right) \;, &
\chi \le 0 \\
\tan^{-1} \left( \frac{t_{(\ell+1)d} \left| \kappa^-_{\chi} \right|}
{t_{(\ell-1)d}} \right) \;, &
\chi \ge 1
\end{cases} \;, \label{eq:thetaMinus}
\end{align}
assuming $\kappa^{\pm}_{\chi} \ne 0$,
where the ambiguity of each $\tan^{-1}(\cdot)$ is resolved by Table \ref{tb:thetaDomains}.

\begin{table}[t!]
\begin{center}
\begin{tabular}{|c|c|}
$a<0$ & $a>0$ \\
\cline{1-2} & \\[-3.5mm]
$\theta^+_{\chi} \in \left( \frac{3 \pi}{2}, 2 \pi \right)$ &
$\theta^+_{\chi} \in \left( \frac{\pi}{2}, \pi \right)$ \\[1mm]
$\theta^-_{\chi} \in \left( \frac{\pi}{2}, \pi \right)$ &
$\theta^-_{\chi} \in \left( \frac{3 \pi}{2}, 2 \pi \right)$ \\[1mm]
\cline{1-2}
\end{tabular}
\caption{
This table indicates the interval to which each $\theta^+_{\chi}$ (\ref{eq:thetaPlus})
and $\theta^-_{\chi}$ (\ref{eq:thetaMinus}) belong,
as determined by the sign of $a$.
\label{tb:thetaDomains}
}
\end{center}
\end{table}


\begin{theorem}
Suppose (\ref{eq:f}) with $K \ge 2$ has an $\cS$-shrinking point
satisfying $\sigma < 1$ and $\det(J) \ne 0$,
and write $\cS = \cF[\ell,m,n]$.
Then for all $\chi_{\rm max} \in \mathbb{Z}^+$,
there exists $k_{\rm min} \in \mathbb{Z}^+$,
and a neighbourhood $\cN$ of $(\eta,\nu) = (0,0)$,
such that for all $\cT = \cG^\pm[k,\chi] \in \Xi_{\chi_{\rm max}}$ with $k \ge k_{\rm min}$,
if $\kappa^{\pm}_{\chi} \ne 0$, then within $\cN$:
\begin{enumerate}
\item
there exists a unique point $\left( \eta_{\cT}, \nu_{\cT} \right)$
at which $\det \left( I - M_{\cT} \right) = \det \left( P_{\cT} \right) = 0$
if and only if $\kappa^{\pm}_{\chi} > 0$,
and $\left( \eta_{\cT}, \nu_{\cT} \right)$
lies within $\cO \left( \frac{1}{k^2} \right)$
of (\ref{eq:nearbyCurve}) with $\theta = \theta^{\pm}_{\chi}$;
\item
there exists a unique $C^K$ curve\removableFootnote{
With a bit more work I could probably approximate its slope.
}
on which (\ref{eq:f}) has a $\cT$-cycle with an associated stability multiplier of $-1$
if and only if $\kappa^{\pm}_{\chi} < 0$,
and this curve intersects $\det \left( P_{\cT} \right) = 0$
at a point within $\cO \left( \frac{1}{k^2} \right)$
of (\ref{eq:nearbyCurve}) with $\theta = \theta^{\pm}_{\chi}$.
\end{enumerate}
\label{th:main2}
\end{theorem}

By Theorem \ref{th:main2},
each $\left( \eta_{\cG^\pm[k,\chi]}, \nu_{\cG^\pm[k,\chi]} \right)$ is a potential $\cG^\pm[k,\chi]$-shrinking point
and exists if $\kappa^{\pm}_{\chi} > 0$.
If $\kappa^{\pm}_{\chi} < 0$, then no such points exist (for sufficiently large values of $k$).
For a fixed value of $\chi$, $\left( \eta_{\cG^+[k,\chi]}, \nu_{\cG^+[k,\chi]} \right)$
and $\left( \eta_{\cG^-[k,\chi]}, \nu_{\cG^-[k,\chi]} \right)$
are sequences of points that limit to the $\cS$-shrinking point as $k \to \infty$.
Each $\left( \eta_{\cG^\pm[k,\chi]}, \nu_{\cG^\pm[k,\chi]} \right)$ is a $\cG^\pm[k,\chi]$-shrinking point
if all the statements in Definition \ref{df:shrPoint} (applied to $\cG^\pm[k,\chi]$) are satisfied.
Naturally we would like identify a practical set of testable conditions that ensure
$\left( \eta_{\cG^\pm[k,\chi]}, \nu_{\cG^\pm[k,\chi]} \right)$ is a $\cG^\pm[k,\chi]$-shrinking point.
However, this is difficult (and not achieved in this paper)
because in order to show that the $\cG^\pm[k,\chi]^{\overline{0}}$-cycle is admissible
we have to determine the sign of $s^{\cG^\pm[k,\chi]^{\overline{0}}}_i$ for all $k n + n^{\pm}$ values of $i$.
Numerical investigations reveal that
the $\cG^\pm[k,\chi]^{\overline{0}}$-cycle is often admissible when $\kappa^{\pm}_{\chi} > 0$,
but this is not always the case.


Already the identity (\ref{eq:uvIdentity12early})
provides some restrictions on the combinations of signs possible for the $\kappa^{\pm}_{\chi}$.
The next result tells us that, for large $k$,
if $\left( \eta_{\cG^\pm[k,\chi]}, \nu_{\cG^\pm[k,\chi]} \right)$
is a $\cG^\pm[k,\chi]$-shrinking point (this requires $\kappa^{\pm}_{\chi} > 0$),
then either $\kappa^{\pm}_{\chi-1} > 0$ or $\kappa^{\pm}_{\chi+1} > 0$ (as determined by the sign of $a$).
Essentially this says that the existence of a $\cG^\pm[k,\chi]$-shrinking point
implies the existence of a neighbouring shrinking point in the $\cG^\pm_k$-mode-locking region,
if admissibility is satisfied.

\begin{theorem}
Suppose (\ref{eq:f}) with $K \ge 2$ has an $\cS$-shrinking point
satisfying $\sigma < 1$ and $\det(J) \ne 0$,
and write $\cS = \cF[\ell,m,n]$.
For any $\chi \in \mathbb{Z}$,
if $\kappa^{\pm}_{\chi} > 0$ and the point $\left( \eta_{\cG^+[k,\chi]}, \nu_{\cG^+[k,\chi]} \right)$
(as specified by Theorem \ref{th:main2})
is a $\cG^+[k,\chi]$-shrinking point
for arbitrarily large values of $k$, then,
\begin{equation}
\begin{gathered}
{\rm if~} a < 0, {\rm ~then~} \kappa^{\pm}_{\chi-1} > 0 \;, \\
{\rm and~if~} a > 0, {\rm ~then~} \kappa^{\pm}_{\chi+1} > 0 \;.
\end{gathered}
\label{eq:shrPointImplication}
\end{equation}
\label{th:main3}
\end{theorem}

Finally we provide some properties of nearby $\cG^\pm[k,\chi]$-shrinking points.
Here we use tildes to denote quantities of a $\cG^\pm[k,\chi]$-shrinking point.
For any $\cT \in \Xi_{\chi_{\rm max}}$ with $\kappa^{\pm}_{\chi} > 0$,
we let $\tilde{a} = \det \left( I - M_{\cT^{\overline{0}}} \right)$
and $\tilde{b} = \det \left( I - M_{\cT^{\overline{\tilde{\ell} d_k}}} \right)$,
evaluated at $\left( \eta_{\cT}, \nu_{\cT} \right)$.
We denote the $\cT^{\overline{0}}$-cycle by
$\left\{ \tilde{y}_i \right\}$ and let $\tilde{t}_i = e_1^{\sf T} \tilde{y}_i$.
We also let
$\tilde{\eta} = s^{\cT^{\overline{0}}}_0$,
$\tilde{\nu} = s^{\cT^{\overline{0}}}_{\tilde{\ell} d_k^\pm}$, and
\begin{equation}
\tilde{J} \defeq \left. \begin{bmatrix}
\frac{\partial \tilde{\eta}}{\partial \eta} &
\frac{\partial \tilde{\eta}}{\partial \nu} \\
\frac{\partial \tilde{\nu}}{\partial \eta} &
\frac{\partial \tilde{\nu}}{\partial \nu}
\end{bmatrix} \right|_{\left( \eta_{\cT}, \nu_{\cT} \right)} \;.
\label{eq:JtildeEarly}
\end{equation}

\begin{theorem}
Suppose (\ref{eq:f}) with $K \ge 2$ has an $\cS$-shrinking point
satisfying $\sigma < 1$ and $\det(J) \ne 0$,
and write $\cS = \cF[\ell,m,n]$.
Then for all $\chi_{\rm max} \in \mathbb{Z}^+$,
there exists $k_{\rm min} \in \mathbb{Z}^+$,
such that for all $\cT \in \Xi_{\chi_{\rm max}}$ with $k \ge k_{\rm min}$,
if $\kappa^{\pm}_{\chi} > 0$ and $\left( \eta_{\cT}, \nu_{\cT} \right)$ is a $\cT$-shrinking point, then
\begin{enumerate}
\item
${\rm sgn}(\tilde{a}) = {\rm sgn}(a)$ and $\det \left( \tilde{J} \right) > 0$,
\item
at $\left( \eta_{\cT}, \nu_{\cT} \right)$,
all eigenvalues of $M_{\cT}$, excluding the unit eigenvalue,
have modulus $\cO \left( \sigma^k \right)$.
\end{enumerate}
\label{th:main4}
\end{theorem}

By part (i) of Theorem \ref{th:main4}, $\cG^\pm[k,\chi]$-shrinking points
exhibit the unfolding depicted in Fig.~\ref{fig:shrPointUnfoldingBasic},
and have the same orientation as the $\cS$-shrinking point.
By part (ii) of Theorem \ref{th:main4}, it follows that there exists a stable $\cG^\pm[k,\chi]$-cycle
on one side of the $\cG^\pm[k,\chi]$-shrinking point.

\subsection{Examples illustrating the main results}
\label{sub:examples}

\subsubsection*{The three-dimensional border-collision normal form}

\begin{figure}[b!]
\begin{center}
\setlength{\unitlength}{1cm}
\begin{picture}(10,12.5)
\put(1,6.5){\includegraphics[height=6cm]{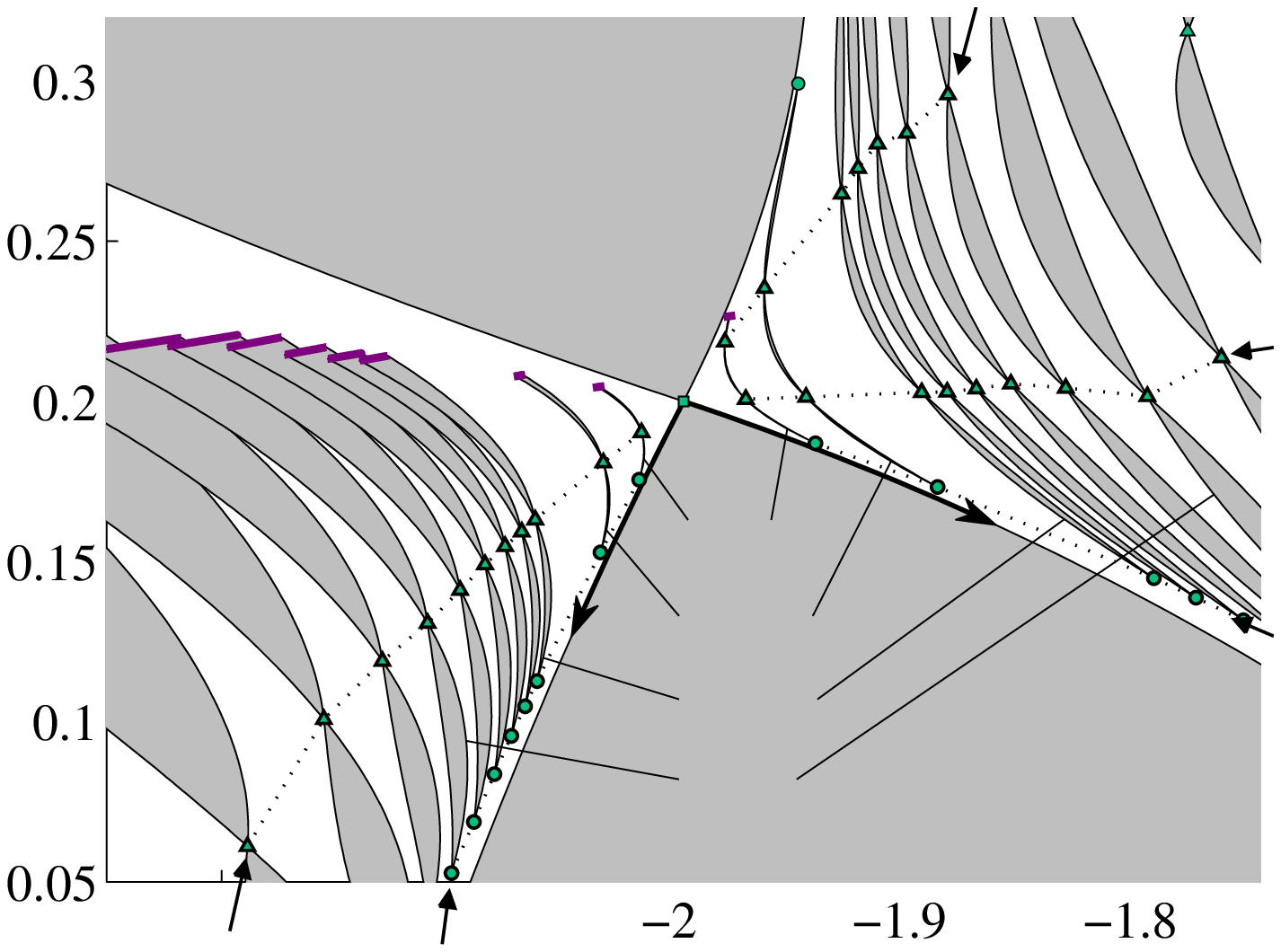}}
\put(1,0){\includegraphics[height=6cm]{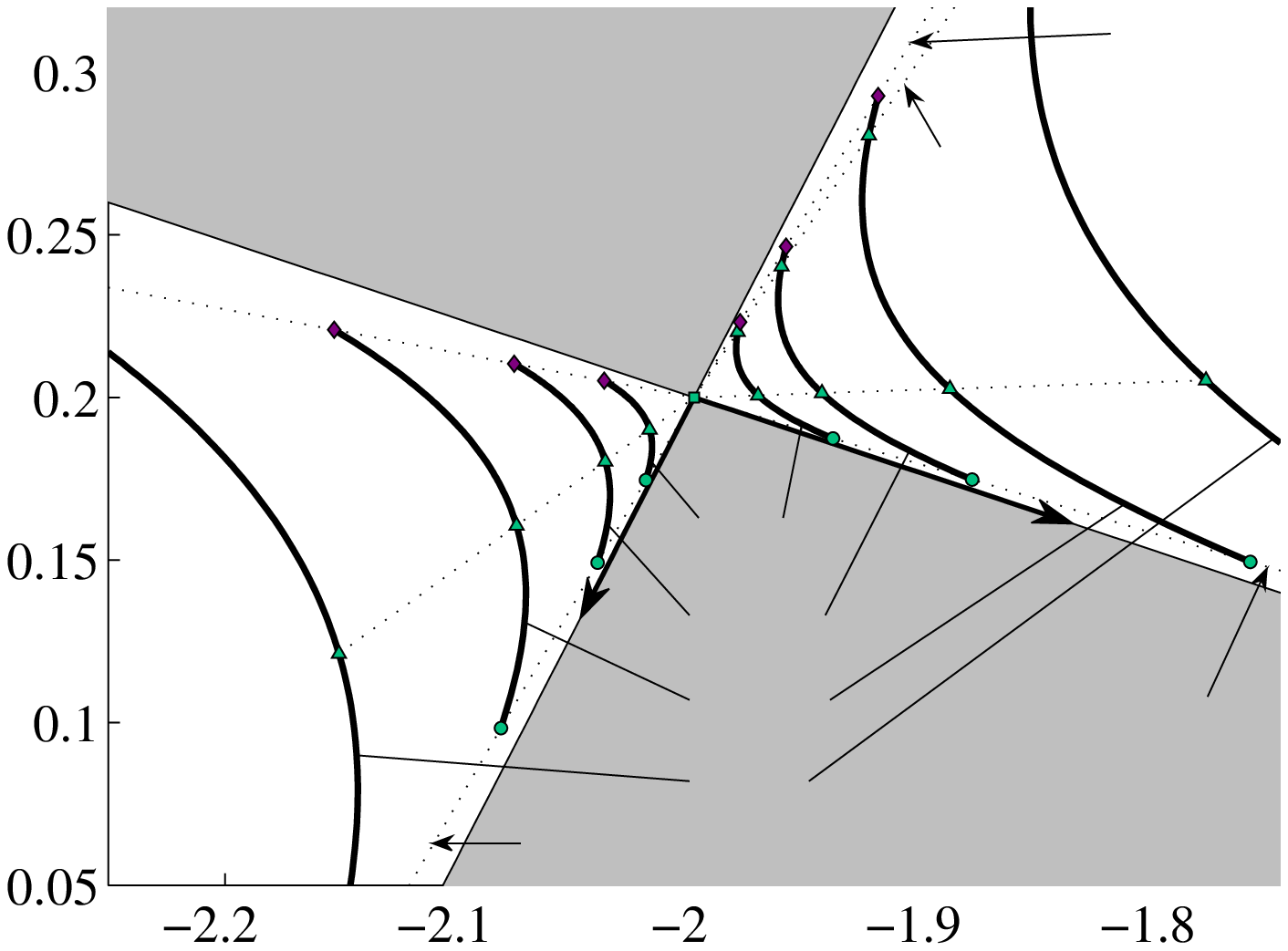}}
\put(.7,12.1){\large \sf \bfseries A}
\put(5.15,6.5){$\tau_R$}
\put(1,10.4){$\delta_L$}
\put(4.9,8.7){\tiny $\eta$}
\put(6.96,9.2){\tiny $\nu$}
\put(5.4,9.15){\tiny $k = 40$}
\put(5.4,8.67){\tiny $k = 20$}
\put(5.4,8.19){\tiny $k = 10$}
\put(5.4,7.71){\tiny $k = 5$}
\put(2.1,6.66){\tiny $\cG^+[k,-1]$}
\put(3.45,6.59){\tiny $\cG^+[k,-2]$}
\put(9.04,8.58){\tiny $\cG^-[k,-1]$}
\put(9.04,10.37){\tiny $\cG^-[k,0]$}
\put(6.8,12.6){\tiny $\cG^-[k,1]$}
\put(.7,5.6){\large \sf \bfseries B}
\put(5.15,0){$\tau_R$}
\put(1,3.9){$\delta_L$}
\put(4.82,2.23){\tiny $\eta$}
\put(7.35,2.67){\tiny $\nu$}
\put(5.4,2.65){\tiny $k = 40$}
\put(5.4,2.17){\tiny $k = 20$}
\put(5.4,1.69){\tiny $k = 10$}
\put(5.4,1.21){\tiny $k = 5$}
\put(2.48,3.84){\tiny $\theta^+_0$}
\put(3.6,2.15){\tiny $\theta^+_{-1}$}
\put(4.39,.86){\tiny $\theta^+_{-2}$}
\put(8.25,1.58){\tiny $\theta^-_{-1}$}
\put(7.62,3.46){\tiny $\theta^-_0$}
\put(6.86,4.93){\tiny $\theta^-_1$}
\put(7.92,5.71){\tiny $\theta^-_2$}
\end{picture}
\caption{
Panel A shows mode-locking regions of (\ref{eq:f}) with (\ref{eq:ALAREx20}),
$\mu > 0$, and with the remaining parameter values given by (\ref{eq:shrPointEx20})
(except $\tau_R$ and $\delta_L$ are variable).
This figure is centred about the $\cF[2,2,5]$-shrinking point shown in Fig.~\ref{fig:modeLockEx20}.
The other mode-locking regions correspond to stable $\cG^\pm[k,\chi]$-cycles.
Panel B illustrates the predictions of Theorems \ref{th:main1} and \ref{th:main2},
as discussed in the text.
\label{fig:tonguesEx20}
}
\end{center}
\end{figure}

The map (\ref{eq:f}) with (\ref{eq:ALAREx20}) has an $\cF[2,2,5]$-shrinking point at
\begin{equation}
\begin{gathered}
\tau_L = 0 \;, \qquad
\sigma_L = -1 \;, \qquad
\delta_L = 0.2 \;, \\
\tau_R = -2 \;, \qquad
\sigma_R = 0 \;, \qquad
\delta_R = 2 \;,
\end{gathered}
\label{eq:shrPointEx20}
\end{equation}
and $\mu > 0$.
This point is located in the middle of Fig.~\ref{fig:modeLockEx20}.
The corresponding mode-locking region 
has a stable $\cF[2,2,5]$-cycle in the lower section ($\delta_L < 0.2$),
and a stable $\cF[3,2,5]$-cycle in the upper section ($\delta_L > 0.2$).
The parameters $\tau_R$ and $\delta_L$ unfold the shrinking point generically
in the sense that $\det(J) \ne 0$, where $J$ is given by (\ref{eq:J})
using $\xi_1 = \tau_R$ and $\xi_2 = \delta_L$.
This implies that under the smooth transformation to $(\eta,\nu)$-coordinates (\ref{eq:etanu}),
the mode-locking region conforms to Fig.~\ref{fig:shrPointUnfoldingBasic} locally.

Fig.~\ref{fig:tonguesEx20}-A shows a magnified area of Fig.~\ref{fig:shrPointUnfoldingBasic}
and indicates the $\eta$ and $\nu$ axes.
Parts of $\cG^\pm_k$-mode-locking regions are also shown.
These were computed by numerically continuing the bifurcation boundaries.
The regions are shown for all $k \le 10$
in order to illustrate the proximity of the regions to one another
(indeed they overlap slightly) 
and also for $k = 20$ and $k = 40$
in order to illustrate the location and shape of the regions for relatively large values of $k$
without cluttering the figure.
Triangles indicate $\cT$-shrinking points for $\cT \in \Xi_1$.
Circles indicate $\cT$-shrinking points for $\cT \in \Xi_2 \setminus \Xi_1$.
For clarity the mode-locking regions are not shown beyond these shrinking points.
The additional (purple) curves are boundaries at which stable periodic solutions lose stability
via an associated stability multiplier attaining the value $-1$.

The results of \S\ref{sub:theorems} explain how the $\cG^\pm_k$-mode-locking regions
behave in the limit $k \to \infty$
based on various key quantities of the $\cF[2,2,5]$-shrinking point\removableFootnote{
For this shrinking point we have
\begin{equation}
a = -\frac{41}{5} \;, \qquad
b = \frac{123}{125} \;, \qquad
c = \frac{33}{25} \;.
\end{equation}
}.
To begin with, (\ref{eq:nearbyCurve}) provides an approximation to 
the location and shape of the mode-locking regions. 
Fig.~\ref{fig:tonguesEx20}-B illustrates this approximation for $k = 5,10,20,40$.
For simplicity we used a linear approximation to the coordinate change
$\left( \tau_R, \delta_L \right) \to (\eta,\nu)$ to produce Fig.~\ref{fig:tonguesEx20}-B\removableFootnote{
Consequently curves on which $\det \left( P_{\cS} \right) = 0$
and $\det \left( P_{\cS^{((\ell-1)d)}} \right) = 0$
appear as straight lines in Fig.~\ref{fig:tonguesEx20}-B.

In an earlier write-up I used Newton's root finding method
to perform the coordinate change more accurately.
}.

Theorem \ref{th:main2} provides approximations to the locations of nearby shrinking points
and stability loss bifurcation boundaries
based on the scalar quantities $\kappa^{\pm}_{\chi}$ (\ref{eq:kappaPlus})-(\ref{eq:kappaMinus})
and $\theta^{\pm}_{\chi}$ (\ref{eq:thetaPlus})-(\ref{eq:thetaMinus}).
For the shrinking point (\ref{eq:shrPointEx20}) we have
\begin{equation}
\begin{aligned}
\kappa^+_{-2} &= \frac{236}{33} \;, \hspace{30mm} &
\kappa^-_{-1} &= \frac{494}{55} \;, \\
\kappa^+_{-1} &= \frac{38}{55} \;, &
\kappa^-_0 &= \frac{43}{55} \;, \\
\kappa^+_0 &= -\frac{5}{11} \;, &
\kappa^-_1 &= \frac{10}{33} \;, \\
\kappa^+_1 &= \frac{26}{33} \;, &
\kappa^-_2 &= -\frac{32}{165} \;.
\end{aligned}
\label{eq:kappaValAll}
\end{equation}
We have used these values to generate Fig.~\ref{fig:tonguesEx20}-B.
Recall, Theorem \ref{th:main2} tells us that
$\kappa^{\pm}_{\chi} > 0$ implies there may exist $\cG^+[k,\chi]$-shrinking points, 
whereas $\kappa^{\pm}_{\chi} < 0$ implies stability loss curves (corresponding to a stability multiplier of $-1$).
In both cases these are located within
$\cO \left( \frac{1}{k^2} \right)$ of the curve (\ref{eq:nearbyCurve}) at $\theta = \theta^{\pm}_{\chi}$.
Only $\theta^+_1$ is not shown in Fig.~\ref{fig:tonguesEx20}-B
because at each $\left( \eta_{\cG^+[k,1]}, \nu_{\cG^+[k,1]} \right)$
the $\cG^+[k,1]^{\overline{0}}$-cycle is virtual.


In summary, Fig.~\ref{fig:tonguesEx20}-B illustrates
Theorems \ref{th:main1} and \ref{th:main2} with $\chi_{\rm max} = 2$.
We see that Fig.~\ref{fig:tonguesEx20}-B provides a good approximation to the mode-locking regions
shown in Fig.~\ref{fig:tonguesEx20}-A, including the shrinking points and stability loss boundaries,
and the accuracy of the approximation increases with $k$.
If we double the value of $k$, for example, then to leading order
the distance of the mode-locking region to the $\cF[2,2,5]$-shrinking point is halved.

\subsubsection*{Alternate cross-sections of parameter space}

\begin{figure}[b!]
\begin{center}
\setlength{\unitlength}{1cm}
\begin{picture}(10,12.5)
\put(1,6.5){\includegraphics[height=6cm]{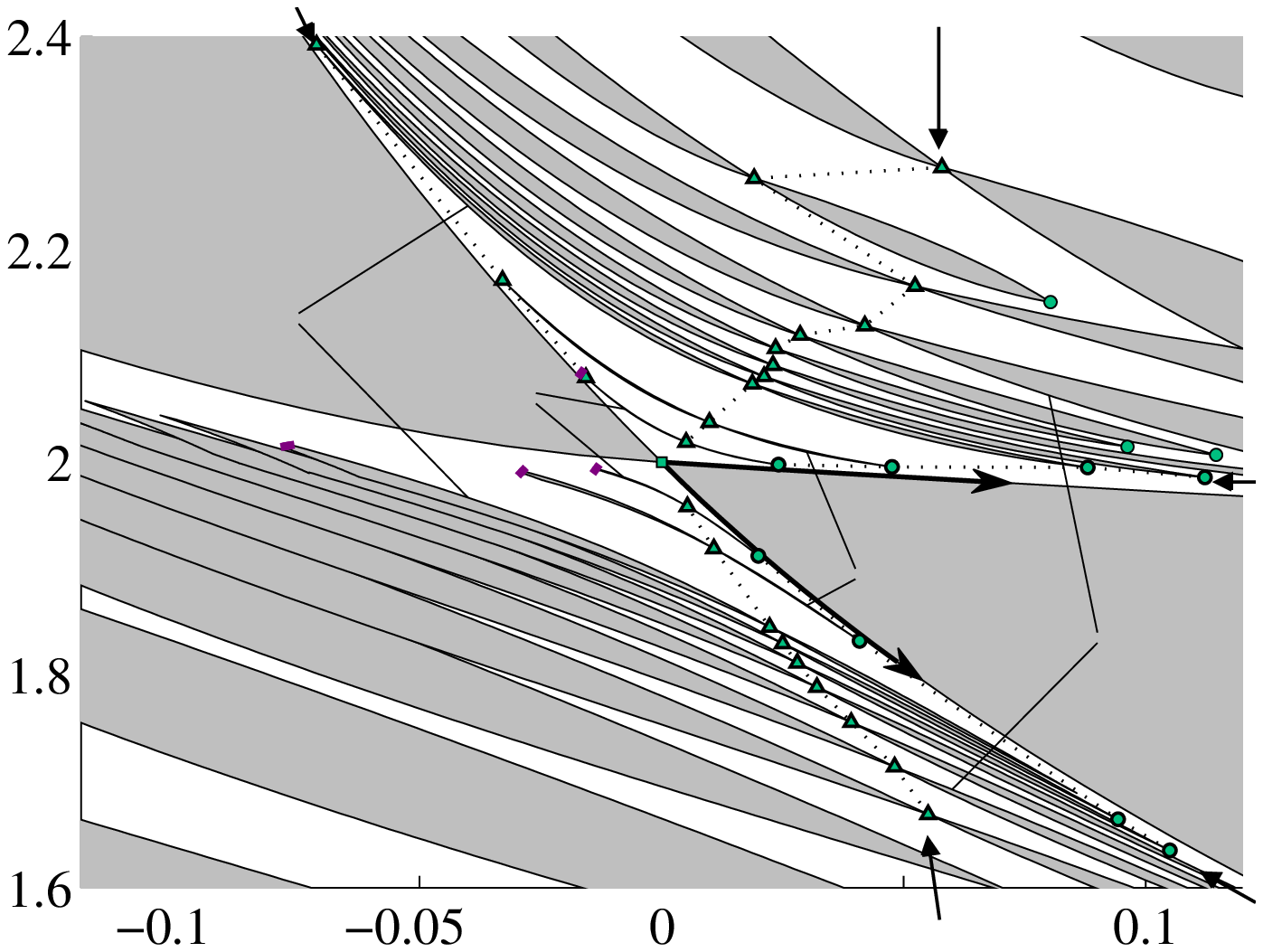}}
\put(1,0){\includegraphics[height=6cm]{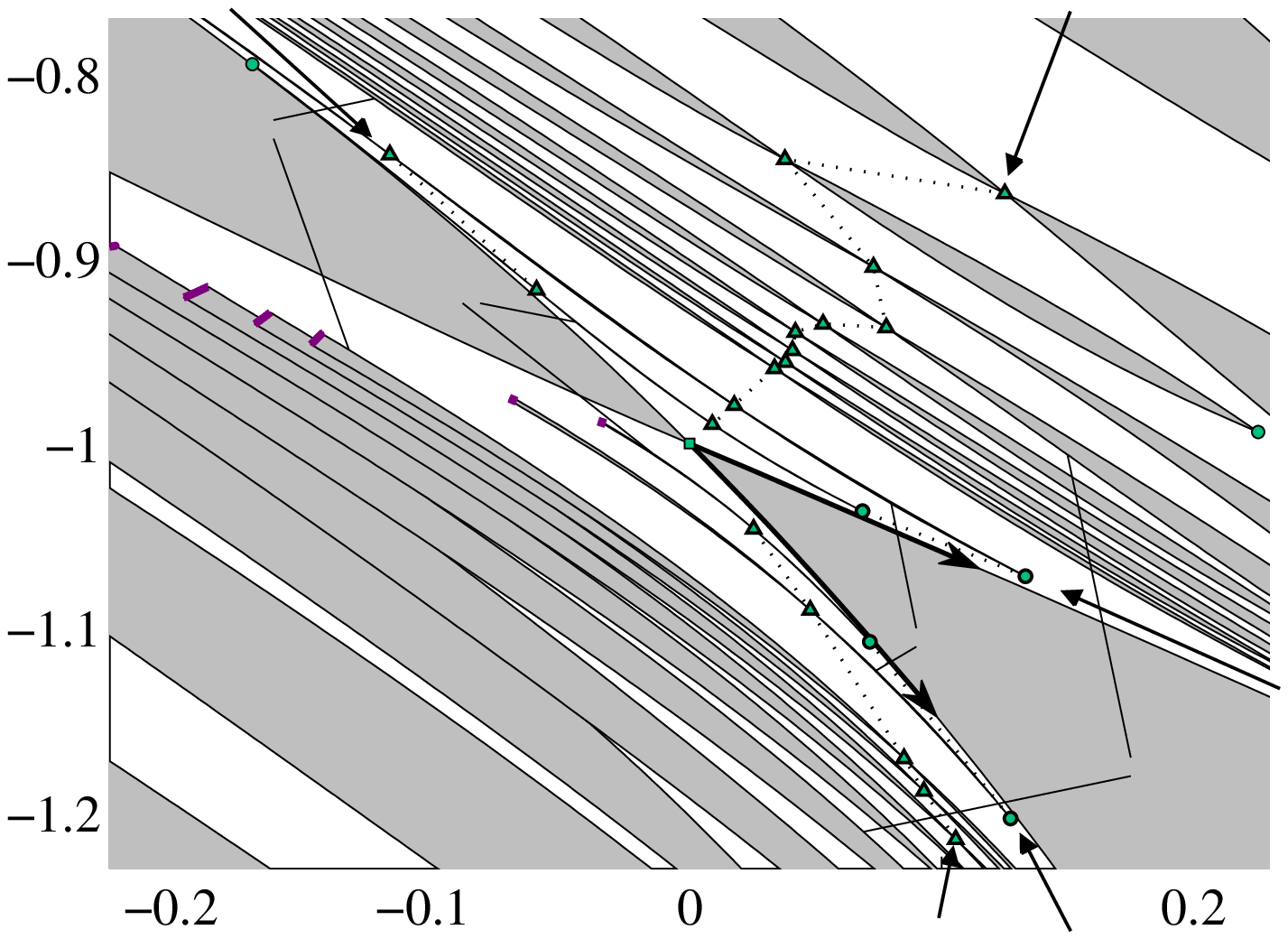}}
\put(.7,12.1){\large \sf \bfseries A}
\put(5.15,6.5){$\tau_L$}
\put(1,10.2){$\delta_R$}
\put(6.9,8.6){\tiny $\eta$}
\put(7.36,9.4){\tiny $\nu$}
\put(3.84,10.08){\tiny $k = 40$}
\put(6.6,9){\tiny $k = 20$}
\put(2.4,10.54){\tiny $k = 10$}
\put(8.08,8.62){\tiny $k = 5$}
\put(6.62,6.76){\tiny $\cG^+[k,-1]$}
\put(9.04,6.96){\tiny $\cG^+[k,-2]$}
\put(9.04,9.56){\tiny $\cG^-[k,-1]$}
\put(6.64,12.5){\tiny $\cG^-[k,0]$}
\put(2.66,12.58){\tiny $\cG^-[k,1]$}
\put(.7,5.6){\large \sf \bfseries B}
\put(5.15,0){$\sigma_R$}
\put(1,3.7){$\sigma_L$}
\put(6.87,1.7){\tiny $\eta$}
\put(6.98,2.35){\tiny $\nu$}
\put(3.47,4.16){\tiny $k = 40$}
\put(6.79,2.01){\tiny $k = 20$}
\put(2.09,5.09){\tiny $k = 10$}
\put(8.09,1.22){\tiny $k = 5$}
\put(6.14,.22){\tiny $\cG^+[k,-1]$}
\put(7.37,.07){\tiny $\cG^+[k,-2]$}
\put(9.01,1.68){\tiny $\cG^-[k,-1]$}
\put(7.23,5.98){\tiny $\cG^-[k,0]$}
\put(2.04,5.98){\tiny $\cG^-[k,1]$}
\end{picture}
\caption{
Panel A shows mode-locking regions of (\ref{eq:f}) with (\ref{eq:ALAREx20})
with $(\sigma_L,\delta_L,\tau_R,\sigma_R) = (-1,0.2,-2,0)$ and $\mu > 0$.
Panel B shows mode-locking regions of (\ref{eq:f}) with (\ref{eq:ALAREx20})
with $(\tau_L,\delta_L,\tau_R,\delta_R) = (0,0.2,-2,2)$ and $\mu > 0$.
\label{fig:tonguesEx2122}
}
\end{center}
\end{figure}

The signs of the $\kappa^{\pm}_{\chi}$ determine which $\cG^+[k,\chi]$ have shrinking points
and which $\cG^+[k,\chi]$ have stability loss boundaries.
The values of $\theta^{\pm}_{\chi}$ determine the relative location and spacing of these
features on the curves (\ref{eq:nearbyCurve}).
Yet each $\kappa^{\pm}_{\chi}$ and $\theta^{\pm}_{\chi}$ is a property of the $\cS$-shrinking point,
and so is independent to the parameters used to unfold the $\cS$-shrinking point.
Therefore with any non-degenerate two-dimensional slice of parameter space
through the $\cF[2,2,5]$-shrinking point (\ref{eq:shrPointEx20}),
the nearby mode-locking regions will appear as some smooth transformation of
Fig.~\ref{fig:tonguesEx20}-B\removableFootnote{
A shrinking point is a codimension-two phenomenon.
Suppose (\ref{eq:f}) has a shrinking point along some codimension-two curve $\Lambda$ in parameter space.
We could write
\begin{equation}
\Lambda = \left\{ \xi \in \mathbb{R}^M ~\big|~
e_1^{\sf T} x^{\cS^{\overline{0}}}_0(\xi)= 0 ,\,
e_1^{\sf T} x^{\cS^{\overline{0}}}_{ld}(\xi) = 0 \right\} \;.
\end{equation}
Suppose $F(\xi) : \mathbb{R}^M \to \mathbb{R}$ is a smooth function that is zero on $\Lambda$
(e.g.~$F(\xi) = \det(I-M_{\cS})$).
Let $\xi_0 \in \Lambda$.
Then we must have $F(\xi) = \alpha_1 \nabla
\left( e_1^{\sf T} x^{\cS^{\overline{0}}}_0 \right)^{\sf T} (\xi-\xi_0)
+ \alpha_2 \nabla \left( e_1^{\sf T} x^{\cS^{\overline{0}}}_{ld} \right)^{\sf T} (\xi-\xi_0) + \cO(2)$,
for some constants $\alpha_1, \alpha_2 \in \mathbb{R}$
because the derivative of $F(\xi)$ at $\xi = \xi_0$ is zero in other directions.
If we let $\eta = e_1^{\sf T} x^{\cS^{\overline{0}}}_0$
and $\nu = e_1^{\sf T} x^{\cS^{\overline{0}}}_{ld}$,
then $F(\xi) = \alpha_1 \eta + \alpha_2 \nu + \cO(2)$.
That is, once we convert to $(\eta,\nu)$ coordinates,
the linear approximation to $F(\xi)$ is independent of any non-degenerate two-dimensional
cross section of parameter space that we take.
}.

To illustrate this, Fig.~\ref{fig:tonguesEx2122} shows the nearby mode-locking regions
for two different cross-sections.
These indeed appear as distortions of Fig.~\ref{fig:tonguesEx20}-B.
For example, in both panels of Fig.~\ref{fig:tonguesEx2122},
there are $\cG^+[k,-1]$-shrinking points
and short curves along which a $\cG^+[k,0]$-cycle has a stability multiplier of $-1$,
and $k$ these features are relatively far apart.

\subsubsection*{Changes in the properties of an $\cS$-shrinking point}

A shrinking point is a codimension-two phenomenon,
therefore within three-dimensional regions of parameter space there exist curves of shrinking points.
As we follow a curve of $\cS$-shrinking points in a continuous manner,
the values of $\kappa^{\pm}_{\chi}$ and $\theta^{\pm}_{\chi}$ change continuously.
Therefore the structure of nearby mode-locking regions varies along the curve,
and there may be critical points at which the structure changes in a fundamental way
(e.g.~at a point where one of the $\kappa^{\pm}_{\chi}$ changes sign).
Here we show an example.

For all $\delta_R > 0.5865$, approximately,
the map (\ref{eq:f}) with (\ref{eq:ALAREx20}) has an $\cF[2,2,5]$-shrinking point at
\begin{equation}
\begin{gathered}
\tau_L = 0 \;, \qquad
\sigma_L = -1 \;, \qquad
\delta_L = \frac{\delta_R + 2}{\delta_R \left( \delta_R^2 + 2 \delta_R + 2 \right)} \;, \\
\tau_R = -\frac{\delta_R^2 + \delta_R + 2}{\delta_R + 2} \;, \qquad
\sigma_R = 0 \;, \qquad
\end{gathered}
\label{eq:shrPointEx2526}
\end{equation}
and $\mu > 0$, see \cite{Si15}.
The $\cF[2,2,5]$-shrinking point (\ref{eq:shrPointEx20}), considered above,
is given by (\ref{eq:shrPointEx2526}) with $\delta_R = 2$.

\begin{figure}[b!]
\begin{center}
\setlength{\unitlength}{1cm}
\begin{picture}(10,12.5)
\put(1,6.5){\includegraphics[height=6cm]{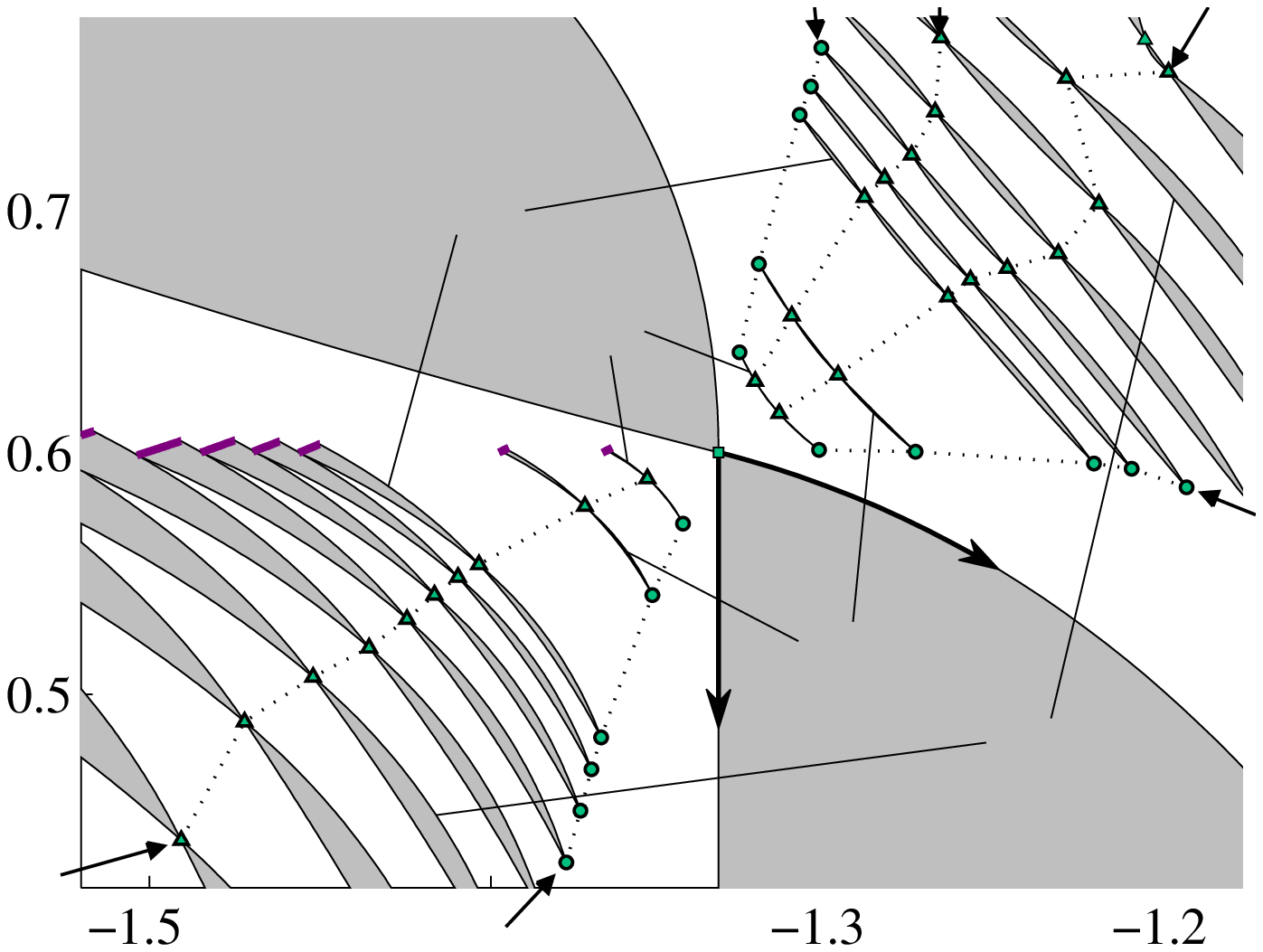}}
\put(1,0){\includegraphics[height=6cm]{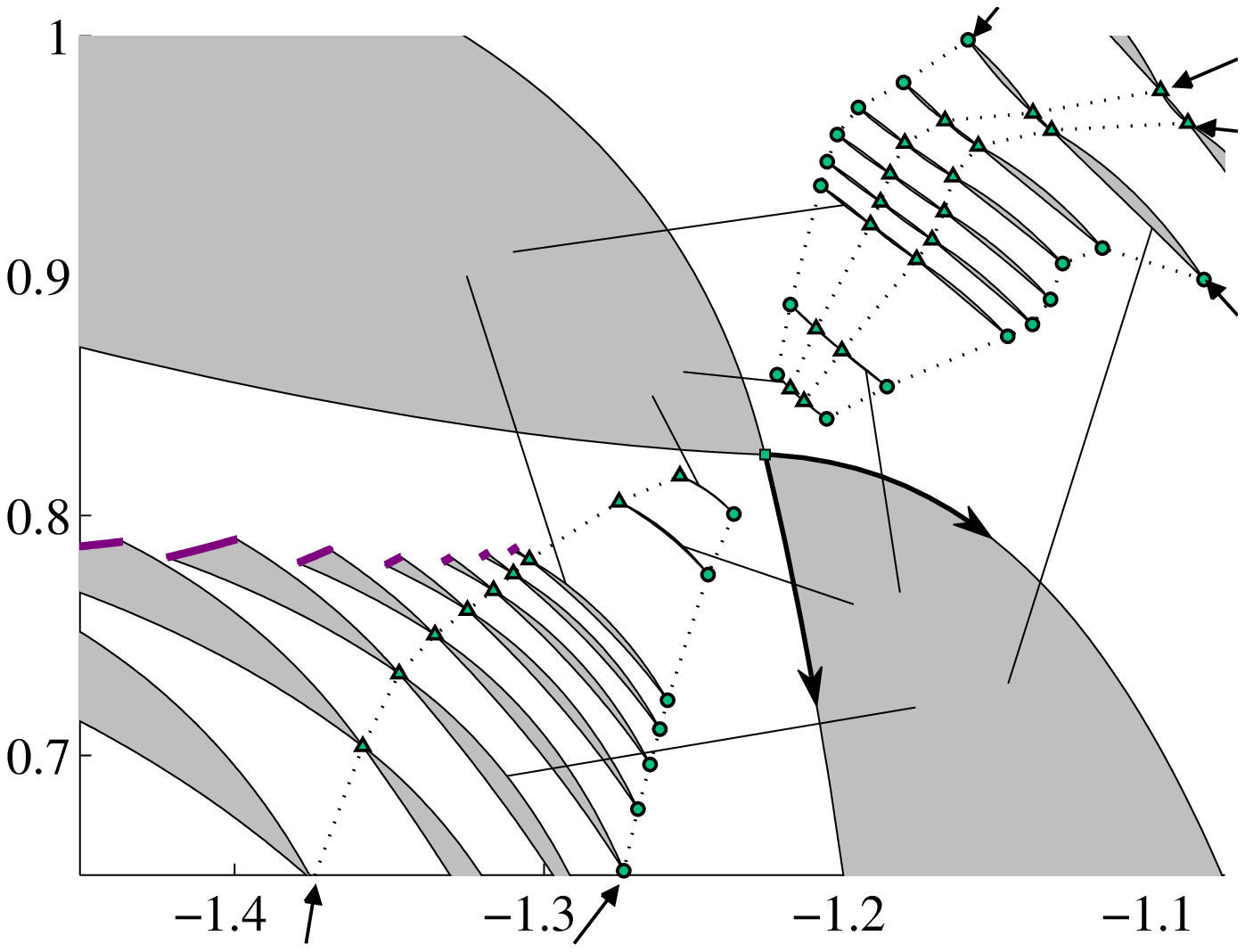}}
\put(.7,12.1){\large \sf \bfseries A}
\put(5.15,6.5){$\tau_R$}
\put(1,10){$\delta_L$}
\put(5.81,8.16){\tiny $\eta$}
\put(7.18,8.97){\tiny $\nu$}
\put(4.5,10.46){\tiny $k = 40$}
\put(6.25,8.6){\tiny $k = 20$}
\put(3.77,11.18){\tiny $k = 10$}
\put(7.4,8.01){\tiny $k = 5$}
\put(.6,7.2){\tiny $\cG^+[k,-1]$}
\put(3.84,6.72){\tiny $\cG^+[k,-2]$}
\put(9.04,9.32){\tiny $\cG^-[k,-1]$}
\put(8.5,12.58){\tiny $\cG^-[k,0]$}
\put(6.78,12.58){\tiny $\cG^-[k,1]$}
\put(5.66,12.58){\tiny $\cG^-[k,2]$}
\put(.7,5.6){\large \sf \bfseries B}
\put(5.15,0){$\tau_R$}
\put(1,3.5){$\delta_L$}
\put(6.45,1.76){\tiny $\eta$}
\put(7.11,2.66){\tiny $\nu$}
\put(4.79,3.68){\tiny $k = 40$}
\put(6.65,2.24){\tiny $k = 20$}
\put(3.73,4.41){\tiny $k = 10$}
\put(7.03,1.66){\tiny $k = 5$}
\put(2.55,.04){\tiny $\cG^+[k,-1]$}
\put(3.95,.04){\tiny $\cG^+[k,-2]$}
\put(9,3.97){\tiny $\cG^-[k,-1]$}
\put(9.04,5.17){\tiny $\cG^-[k,0]$}
\put(9.04,5.68){\tiny $\cG^-[k,1]$}
\put(7.28,6.07){\tiny $\cG^-[k,2]$}
\end{picture}
\caption{
Mode-locking regions of (\ref{eq:f}) with (\ref{eq:ALAREx20})
about the $\cF[2,2,5]$-shrinking point (\ref{eq:shrPointEx2526})
(i.e.~$\tau_L = 0$, $\sigma_L = -1$, $\sigma_R = 0$ and $\mu > 0$).
In panel A, $\delta_R = 1$; in panel B, $\delta_R = 0.8$.
These regions are shown for $\delta_R = 2$ in Fig.~\ref{fig:tonguesEx20}-A.
\label{fig:tonguesEx2526}
}
\end{center}
\end{figure}

As we decrease the value of $\delta_R$ from $\delta_R = 2$
(at which the $\kappa^{\pm}_{\chi}$ are given by (\ref{eq:kappaValAll})) to $\delta_R = 1$,
the sign of $\kappa^-_2$ changes from negative to positive (at $\delta_R \approx 1.4597$),
whereas the signs of the other seven values of $\kappa^{\pm}_{\chi}$ remains unchanged.
With $\delta_R = 1$ the nearby mode-locking regions, as shown in Fig.~\ref{fig:tonguesEx2526}-A,
have the same structure as those in Fig.~\ref{fig:tonguesEx20}-A except that
$\cG^-[k,2]$-shrinking points exist for arbitrarily large values of $k$.

Upon a further decrease in the value of $\delta_R$,
we have $\theta^+_0 = \theta^+_{-1}$\removableFootnote{
By using (\ref{eq:uvIdentity1}) and (\ref{eq:fourtIdentity}),
this occurs when $\kappa^+_0 = \frac{a}{2 c}$.
}
at $\delta_R \approx 0.8665$.
Fig.~\ref{fig:tonguesEx2526}-B shows nearby mode-locking regions using $\delta_R = 0.8$.
At this value of $\delta_R$
the mode-locking regions do not extend beyond the $\cG^+[k,-1]$ shrinking points,
for sufficiently large values of $k$.
This is because $\theta^+_0 > \theta^+_{-1}$ and so the bifurcation boundaries on which
an $\cG^+[k,0]$-cycle has a stability multiplier of $-1$ have become virtual.

\subsubsection*{Nonsmooth Neimark-Sacker-like bifurcations}

As another example, consider (\ref{eq:f}) with
\begin{equation}
A_L = \begin{bmatrix}
2 r_L \cos \left( 2 \pi \omega_L \right) & 1 \\
-r_L^2 & 0
\end{bmatrix} \;, \qquad
A_R = \begin{bmatrix}
\frac{2}{s_R} \cos \left( 2 \pi \omega_R \right) & 1 \\
-\frac{1}{s_R^2} & 0
\end{bmatrix} \;, \qquad
B = \begin{bmatrix}
1 \\ 0
\end{bmatrix} \;,
\label{eq:ALAREx30}
\end{equation}
where $r_L, s_R, \omega_L, \omega_R \in \mathbb{R}$ are parameters.
The map (\ref{eq:f}) with (\ref{eq:ALAREx30}) was studied in \cite{SiMe08b}
in order to investigate nonsmooth Neimark-Sacker-like bifurcations.

\begin{figure}[b!]
\begin{center}
\setlength{\unitlength}{1cm}
\begin{picture}(8,6)
\put(0,0){\includegraphics[height=6cm]{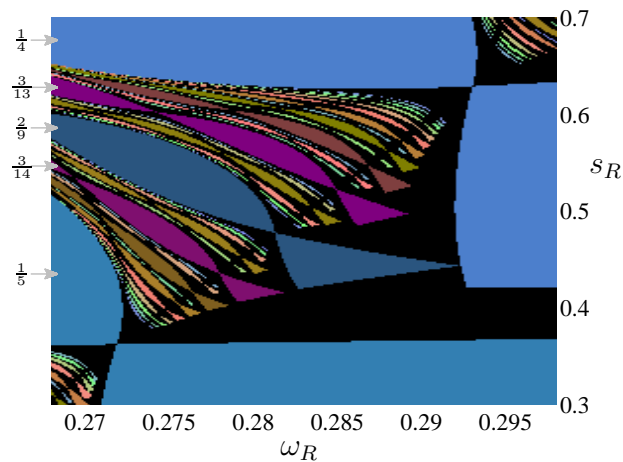}}
\put(3.6,0){$\omega_R$}
\put(7.7,3.75){$s_R$}
\put(.05,2.35){\tiny $\frac{1}{5}$}
\put(0,3.78){\tiny $\frac{3}{14}$}
\put(.05,4.295){\tiny $\frac{2}{9}$}
\put(0,4.835){\tiny $\frac{3}{13}$}
\put(.05,5.47){\tiny $\frac{1}{4}$}
\end{picture}
\caption{
Mode-locking regions of (\ref{eq:f}) with (\ref{eq:ALAREx30}),
$r_L = 0.3$, $\omega_L = 0.09$ and $\mu > 0$.
(as in Fig.~13 of \cite{SiMe08b}).
\label{fig:modeLockEx30}
}
\end{center}
\end{figure}

Fig.~\ref{fig:modeLockEx30} shows mode-locking regions of (\ref{eq:f}) with (\ref{eq:ALAREx30}).
This figure can be interpreted as showing the mode-locking
dynamics created in the border-collision bifurcation at $\mu = 0$,
where this bifurcation is akin to a Neimark-Sacker bifurcation in that an invariant circle is usually created
as the values of $\mu$ passes through $0$.
In Fig.~\ref{fig:modeLockEx30} there is a dominant curve of shrinking points
running diagonally from the bottom-left of the figure to the top-right.
Numerical computations of Lyapunov exponents reveal that this curve
appears to be a boundary for chaotic dynamics \cite{SiMe08b}.
The geometric mechanism responsible for this boundary of chaos is not fully understood.

\begin{figure}[b!]
\begin{center}
\setlength{\unitlength}{1cm}
\begin{picture}(10,12.5)
\put(1,6.5){\includegraphics[height=6cm]{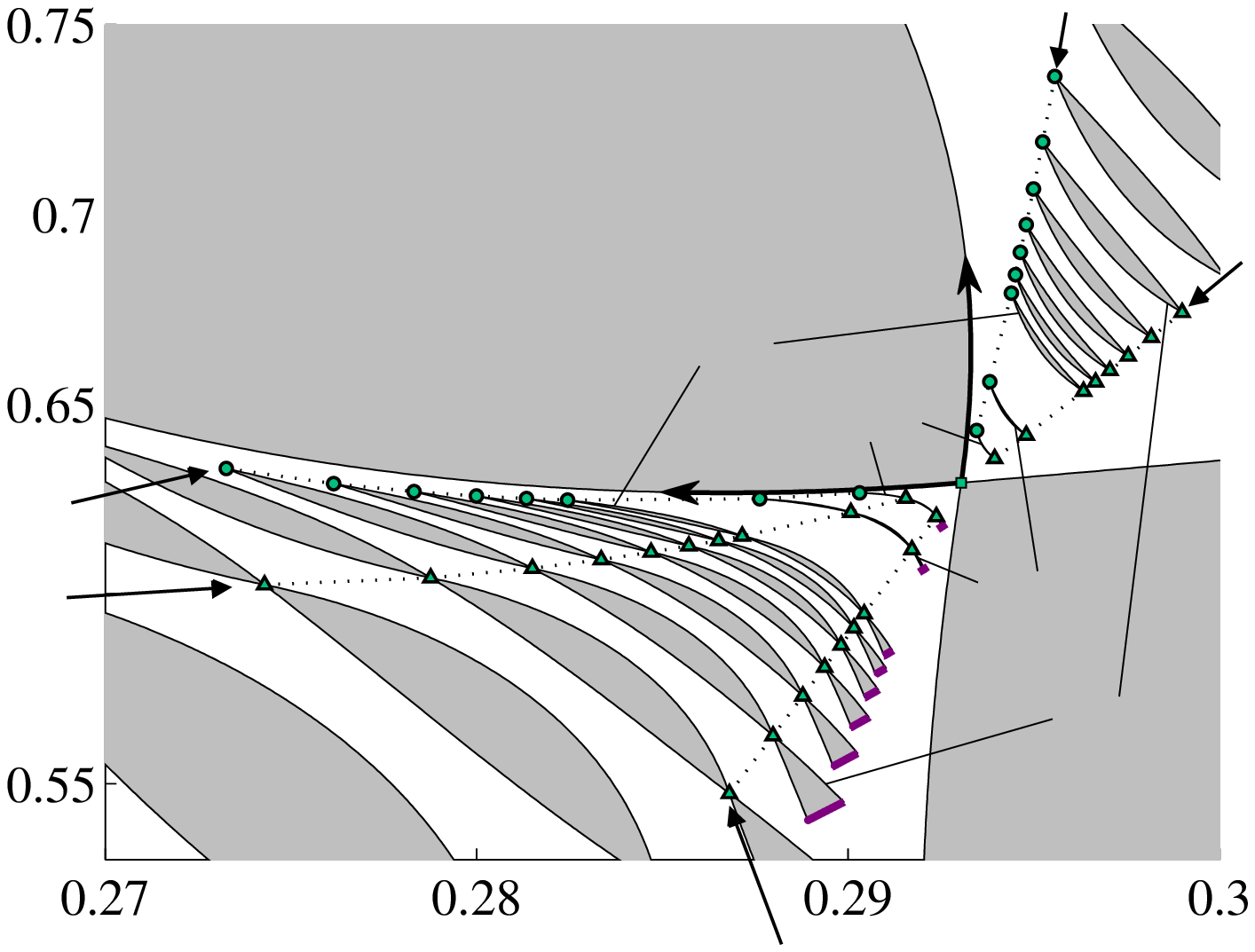}}
\put(1,0){\includegraphics[height=6cm]{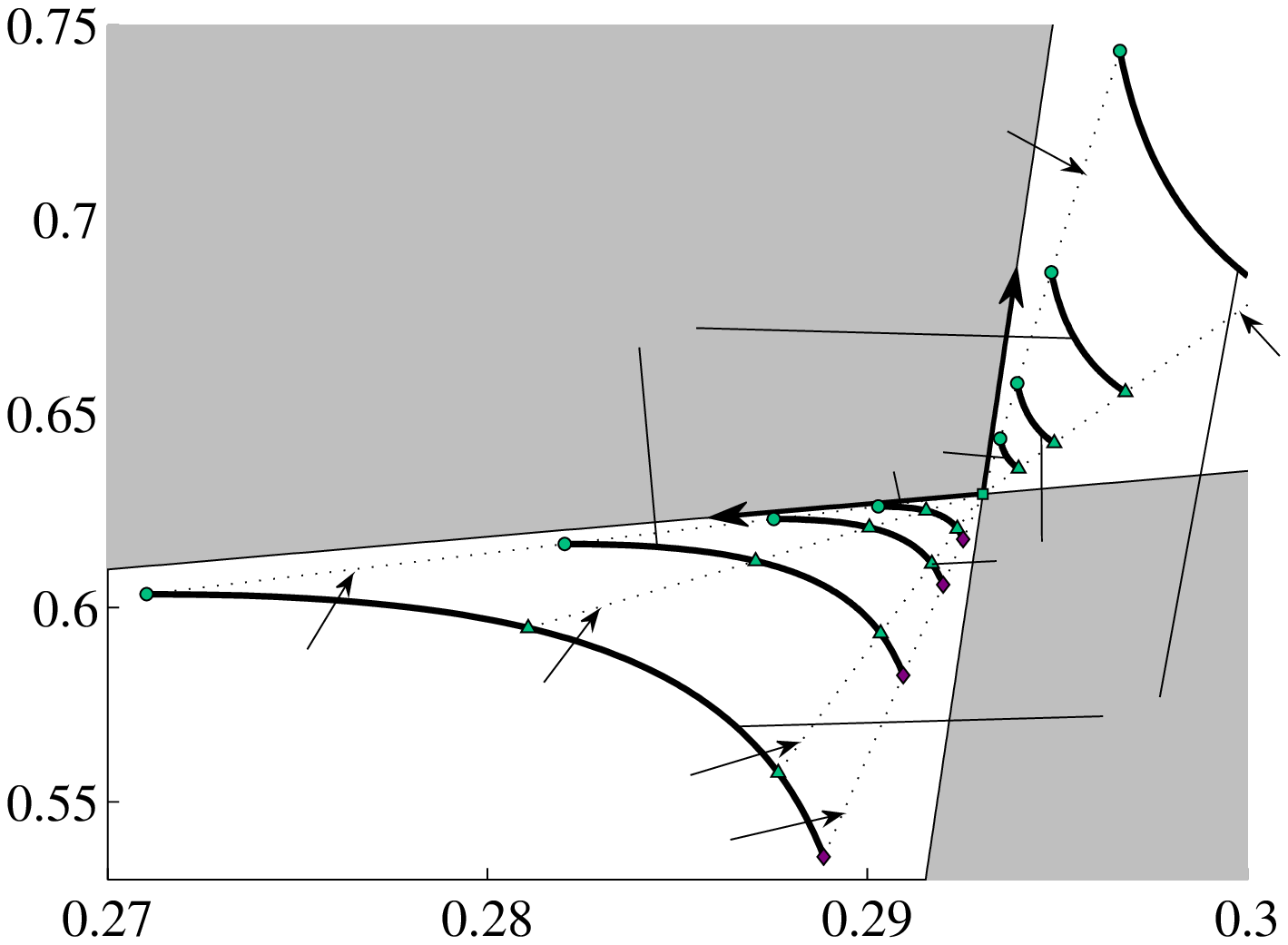}}
\put(.7,12.1){\large \sf \bfseries A}
\put(5.2,6.5){$\omega_R$}
\put(1,10.45){$s_R$}
\put(5.28,9.56){\tiny $\nu$}
\put(6.94,10.73){\tiny $\eta$}
\put(6.1,9.8){\tiny $k = 40$}
\put(7.28,8.78){\tiny $k = 20$}
\put(5.18,10.28){\tiny $k = 10$}
\put(7.74,8){\tiny $k = 5$}
\put(8.93,10.84){\tiny $\cG^+[k,-1]$}
\put(7.3,12.5){\tiny $\cG^+[k,-2]$}
\put(.5,9.32){\tiny	$\cG^-[k,-1]$}
\put(.67,8.74){\tiny $\cG^-[k,0]$}
\put(5.86,6.43){\tiny $\cG^-[k,1]$}
\put(.7,5.6){\large \sf \bfseries B}
\put(5.2,0){$\omega_R$}
\put(1,3.95){$s_R$}
\put(5.49,2.97){\tiny $\nu$}
\put(7.09,4.25){\tiny $\eta$}
\put(6.12,3.21){\tiny $k = 40$}
\put(7.27,2.55){\tiny $k = 20$}
\put(4.64,3.94){\tiny $k = 10$}
\put(7.91,1.58){\tiny $k = 5$}
\put(2.9,1.84){\tiny $\theta^-_{-1}$}
\put(4.34,1.64){\tiny $\theta^-_0$}
\put(5.06,1.2){\tiny $\theta^-_1$}
\put(5.29,.8){\tiny $\theta^-_2$}
\put(6.82,5.23){\tiny $\theta^+_{-2}$}
\put(8.92,3.68){\tiny $\theta^+_{-1}$}
\end{picture}
\caption{
Panel A shows mode-locking regions of (\ref{eq:f}) with (\ref{eq:ALAREx30}),
$r_L = 0.3$, $\omega_L = 0.09$ and $\mu > 0$,
obtained by numerically continuing bifurcation boundaries.
Panel B shows the leading-order approximation to the mode-locking regions,
as well as shrinking points and stability loss boundaries,
as given by Theorems \ref{th:main1} and \ref{th:main2}
(using a linear approximation to the coordinate change $(\omega_R,s_R) \leftrightarrow (\eta,\nu)$).
For a further explanation refer to the discussion surrounding Fig.~\ref{fig:tonguesEx20}.
\label{fig:tonguesEx30}
}
\end{center}
\end{figure}

The shrinking point in the large mode-locking region in the top-right of Fig.~\ref{fig:modeLockEx30}
corresponds to $\cS = \cF[2,1,4]$.
Fig.~\ref{fig:tonguesEx30}-A shows a magnification of Fig.~\ref{fig:modeLockEx30}
about this shrinking point, 
and Fig.~\ref{fig:tonguesEx30}-B illustrates the predictions of Theorems \ref{th:main1} and \ref{th:main2}.
As expected, the leading-order approximations to the nearby mode-locking regions,
and their shrinking points and stability loss boundaries, match well to their true locations
with the degree of accuracy increasing with the value of $k$.

\subsubsection*{Application to grazing-sliding bifurcations}

\begin{figure}[b!]
\begin{center}
\setlength{\unitlength}{1cm}
\begin{picture}(8,6)
\put(0,0){\includegraphics[height=6cm]{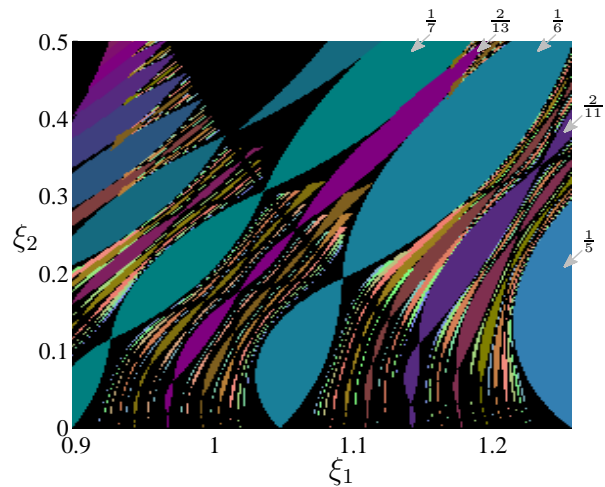}}
\put(4.2,0){$\xi_1$}
\put(0,3.1){$\xi_2$}
\put(5.46,6.04){\tiny $\frac{1}{7}$}
\put(6.32,6.04){\tiny $\frac{2}{13}$}
\put(7.14,6.04){\tiny $\frac{1}{6}$}
\put(7.55,4.9){\tiny $\frac{2}{11}$}
\put(7.555,3.11){\tiny $\frac{1}{5}$}
\end{picture}
\caption{
Mode-locking regions of (\ref{eq:f}) with (\ref{eq:ALAREx40}) and $\mu < 0$
(as in Fig.~4 of \cite{SzOs09}).
\label{fig:modeLockEx40}
}
\end{center}
\end{figure}

Dynamics near grazing-sliding bifurcations of $(N+1)$-dimensional
piecewise-smooth systems of ordinary differential equations
are captured by return maps of the form (\ref{eq:f})
for which either $\det \left( A_L \right) = 0$ or $\det \left( A_R \right) = 0$ \cite{DiKo02,JeHo11}.
In \cite{SzOs09}, the authors investigate mode-locking regions
near a grazing-sliding bifurcation in a model of a mechanical oscillator subject to dry friction.
The return map that they analyse can, through an affine change of variables,
be put in the form (\ref{eq:f}) with
\begin{equation}
A_L = \begin{bmatrix}
2 {\rm e}^{\xi_2} \cos \left( \xi_1 \right) & 1 \\
-{\rm e}^{2 \xi_2} & 0
\end{bmatrix} \;, \qquad
A_R = \begin{bmatrix}
{\rm e}^{\xi_2} \cos\left( \xi_1 \right) & 1 \\
0 & 0
\end{bmatrix} \;, \qquad
B = \begin{bmatrix}
1 \\ 0
\end{bmatrix} \;,
\label{eq:ALAREx40}
\end{equation}
where $\xi_1, \xi_2 \in \mathbb{R}$.
Mode-locking regions of (\ref{eq:f}) with (\ref{eq:ALAREx40}) are shown in Fig.~\ref{fig:modeLockEx40}.

Fig.~\ref{fig:tonguesEx40}-A shows a magnification of
Fig.~\ref{fig:modeLockEx40} about an $\cF[8,2,13]$-shrinking point.
It is interesting that since $\det \left( A_R \right) = 0$,
curves of shrinking points admit simple analytic expressions \cite{SzOs09}.
For example the dashed curve in Fig.~\ref{fig:tonguesEx40}-A is given by
\begin{equation}
{\rm e}^{5 \xi_2} \sin \left( 4 \xi_1 \right) -
{\rm e}^{4 \xi_2} \sin \left( 5 \xi_1 \right) +
\sin \left( \xi_1 \right) = 0 \;.
\label{eq:sideRemovalCurve}
\end{equation}

As with the previous two examples,
Fig.~\ref{fig:tonguesEx40}-B illustrates the predictions of Theorems \ref{th:main1} and \ref{th:main2}.
Here $a > 0$, so $\cG^+_k$-mode-locking regions lie on the same side as the $\nu$-axis
(the previous examples have $a < 0$).
Again, the leading order approximations of Fig.~\ref{fig:tonguesEx40}-B
match well to the numerical computations of Fig.~\ref{fig:tonguesEx40}-A.

\begin{figure}[t!]
\begin{center}
\setlength{\unitlength}{1cm}
\begin{picture}(10,12.5)
\put(1,6.5){\includegraphics[height=6cm]{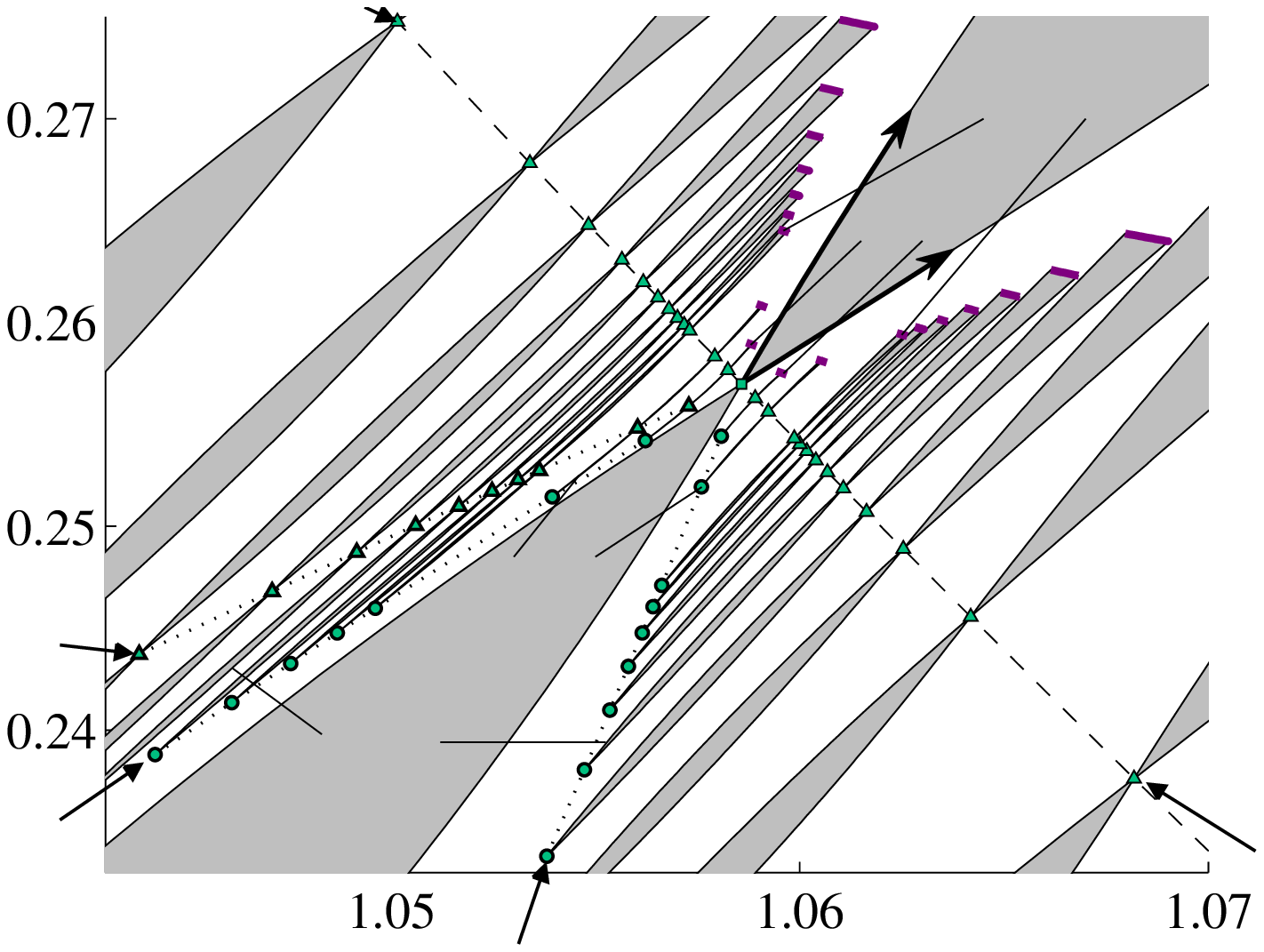}}
\put(1,0){\includegraphics[height=6cm]{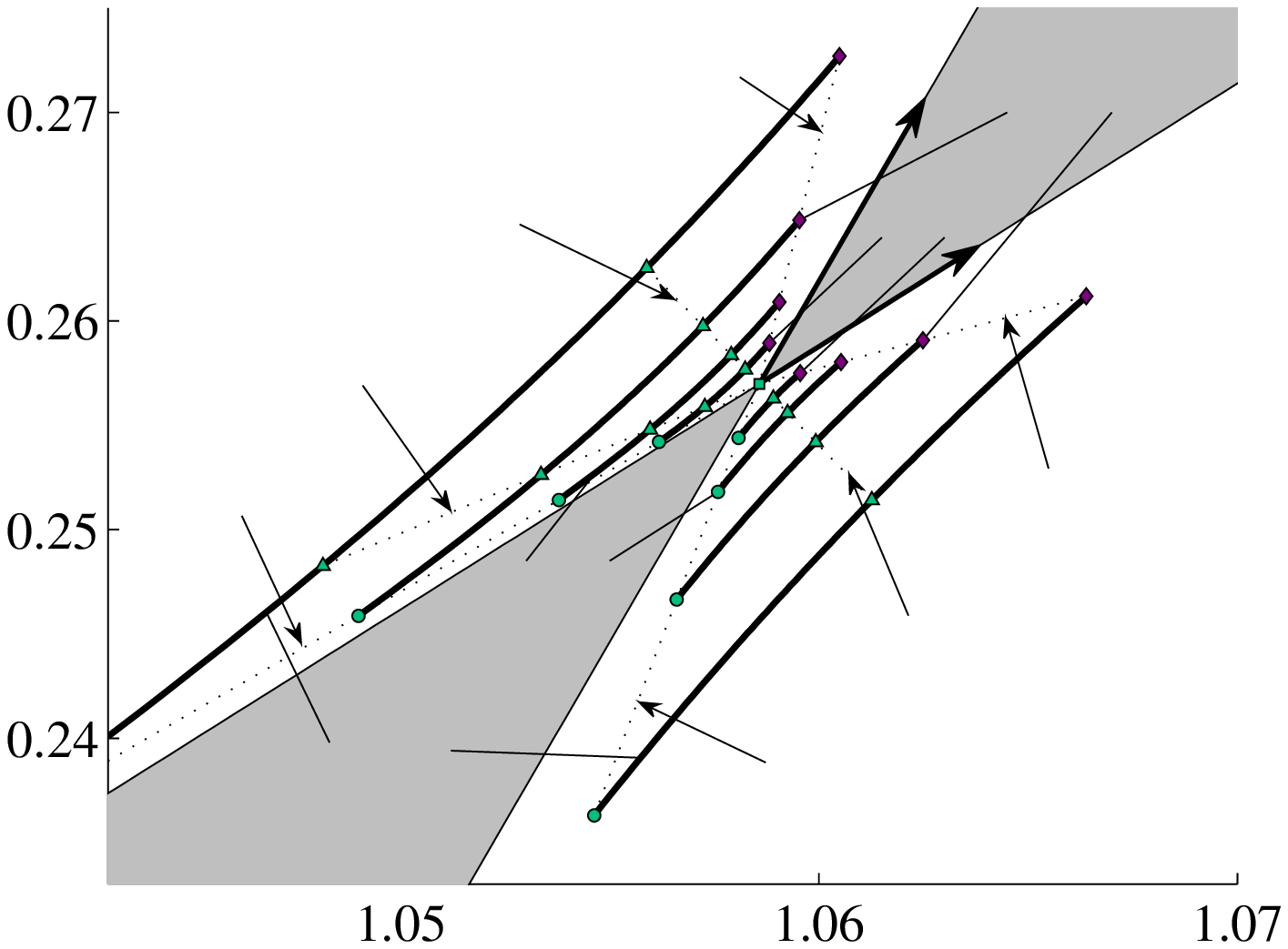}}
\put(.7,12.1){\large \sf \bfseries A}
\put(5.2,6.5){$\xi_1$}
\put(1,9.8){$\xi_2$}
\put(6.7,10.605){\tiny $\nu$}
\put(6.47,11.58){\tiny $\eta$}
\put(6.48,11.12){\tiny $k = 40$}
\put(4.18,8.91){\tiny $k = 20$}
\put(7.3,11.9){\tiny $k = 10$}
\put(3.25,7.9){\tiny $k = 5$}
\put(4,6.52){\tiny $\cG^+[k,1]$}
\put(9,7.2){\tiny $\cG^+[k,0]$}
\put(2.52,12.5){\tiny	$\cG^-[k,0]$}
\put(.63,8.52){\tiny $\cG^-[k,1]$}
\put(.63,7.44){\tiny $\cG^-[k,2]$}
\put(.7,5.6){\large \sf \bfseries B}
\put(5.2,0){$\xi_1$}
\put(1,3.3){$\xi_2$}
\put(6.78,4.15){\tiny $\nu$}
\put(6.43,5.1){\tiny $\eta$}
\put(6.48,4.62){\tiny $k = 40$}
\put(4.18,2.41){\tiny $k = 20$}
\put(7.3,5.4){\tiny $k = 10$}
\put(3.25,1.4){\tiny $k = 5$}
\put(2.48,3.02){\tiny $\theta^-_2$}
\put(3.14,3.72){\tiny $\theta^-_1$}
\put(4.06,4.69){\tiny $\theta^-_0$}
\put(5.21,5.57){\tiny $\theta^-_{-1}$}
\put(5.85,1.3){\tiny $\theta^+_1$}
\put(6.65,2.1){\tiny $\theta^+_0$}
\put(7.48,2.98){\tiny $\theta^+_{-1}$}
\end{picture}
\caption{
Panel A shows mode-locking regions of (\ref{eq:f}) with (\ref{eq:ALAREx40}) and $\mu < 0$
obtained by numerically continuing bifurcation boundaries.
Panel B shows the leading-order approximation to the mode-locking regions,
as well as shrinking points and stability loss boundaries,
as given by Theorems \ref{th:main1} and \ref{th:main2}
(using a linear approximation to the coordinate change $(\xi_1,\xi_2) \leftrightarrow (\eta,\nu)$).
For a further explanation refer to the discussion surrounding Fig.~\ref{fig:tonguesEx20}.
\label{fig:tonguesEx40}
}
\end{center}
\end{figure}

\section{Symbolic representation}
\label{sec:symbol}
\setcounter{equation}{0}

Here we develop symbolic notation on the alphabet $\{ L ,\, R \}$, following \cite{Si15,SiMe09,Si10}.
A {\em word} is a finite list of the symbols $L$ and $R$, e.g.~$\cS = LRR$.
We index the elements of a word from $i=0$ to $i=n-1$, where $n$ is the length of the word.
Thus for the previous example $n=3$ and $\cS_0 = L$, $\cS_1 = R$, $\cS_2 = R$.

Given two words $\cS$ and $\cT$,
the concatenation $\cS \cT$ is a word that has a length
equal to the sum of the lengths of $\cS$ and $\cT$.
The power $\cS^k$, where $k \in \mathbb{Z}^+$, is the concatenation of $k$ instances of $\cS$.
A word is said to be primitive if it cannot be written as a power with $k > 1$.

Given a word $\cS$, for any $j \in \mathbb{Z}$
we let $\cS^{(j)}$ denote the $j^{\rm th}$ left cyclic permutation of $\cS$.
That is, $\cS^{(j)}_i = \cS_{(i+j) {\rm \,mod\,} n}$ for all $i = 0,\ldots,n-1$.
Also we let $\cS^{\overline{j}}$ denote the word that differs from $\cS$
in only the symbol $\cS_{j {\rm \,mod\,} n}$.
For example, with $\cS = LRR$,
{\setlength{\jot}{0mm}
\begin{align*}
\cS^{(0)} = \cS &= LRR \;, & \cS^{\overline{0}} &= RRR \;, \\
\cS^{(1)} &= RRL \;, & \cS^{\overline{1}} &= LLR \;, \\
\cS^{(2)} &= RLR \;, & \cS^{\overline{2}} &= LRL \;.
\end{align*}
\vspace{-5mm}								
}

A {\em symbol sequence} $\cS$ is a bi-infinite list of the symbols $L$ and $R$.
$\cS$ is periodic with period $n \in \mathbb{Z}^+$,
if $\cS_{i + j n} = \cS_i$, for all $i = 0,\ldots,n-1$ and $j \in \mathbb{Z}$.
Given a periodic symbol sequence $\cS$ of minimal period $n$,
the word $\cS_0 \cdots \cS_{n-1}$ is primitive and completely determines $\cS$.
Conversely, given a primitive word $\cS$ of length $n$,
the infinite repetition of this word generates a symbol sequence $\cS$ of minimal period $n$.
For these reasons, in order to minimise the complexity of our notation,
we use periodic symbol sequences and primitive words interchangeably
and denote them with the same symbol, e.g.~$\cS$.


\subsection{Rotational symbol sequences}
\label{sub:rss}

In this paper we use periodic symbol sequences to describe periodic solutions to (\ref{eq:f}).
We restrict our attention to rotational symbol sequences $\cF[\ell,m,n]$, defined in Definition \ref{df:rss},
as these relate to shrinking points of (\ref{eq:f}), \cite{SiMe09,Si10,SiMe10}.
Rotational symbol sequences were originally studied independently of piecewise-linear maps
by Slater in \cite{Sl50,Sl67}.

\begin{figure}[b!]
\begin{center}
\setlength{\unitlength}{1cm}
\begin{picture}(8,6)
\put(0,0){\includegraphics[height=6cm]{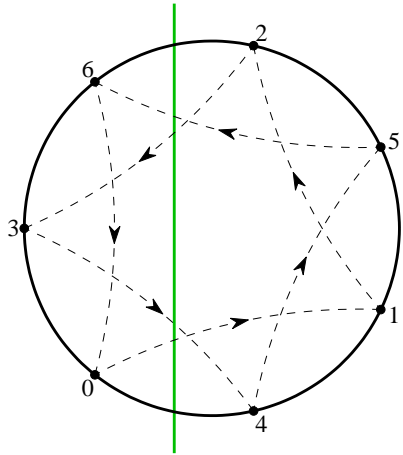}}
\end{picture}
\caption{
A pictorial interpretation of the rotational symbol sequence $\cF[3,5,7] = LRRLRRL$.
\label{fig:rotSymSeqSchem}
}
\end{center}
\end{figure}

Rotational symbol sequences can be interpreted as rigid rotation on a circle \cite{SiMe09,Si10}.
To see this, we treat the values $i m {\rm ~mod~} n$, that appear in the definition (\ref{eq:rss}),
as points on a circle.
This is shown in Fig.~\ref{fig:rotSymSeqSchem} for $\cF[3,5,7]$,
which was constructed by applying the following steps.	
\begin{enumerate}
\renewcommand{\labelenumi}{\arabic{enumi})}
\setlength{\itemsep}{0pt}
\item
Draw a circle intersecting a vertical line.
\item
Draw $n$ nodes on the circle, with $\ell$ of them to the left of the line.
\item
Find the first node that lies to the left of the lower intersection point of the circle and the line,
and call it $0$.
\item
For each $i = 1,\ldots,n-1$,
step $m$ nodes clockwise from node $i-1$, and call it node $i$.
\item
Then $\cF[\ell,m,n]_i$ is equal to $L$ if node $i$ is left of the line
and equal to $R$ otherwise.
\end{enumerate}
As we label the $n$ nodes in this fashion we revolve clockwise around the circle $m$ times.
For this reason we refer to $\frac{m}{n}$ as the rotation number of $\cF[\ell,m,n]$.

It is a simple combinatorial exercise to show that (\ref{eq:rss}) has the equivalent form,
\begin{equation}
\cF[\ell,m,n]_{j d {\rm \,mod\,} n} =
\begin{cases}
L \;, & j = 0,\ldots,\ell-1 \\
R \;, & j = \ell,\ldots,n-1
\end{cases} \;,
\label{eq:rssAlt}
\end{equation}
where $d$ is the multiplicative inverse of $m$ modulo $n$.
The next result is a simple consequence of (\ref{eq:rssAlt}) and we omit a proof.

\begin{proposition}
For any $\cF[\ell,m,n]$,
if $\ell \ne 1$ then
\begin{align}
\cF[\ell,m,n]^{\overline{(\ell-1)d}} &= \cF[\ell-1,m,n] \;, \label{eq:rssellm1d} \\
\cF[\ell,m,n]^{\overline{0}} &= \cF[\ell-1,m,n]^{(-d)} \;, \label{eq:rss0}
\end{align}
and if $\ell \ne n-1$ then
\begin{align}
\cF[\ell,m,n]^{\overline{\ell d}} &= \cF[\ell+1,m,n] \;, \label{eq:rsselld} \\
\cF[\ell,m,n]^{\overline{-d}} &= \cF[\ell+1,m,n]^{(d)} \;. \label{eq:rssmd}
\end{align}
\label{pr:rssIdentities}
\end{proposition}

By combining (\ref{eq:rss0}) and (\ref{eq:rsselld}) we obtain the following identity
which is central to our understanding of shrinking points.
\begin{corollary}
For any $\cF[\ell,m,n]$,
\begin{equation}
\cF[\ell,m,n]^{\overline{0} \, \overline{\ell d} (d)} = \cF[\ell,m,n] \;.
\label{eq:rssMainIdentity}
\end{equation}
\end{corollary}

\begin{example}
As an example, let us illustrate (\ref{eq:rssmd}) with $(\ell,m,n) = (3,5,7)$.
We have
\begin{center}
\begin{tabular}{c|ccccccc}
$i$ & 0 & 1 & 2 & 3 & 4 & 5 & 6 \\
\hline
$i m {\rm ~mod~} n$ & 0 & 5 & 3 & 1 & 6 & 4 & 2
\end{tabular}
\end{center}
thus by (\ref{eq:rss}), $\cF[3,5,7] = LRRLRRL$ (see also Fig.~\ref{fig:rotSymSeqSchem}).
Moreover, $\cF[4,5,7] = LRLLRRL$.
Here $d = 3$, and so to evaluate the left hand-side of (\ref{eq:rssmd}) we use $-d = 4$
(in modulo $7$ arithmetic).
Flipping the symbol $\cF[3,5,7]_4$ produces
$\cF[3,5,7]^{\overline{-d}} = LRRLLRL$.
Conversely, the right hand-side of (\ref{eq:rssmd})
is given by the third left shift permutation of $\cF[4,5,7]$, namely
$\cF[4,5,7]^{(3)} = LRRLLRL$, which we see is indeed the same as $\cF[3,5,7]^{\overline{-d}}$.
\label{ex:rss}
\end{example}

\subsection{Partitions of rotational symbol sequences}
\label{sub:partitions}

\begin{definition}
For any $\cS = \cF[\ell,m,n]$, let\removableFootnote{
A more general notation one could use is
$\cS^{(i,j)} = \cS_{i {\rm \,mod\,} n} \cdots \cS_{(j-1) {\rm \,{\rm mod}\,} n}$.
}
\begin{align}
\cX &= \cS_0 \cdots \cS_{(\ell d-1) {\rm \,mod\,} n} \;, &
\cY &= \cS_{\ell d {\rm \,mod\,} n} \cdots \cS_{n-1} \;, \label{eq:XY} \\
\hat{\cX} &= \cS_0 \cdots \cS_{(-d-1) {\rm \,mod\,} n} \;, &
\hat{\cY} &= \cS_{-d {\rm \,mod\,} n} \cdots \cS_{n-1} \;, \label{eq:XYhat} \\
\check{\cX} &= \cS_{\ell d {\rm \,mod\,} n} \cdots \cS_{((\ell-1)d-1) {\rm \,mod\,} n} \;, &
\check{\cY} &= \cS_{(\ell-1)d {\rm \,mod\,} n} \cdots \cS_{(\ell d-1) {\rm \,mod\,} n} \;. \label{eq:XYcheck}
\end{align}
\label{df:XYall}
\end{definition}
\vspace{-5mm}								

These six words are determined by the values of $\ell$, $m$ and $n$.
We do not explicitly write them as functions of $\ell$, $m$ and $n$
as it should always be clear which values of $\ell$, $m$ and $n$ are being used.

The word $\cX$, for instance, consists of the first $\ell d {\rm ~mod~} n$ symbols of $\cS$,
and $\cY$ consists of the remaining symbols of the word $\cS$.
We can therefore write $\cF[\ell,m,n] = \cX \cY$.
Further partitions of $\cF[\ell,m,n]$ are provided below in Proposition \ref{pr:partitions}.

First let us resolve a minor ambiguity in the definitions of $\check{\cX}$ and $\check{\cY}$.
To be precise, $\check{\cX}$ consists of the symbols of $\cS$ in cyclical order
starting from $\cS_{\ell d {\rm \,mod\,} n}$ and ending with $\cS_{((\ell-1)d-1) {\rm \,mod\,} n}$
(and similarly for $\check{\cY}$).
For example, with $\cF[3,5,7] = LRRLRRL$, we have
$d = 3$, thus $\ell d {\rm \,mod\,} n = 2$,
and $((\ell-1)d-1) {\rm \,mod\,} n = 5$.
Thus $\check{\cX} = RLRR$, and similarly $\check{\cY} = LLR$.

\begin{figure}[b!]
\begin{center}
\setlength{\unitlength}{1cm}
\begin{picture}(8,6)
\put(0,0){\includegraphics[height=6cm]{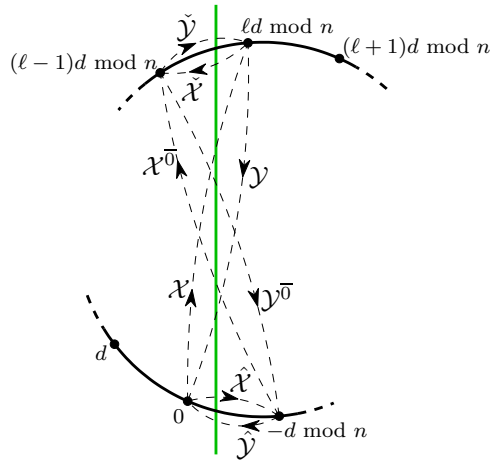}}
\put(2.72,2.1){\footnotesize $\cX$}
\put(3.8,3.6){\footnotesize $\cY$}
\put(2.38,3.74){\footnotesize $\cX^{\overline{0}}$}
\put(3.96,1.96){\footnotesize $\cY^{\overline{0}}$}
\put(3.54,.85){\footnotesize $\hat{\cX}$}
\put(3.63,0){\footnotesize $\hat{\cY}$}
\put(2.9,4.72){\footnotesize $\check{\cX}$}
\put(2.84,5.56){\footnotesize $\check{\cY}$}
\put(2.8,.44){\scriptsize $0$}
\put(4.04,.25){\scriptsize $-d {\rm ~mod~} n$}
\put(1.76,1.28){\scriptsize $d$}
\put(.62,5.14){\scriptsize $(\ell-1)d {\rm ~mod~} n$}
\put(3.7,5.62){\scriptsize $\ell d {\rm ~mod~} n$}
\put(5.04,5.36){\scriptsize $(\ell+1)d {\rm ~mod~} n$}
\end{picture}
\caption{
A pictorial interpretation of the words (\ref{eq:XY})-(\ref{eq:XYcheck}).
\label{fig:partitionsSchem}
}
\end{center}
\end{figure}

Fig.~\ref{fig:partitionsSchem} illustrates the words (\ref{eq:XY})-(\ref{eq:XYcheck}) pictorially.
For instance, the word $\cX$ ``follows'' $\cS$ from node $0$ to node $\ell d {\rm ~mod~}n$,
and the word $\check{\cX}$ ``follows'' $\cS$ from node $\ell d {\rm ~mod~}n$
to node $(\ell-1) d {\rm ~mod~}n$.

The words (\ref{eq:XY})-(\ref{eq:XYcheck}) can be used to partition $\cF[\ell,m,n]$ in different ways
and are useful to us for constructing the symbol sequences
of periodic solutions in nearby mode-locking regions of (\ref{eq:f}).
We omit a proof of Proposition \ref{pr:partitions} as it follows simply from (\ref{eq:XY})-(\ref{eq:XYcheck})
and Proposition \ref{pr:rssIdentities}\removableFootnote{
These are trivial, except
$\cF[\ell,m,n]^{(-d)} = \cX^{\overline{0}} \cY^{\overline{0}}$
which follows from (\ref{eq:rss0}) and (\ref{eq:rsselld}),
and $\cF[\ell,m,n]^{((\ell-1)d)} = \cY^{\overline{0}} \cX^{\overline{0}}$,
which follows from taking the $\ell d$ permutation of both sides of the first equation.
}.

\begin{proposition}
For any $\cF[\ell,m,n]$,
\begin{align}
\cF[\ell,m,n] &= \cX \cY \;, &
\cF[\ell,m,n] &= \hat{\cX} \hat{\cY} \;, \label{eq:partitions1} \\
\cF[\ell,m,n]^{(\ell d)} &= \cY \cX \;, &
\cF[\ell,m,n]^{(\ell d)} &= \check{\cX} \check{\cY} \;, \label{eq:partitions2} \\
\cF[\ell,m,n]^{(-d)} &= \cX^{\overline{0}} \cY^{\overline{0}} \;, &
\cF[\ell,m,n]^{(-d)} &= \hat{\cY} \hat{\cX} \;, \label{eq:partitions3} \\
\cF[\ell,m,n]^{((\ell-1)d)} &= \cY^{\overline{0}} \cX^{\overline{0}} \;, &
\cF[\ell,m,n]^{((\ell-1)d)} &= \check{\cY} \check{\cX} \;. \label{eq:partitions4}
\end{align}
\label{pr:partitions}
\end{proposition}
\vspace{-5mm}								

The next result equates various concatenations of the words (\ref{eq:XY})-(\ref{eq:XYcheck}).
These can be understood intuitively by following the arrows in Fig.~\ref{fig:partitionsSchem}.
For instance, roughly speaking, both $\cX \check{\cX}$ and $\hat{\cX} \cX^{\overline{0}}$
take us from node $0$ to node $(\ell-1) d {\rm ~mod~}n$ via an intermediary node. 

\begin{proposition}
For any $\cF[\ell,m,n]$,
\begin{align}
\cX \check{\cX} &= \hat{\cX} \cX^{\overline{0}} \;, \label{eq:XXcheck} \\
\cY \hat{\cX} &= \check{\cX} \cY^{\overline{0}} \;, \label{eq:YXhat} \\
\hat{\cY} \cX &= \cX^{\overline{0}} \check{\cY} \;, \label{eq:YhatX} \\
\check{\cY} \cY &= \cY^{\overline{0}} \hat{\cY} \;. \label{eq:YcheckY}
\end{align}
\label{pr:XYformulas}
\end{proposition}
\vspace{-5mm}								

\begin{proof}
Here we derive (\ref{eq:XXcheck}).
The remaining identities can be derived similarly.

Let $\cS = \cF[\ell,m,n]$.
By (\ref{eq:XY}) and (\ref{eq:XYcheck}),
\begin{equation}
\cX \check{\cX} = \cS_0 \cdots \cS_{(\ell d-1) {\rm \,mod\,} n}
\cS_{\ell d {\rm \,mod\,} n} \cdots \cS_{((\ell-1)d-1) {\rm \,mod\,} n} \;.
\end{equation}
Also $\cS^{(-d)} = \cX^{\overline{0}} \cY^{\overline{0}}$, (\ref{eq:partitions3}).
Therefore $\cX^{\overline{0}}$ consists of the first $\ell d$ symbols of $\cS^{(-d)}$,
i.e.~$\cX^{\overline{0}} = \cS_{-d {\rm \,mod\,} n} \cdots \cS_{((\ell-1)d-1) {\rm \,mod\,} n}$.
Therefore
\begin{equation}
\hat{\cX} \cX^{\overline{0}} = \cS_0 \cdots \cS_{(-d-1) {\rm \,mod\,} n}
\cS_{-d {\rm \,mod\,} n} \cdots \cS_{(\ell-1)d-1) {\rm \,mod\,} n} \;,
\end{equation}
where we have substituted the definition of $\hat{\cX}$ (\ref{eq:XYhat}).
Therefore $\cX \check{\cX}$ and $\hat{\cX} \cX^{\overline{0}}$
both consist of the symbols of $\cS$ in cyclical order
starting from $\cS_0$ and ending with $\cS_{(\ell-1)d-1) {\rm \,mod\,} n}$.
The words $\cX$ and $\cX^{\overline{0}}$ both have length $\ell d {\rm ~mod~} n$,
and the words $\check{\cX}$ and $\hat{\cX}$ both have length $-d {\rm ~mod~} n$.
Thus $\cX \check{\cX}$ and $\hat{\cX} \cX^{\overline{0}}$
consist of the same number of symbols, which verifies (\ref{eq:XXcheck}).
\end{proof}

\subsection{Sequences of symbol sequences}
\label{sub:nearbySequences}


We begin by reviewing Farey addition and the Farey tree\removableFootnote{
My understanding is that the Farey tree has $0$ and $1$
as the roots at the top and involves all rational numbers between $0$ and $1$,
whereas the Stern-Brocot tree has just $1$ at the top
and involves all positive rational numbers.
},
and then apply the results to the symbol sequences $\cG^\pm[k,\chi]$ (\ref{eq:Gplusminusearly}).

The Farey tree is a graph with the rational numbers in $[0,1]$ as its vertices \cite{LaTr95,BrRi02}.
The Farey tree can be constructed by starting with the numbers $\frac{0}{1}$ and $\frac{1}{1}$
and supposing that there is an edge between them.
Then all rational numbers between $0$ and $1$ are incorporated into the tree
by repeatedly applying the following rule.
Given any two fractions $\frac{m^-}{n^-}$ and $\frac{m^+}{n^+}$ that are connected by an edge,
we create the new fraction $\frac{m}{n} = \frac{m^- + m^+}{n^- + n^+}$ (this is {\em Farey addition}),
and say that this fraction is connected by an edge to both $\frac{m^-}{n^-}$ and $\frac{m^+}{n^+}$.
Assuming $\frac{m^-}{n^-} < \frac{m^+}{n^+}$,
we refer to $\frac{m^-}{n^-}$ and $\frac{m^+}{n^+}$
as the {\em left} and {\em right roots} of $\frac{m}{n}$, respectively.

For any $\frac{m}{n}$ in the Farey tree,
$\frac{m}{n}$ is irreducible, i.e.~${\rm gcd}(m,n) = 1$,
and its left and right roots satisfy $m^+ n^- - m^- n^+ = 1$.
As a consequence, $m n^- - m^- n = 1$, $m^+ n - m n^+ = 1$,
$d = n^-$, and $-d {\rm ~mod~} n = n^+$
(where again $d$ is the multiplicative inverse of $m$ modulo $n$).
By applying these observations to the quantities in Definition \ref{df:cG}
we immediately obtain the following result (illustrated in Fig.~\ref{fig:Farey}).

\begin{lemma}
Write $m_0^\pm = m^\pm$ and $n_0^\pm = n^\pm$ (to accommodate the case $k=1$).
For all $k \in \mathbb{Z}^+$, the left and right roots of $\frac{m_k^+}{n_k^+}$
are $\frac{m}{n}$ and $\frac{m_{k-1}^+}{n_{k-1}^+}$, respectively,
and the left and right roots of $\frac{m_k^-}{n_k^-}$
are $\frac{m_{k-1}^-}{n_{k-1}^-}$ and $\frac{m}{n}$, respectively.
Moreover $d_k^+ = n$ and $-d_k^- {\rm ~mod~} n_k^- = n$\removableFootnote{
This is more useful than $d_k^- = n_{k-1}^-$.
}.
\label{le:mndkpm}
\end{lemma}

\begin{figure}[b!]
\begin{center}
\setlength{\unitlength}{1cm}
\begin{picture}(8,6)
\put(0,0){\includegraphics[height=6cm]{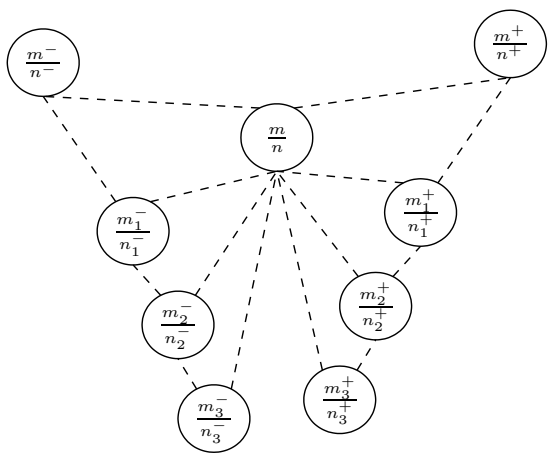}}
\put(.62,5.18){\scriptsize $\frac{m^-}{n^-}$}
\put(1.815,2.935){\scriptsize $\frac{m_1^-}{n_1^-}$}
\put(2.415,1.69){\scriptsize $\frac{m_2^-}{n_2^-}$}
\put(2.89,.445){\scriptsize $\frac{m_3^-}{n_3^-}$}
\put(3.835,4.18){\scriptsize $\frac{m}{n}$}
\put(4.57,.69){\scriptsize $\frac{m_3^+}{n_3^+}$}
\put(5.045,1.935){\scriptsize $\frac{m_2^+}{n_2^+}$}
\put(5.645,3.18){\scriptsize $\frac{m_1^+}{n_1^+}$}
\put(6.83,5.425){\scriptsize $\frac{m^+}{n^+}$}
\end{picture}
\caption{
Part of the Farey tree centred about
an arbitrary irreducible fraction $\frac{m}{n}$.
\label{fig:Farey}
}
\end{center}
\end{figure}

The next result concerns $\tilde{\ell}$ -- the number of $L$'s in $\cG^\pm[k,\chi]$.
This result is useful in later sections because $\tilde{\ell} d_k^\pm {\rm ~mod~} n_k^\pm$
is one of the four indices corresponding to a curve of border-collision bifurcations
emanating from a $\cG^\pm[k,\chi]$-shrinking point.

\begin{lemma}
For any $k \in \mathbb{Z}^+$, $|\chi| < k$, and $\cF[\ell,m,n]$,
\begin{equation}
\tilde{\ell} d_k^\pm {\rm ~mod~} n_k^\pm =
\left( \ell d {\rm ~mod~} n \pm \chi n \right) {\rm ~mod~} n_k^\pm \;, \label{eq:lkdkplusminus}
\end{equation}
\end{lemma}


\begin{proof}
By (\ref{eq:lmnkpm}) and the definition of $\ell^+$,
\begin{equation}
\ell_k^+ = k \ell + \left\lceil \frac{\ell n^+}{n} \right\rceil =
k \ell + \frac{\ell n^+}{n} +
\left( \frac{-\ell n^+}{n} {\rm ~mod~} 1 \right)  =
\frac{\ell n_k^+}{n} + \left( \frac{\ell d}{n} {\rm ~mod~} 1 \right) \;,
\label{eq:ellkplus}
\end{equation}
where in the last step we substituted $n^+ = n - d$.
Since $d_k^+ = n$ (see Lemma \ref{le:mndkpm}), (\ref{eq:ellkplus}) implies (\ref{eq:lkdkplusminus})
in the ``$+$ case''.
Similarly,
\begin{equation}
\ell_k^- = k \ell + \left\lfloor \frac{\ell n^-}{n} \right\rfloor =
k \ell + \frac{\ell n^-}{n} - \left( \frac{\ell n^-}{n} {\rm ~mod~} 1 \right) =
\frac{\ell n_k^-}{n} - \left( \frac{\ell d}{n} {\rm ~mod~} 1 \right) \;,
\label{eq:ellkminus}
\end{equation}
which with $-d_k^- {\rm ~mod~} n_k^- = n$, leads to (\ref{eq:lkdkplusminus})
in the ``$-$ case''.
\end{proof}

The next result provides us with an alternative interpretation of $\ell^+$ and $\ell^-$.
Here we state the result, present an example, then give a proof.

\begin{lemma}
For any $\cF[\ell,m,n]$,
$\ell^+$ is equal to the number of $L$'s in $\hat{\cX}$,
and $\ell^-$ is equal to the number of $L$'s in $\hat{\cY}$.
Moreover, $\ell^+ + \ell^- = \ell$.
\label{le:ellplusminus}
\end{lemma}


\begin{example}
Consider $\cF[3,5,7] = LRRLRRL$.
Here $\frac{m}{n} = \frac{5}{7}$, thus $\frac{m^+}{n^+} = \frac{3}{4}$ and $\frac{m^-}{n^-} = \frac{2}{3}$.
Therefore $\ell^+ = \left\lceil \frac{\ell n^+}{n} \right\rceil =
\left\lceil \frac{12}{7} \right\rceil = 2$,
and $\ell^- = \left\lfloor \frac{\ell n^-}{n} \right\rfloor =
\left\lfloor \frac{9}{7} \right\rfloor = 1$.

Also $d = 3$, so by (\ref{eq:XYhat}), $\hat{\cX} = LRRL$ and $\hat{\cY} = RRL$.
Thus the number of $L$'s in $\hat{\cX}$ is $2$
and the number of $L$'s in $\hat{\cY}$ is $1$,
in agreement with Lemma \ref{le:ellplusminus}.
\end{example}	

\begin{proof}[Proof of Lemma \ref{le:ellplusminus}]
Let $\hat{\ell}^+$ denote the number of $L$'s in $\hat{\cX}$.
By (\ref{eq:rss}),
\begin{equation}
\hat{\ell}^+ = \Big| \left\{ i = 0,\ldots,n^+-1 ~\middle|~ i m {\rm ~mod~} n < \ell \right\} \Big| \;.
\label{eq:ellPlusProof1}
\end{equation}
For each $i = 0,\ldots,n^+-1$, 
\begin{equation}
i m {\rm ~mod~} n =
\left( i m + \frac{i}{n^+} \right) {\rm ~mod~} n =
\frac{i m^+ n}{n^+} {\rm ~mod~} n \;,
\label{eq:ellPlusProof2}
\end{equation}
where we have used $m^+ n - m n^+ = 1$ in the last equality.
Using (\ref{eq:ellPlusProof2}) we can rewrite (\ref{eq:ellPlusProof1}) as
\begin{equation}
\hat{\ell}^+ =
\left| \left\{ i = 0,\ldots,n^+-1 ~\middle|~
i m^+ {\rm ~mod~} n^+ < \frac{\ell n^+}{n} \right\} \right| \;.
\label{eq:ellPlusProof3}
\end{equation}
Since ${\rm gcd}(m^+,n^+) = 1$, (\ref{eq:ellPlusProof3}) is the same as
$\hat{\ell}^+ = 
\left| \left\{ j = 0,\ldots,n^+-1 ~\middle|~
j {\rm ~mod~} n^+ < \frac{\ell n^+}{n} \right\} \right|$.
That is, $\hat{\ell}^+ = \left\lceil \frac{\ell n^+}{n} \right\rceil = \ell^+$,
as required.

Also,
$\ell^- =
\left\lfloor \frac{\ell n^-}{n} \right\rfloor =
\ell + \left\lfloor \frac{\ell n^-}{n} - \ell \right\rfloor =
\ell - \left\lceil \ell - \frac{\ell n^-}{n} \right\rceil =
\ell - \left\lceil \frac{\ell n^+}{n} \right\rceil =
\ell - \ell^+$,
where we have used $n^- + n^+ = n$.
Therefore $\ell^-$ is equal to the
number of $L$'s in $\cF[\ell,m,n]$
minus the number of $L$'s in $\hat{\cX}$.
Since $\cF[\ell,m,n] = \hat{\cX} \hat{\cY}$, see (\ref{eq:partitions1}),
$\ell^-$ equals the number of $L$'s in $\hat{\cY}$.
\end{proof}



The final result of this section provides
explicit expressions for $\cG^\pm[k,\chi]$ in terms of $\cS = \cF[\ell,m,n]$,
$\hat{\cX}$ and $\hat{\cY}$\removableFootnote{
These expressions can be rewritten in the following symmetric manner:
\begin{align}
\cG^+[k,\chi]^{\left( \tilde{\ell} d_k^+ \right)} &= 
\left( \cY \cX^{\overline{0}} \right)^{-\chi-1} \check{\cX}
\left( \cS^{((\ell-1)d)} \right)^{k+\chi+1} \;, \quad \chi = -k+1,\ldots,-1 \;, \\
\cG^+[k,\chi] &=
\left( \cX \cY^{\overline{0}} \right)^{\chi} \hat{\cX}
\left( \cS^{(-d)} \right)^{k-\chi} \;, \quad \chi = 0,\ldots,k-1 \;, \\
\cG^-[k,\chi]^{\left( -d_k^- \right)} &=
\left( \cX^{\overline{0}} \cY \right)^{-\chi} \hat{\cY}
\cS^{k+\chi} \;, \quad \chi = -k+1,\ldots,0 \;, \\
\cG^-[k,\chi]^{\left( \left( \tilde{\ell}-1 \right) d_k^- \right)} &=
\left( \cY^{\overline{0}} \cX \right)^{\chi-1} \check{\cY}
\left( \cS^{(\ell d)} \right)^{k-\chi+1} \;, \quad \chi = 1,\ldots,k-1 \;.
\end{align}
}.
Again we state the result, present an example, then give a proof.

\begin{proposition}
For any $k \in \mathbb{Z}^+$, $|\chi| < k$, and $\cF[\ell,m,n]$,
\begin{align}
\cG^+[k,\chi] &= \begin{cases}
\cS^{k+\chi} \hat{\cX} \left( \cS^{\overline{0}} \right)^{-\chi} \;, &
\chi = -k+1,\ldots,-1 \\
\left( \cS^{\overline{\ell d}} \right)^{\chi} \cS^{k-\chi} \hat{\cX} \;, &
\chi = 0,\ldots,k-1
\end{cases} \;, \label{eq:Gplus} \\
\cG^-[k,\chi] &= \begin{cases}
\cS \left( \cS^{\overline{0}} \right)^{-\chi} \hat{\cY} \cS^{k+\chi-1} \;, &
\chi = -k+1,\ldots,0 \\
\cS^{\overline{\ell d}} \hat{\cY} \cS^{k-\chi}
\left( \cS^{\overline{\ell d}} \right)^{\chi - 1} \;, &
\chi = 1,\ldots,k-1
\end{cases} \;. \label{eq:Gminus}
\end{align}
\label{pr:Gall}
\end{proposition}

\begin{example}
Consider again $\cF[3,5,7] = LRRLRRL$.
Here $\hat{\cX} = LRRL$ and $\hat{\cY} = RRL$.
With $k = 3$, for example,
\begin{align*}
\ell_3^+ &= 3 \times 3 + 2 = 11 \;, &
\ell_3^- &= 3 \times 3 + 1 = 10 \;, \\
m_3^+ &= 3 \times 5 + 3 = 18 \;, &
m_3^- &= 3 \times 5 + 2 = 17 \;, \\
n_3^+ &= 3 \times 7 + 4 = 25 \;, &
n_3^- &= 3 \times 7 + 3 = 24 \;.
\end{align*}
Therefore with $\chi = 0$, for example,
\begin{equation}
\begin{gathered}
\cG^+[3,0] = \cF[11,18,25] = LRRLRRLLRRLRRLLRRLRRLLRRL = \cS^3 \hat{\cX} \;, \\
\cG^-[3,0] = \cF[10,17,24] = LRRLRRLRRLLRRLRRLLRRLRRL = \cS \hat{\cY} \cS^2 \;,
\end{gathered}
\nonumber
\end{equation}
matching Proposition \ref{pr:Gall}.

These sequences are illustrated in Fig.~\ref{fig:nearbySchem}.
Notice that for both $\cG^+[3,0]$ and $\cG^-[3,0]$, node $2$ lies immediately to the right of
the upper intersection of the circle and the line.
This is because with $\chi = 0$, by (\ref{eq:lkdkplusminus}),
$\tilde{\ell} d_k^\pm = \ell d {\rm ~mod~} n = 3 \times 3 {\rm ~mod~} 7 = 2$.
Also $d_k^+ = n = 7$, thus node $7$ of $\cG^+[3,0]$ lies immediately to the left of node $0$.
Similarly $-d_k^- {\rm ~mod~} n_k^- = n = 7$,
thus node $7$ of $\cG^-[3,0]$ lies immediately to the right of node $0$.
\label{ex:cG}
\end{example}


\begin{figure}[b!]
\begin{center}
\setlength{\unitlength}{1cm}
\begin{picture}(12,6.4)
\put(0,0){\includegraphics[height=6cm]{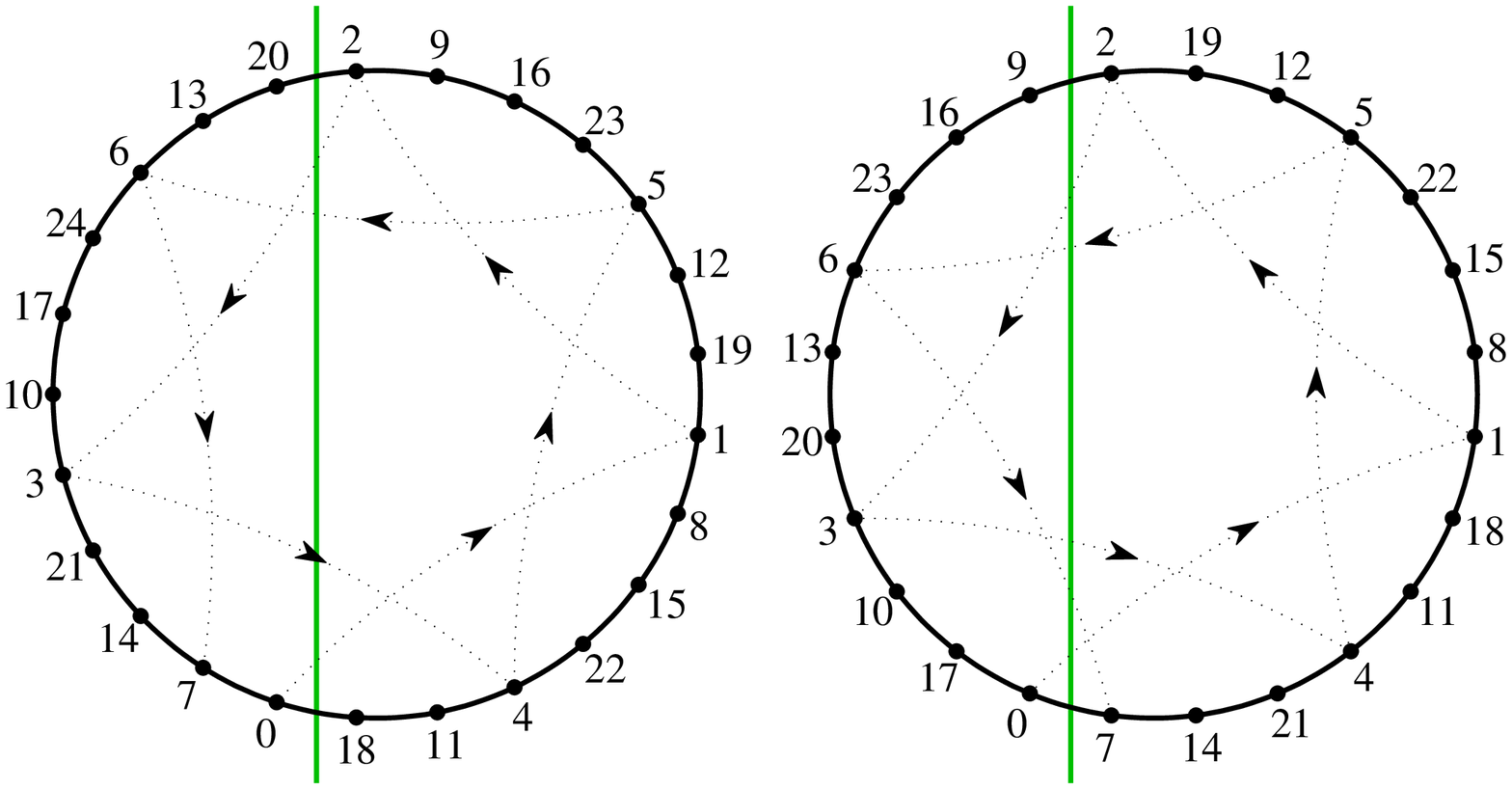}}
\put(1.8,6.3){\small $\cF[11,18,25]$}
\put(7.8,6.3){\small $\cF[10,17,24]$}
\end{picture}
\caption{
Pictorial interpretations of $\cG^+[3,0]$ and $\cG^-[3,0]$,
using $\cF[3,5,7] = LRRLRRL$, given in Example \ref{ex:cG}.
\label{fig:nearbySchem}
}
\end{center}
\end{figure}

\begin{proof}[Proof of Proposition \ref{pr:Gall}]
Here we prove the result for $\cG^+[k,\chi]$.
The result for $\cG^-[k,\chi]$ can be obtained similarly.

By the definition of a rotational symbol sequence (\ref{eq:rss}),
$\cG^+[k,\chi]_i = L$ if $\frac{i m_k^+}{n_k^+} {\rm ~mod~} 1 < \frac{\tilde{\ell}}{n_k^+}$,
and $\cG^+[k,\chi]_i = R$ otherwise.
To evaluate $\frac{i m_k^+}{n_k^+} {\rm ~mod~} 1$,
for each $i = 0,\ldots,n_k^\pm-1$, we write $i = j n + r$
with $j = 0,\ldots,k$ and $r = 0,\ldots,n-1$
(for $j = k$, $r = 0,\ldots,n^+-1$).
Using $m_k^+ n - m n_k^+ = 1$ we obtain
\begin{equation}
\frac{i m_k^+}{n_k^+} {\rm ~mod~} 1 =
\frac{(j n + r)(m n_k^+ + 1)}{n n_k^+} {\rm ~mod~} 1
= \left( \frac{r m}{n} + \frac{j}{n_k^+} + \frac{r}{n n_k^+} \right) {\rm ~mod~} 1 \;.
\end{equation}
Now let $s = r m {\rm ~mod~} n$. 
Then
\begin{equation}
\frac{i m_k^+}{n_k^+} {\rm ~mod~} 1 = \frac{s}{n} + \frac{j}{n_k^+} + \frac{r}{n n_k^+} \;,
\label{eq:Gallproof1}
\end{equation}
where the ``mod 1'' is omitted on the right hand side of (\ref{eq:Gallproof1})
because the right hand side has a value between $0$ and $1$\removableFootnote{
We can interpret the terms in the expression as follows:
$\frac{i m_k^+}{n_k^+} {\rm ~mod~} 1$ is the $i^{\rm th}$ point of the $\cG^+$-cycle,
$\frac{s}{n}$ is the $r^{\rm th}$ point of the $\cS$-cycle,
$\frac{j}{n_k^+}$ is a clockwise increment,
and $\frac{r}{n n_k^+}$ is a small correction.
}.

Next, by combining (\ref{eq:ellkplus}), (\ref{eq:Gallproof1})
and $\tilde{\ell} = \ell_k^+ + \chi$, we obtain
\begin{equation}
\frac{i m_k^+}{n_k^+} {\rm ~mod~} 1 - \frac{\tilde{\ell}}{n_k^+} =
\frac{s-\ell}{n} + \frac{j - \chi}{n_k^+} + \frac{r - \left( \ell d {\rm ~mod~} n \right)}{n n_k^+} \;.
\label{eq:GplusDiff}
\end{equation}
Notice that the sign of (\ref{eq:GplusDiff}) determines the symbol $\cG^+[k,\chi]_i$.
In contrast, $\cS_i$ is determined by the sign of
\begin{equation}
\frac{r m}{n} {\rm ~mod~} 1 - \frac{\ell}{n} = \frac{s-\ell}{n} \;.
\label{eq:SDiff}
\end{equation}
In the case $\chi = 0$,
it is straight-forward to see that for all $j$ and $r$ the right hand-sides of
(\ref{eq:GplusDiff}) and (\ref{eq:SDiff}) have the same sign,
and so $\cG^+[k,0]_i = \cS_i$ for all $i$.
We therefore have (\ref{eq:Gplus}) in the case $\chi = 0$.

With $s = \ell$, we have $r = \ell d {\rm ~mod~} n$.
Substituting this into (\ref{eq:GplusDiff}) produces
$\frac{i m_k^+}{n_k^+} {\rm ~mod~} 1 - \frac{\ell_k^+ + \chi}{n_k^+} =
\frac{j - \chi}{n_k^+}$,
which is negative if and only if $j < \chi$,
and this implies (\ref{eq:Gplus}) in the case $\chi > 0$.

Similarly with $s = \ell - 1$, we have $r = (\ell-1) d {\rm ~mod~} n$.
Since we can rewrite
$i = j n + (\ell-1) d {\rm ~mod~} n$ as
$i = \left( (j-k-1) n + \ell d {\rm ~mod~} n \right) {\rm ~mod~} n_k^+$,
into (\ref{eq:GplusDiff}) we can substitute
$r = \ell d {\rm ~mod~} n$, $s = \ell$ and $j \mapsto j-k-1$ to obtain
$\frac{i m_k^+}{n_k^+} {\rm ~mod~} 1 - \frac{\ell_k^+ + \chi}{n_k^+} =
\frac{j-k+1 - \chi}{n_k^+}$.
This is negative if and only if $j < k + 1 + \chi$,
which implies (\ref{eq:Gplus}) in the case $\chi < 0$.
\end{proof}

\section{Periodic solutions}
\label{sec:periodicSolns}
\setcounter{equation}{0}

In this section we provide some essential algebraic results for periodic solutions of (\ref{eq:f}).
Much of this theory is also developed in \cite{Si15,SiMe09,Si10}.
We begin in \S\ref{sub:linearAlgebra} by briefly reviewing some specific linear algebra concepts.
Then in \S\ref{sub:periodicSolns} we characterise periodic solutions of (\ref{eq:f})
as $\cS$-cycles and describe their properties.
Lastly in \S\ref{sub:aNonzero} we provide additional results relating to $\cS$-cycles
when it is known that the matrix $M_{\cS^{\overline{0}}}$ is non-singular
(as is the case at an $\cS$-shrinking point).

\subsection{Linear algebra tools}
\label{sub:linearAlgebra}

Given an $N \times N$ matrix $A$,
let $m_{ij}$ denote the determinant of the $(N-1) \times (N-1)$ matrix formed 
by removing the $i^{\rm th}$ row and $j^{\rm th}$ column from $A$
(the $m_{ij}$ are the {\em minors} of $A$).
The {\em adjugate} of $A$ 
is then defined by ${\rm adj}(A)_{ij} = (-1)^{i+j} m_{ji}$.
For any $A$,
\begin{equation}
{\rm adj}(A) A = A \,{\rm adj}(A) = \det(A) I \;,
\label{eq:adjIdentity}
\end{equation}
and if $A$ is nonsingular, $A^{-1} = \frac{{\rm adj}(A)}{\det(A)}$,
see \cite{Be92,Ko96,PiOd07} for further details.

The following result is known as the matrix determinant lemma\removableFootnote{
We use this most importantly in the derivations
of the formulas for $\det \left( \rho I - M_{\cT} \right)$, Lemma \ref{le:detAll},
but also in the proofs of Lemma \ref{le:slow}
and Lemma \ref{le:rankNm1} (now omitted?).
}.
For a proof using partitioned matrices, see \cite{Be05}.

\begin{lemma}
Let $A$ be an $N \times N$ matrix, and $u, v \in \mathbb{R}^N$.
Then
\begin{equation}
\det \left( A + v u^{\sf T} \right) = \det(A) + u^{\sf T} {\rm adj}(A) v \;.
\end{equation}
\label{le:matrixDeterminant}
\end{lemma}
\vspace{-5mm}								

The next result is useful to us in view of the relationship between
$A_L$ and $A_R$, (\ref{eq:continuityCondition})\removableFootnote{
In this paper it suffices to consider $u = e_1$, as in \cite{DiFe99,DiGa02,Si14d,Si15}.
Other nonsmooth dynamical systems publications that provide
this formula include \cite{DiGa02,SiMe09,SiKo09,SiMe10,Si10,SiMe12}.
In the appendix of \cite{DiFe99} the result is given with $e_i^{\sf T}$,
for arbitrary $i$, instead of $e_1^{\sf T}$.
On pg.~139 of \cite{DiBu08},
Lemma 3.1 gives a formula for
$u^{\sf T} {\rm adj} \left( C + v u^{\sf T} \right) v$,
where $C$ is a companion matrix.
}.
Indeed in later sections we only require (\ref{eq:adjIdentity2}) with $u = e_1$,
and in this case (\ref{eq:adjIdentity2})
follows immediately from the above definition of an adjugate matrix \cite{Si15,DiFe99}.
For completeness we provide a proof of Lemma \ref{le:adjIdentity2} 
in Appendix \ref{app:proofs}.

\begin{lemma}
Let $A$ be an $N \times N$ matrix, and $u, v \in \mathbb{R}^N$.
Then
\begin{equation}
u^{\sf T} {\rm adj} \left( A + v u^{\sf T} \right) = u^{\sf T} {\rm adj}(A) \;.
\label{eq:adjIdentity2}
\end{equation}
\label{le:adjIdentity2}
\end{lemma}
\vspace{-5mm}								

The next result provides an explicit formula for
the adjugate of a singular matrix.
A proof is given in Appendix \ref{app:proofs}. 
Related properties of adjugate matrices are discussed in \cite{Si93,St98}\removableFootnote{
I came up with (\ref{eq:adjIdentity3}) myself,
and I can't find a place where it is given.
The eigenvalues of ${\rm adj}(A)$ and the singular value decomposition of ${\rm adj}(A)$
are given in \cite{St98}.
My formula for ${\rm adj}(A)$ in the case that $A$ has rank $N-1$ could probably be derived
simply from the singular value decomposition.
Cases for the rank of ${\rm adj}(A)$ are derived in \cite{Si93}.
}.

\begin{lemma}
Let $A$ be an $N \times N$ matrix.
If ${\rm rank}(A) = N-1$, then
\begin{equation}
{\rm adj}(A) = c v u^{\sf T} \;,
\label{eq:adjIdentity3}
\end{equation}
where $u^{\sf T} A = 0$, $A v = 0$, $u^{\sf T} v = 1$,
and $c$ is the product of all nonzero eigenvalues of $A$, counting multiplicity.
If ${\rm rank}(A) < N-1$, then ${\rm adj}(A)$ is the zero matrix.
\label{le:adjugateRank}
\end{lemma}

\subsection{Basic properties of periodic solutions}
\label{sub:periodicSolns}

An $\cS$-cycle is a periodic solution $\{ x^{\cS}_i \}$
of the half maps of (\ref{eq:f}) in the order determined by $\cS$,
refer to Definition \ref{df:Scycle} for a formal statement.
If $s^{\cS}_i \le 0$ whenever $\cS_i = L$, and 
$s^{\cS}_i \ge 0$ whenever $\cS_i = R$,
then the $\cS$-cycle is a periodic solution of (\ref{eq:f}) and said to be {\em admissible},
otherwise it is said to be {\em virtual}.
The following result relates to border-collision bifurcations of $\cS$-cycles
(at which $s^{\cS}_j = 0$, for some $j$)
and is an immediate consequence of the continuity of (\ref{eq:f}).

\begin{lemma}
Let $\{ x^{\cS}_i \}$ be an $\cS$-cycle.
If $s^{\cS}_j = 0$, for some $j$, then $\{ x^{\cS}_i \}$ is also an $\cS^{\overline{j}}$-cycle.
\label{le:sZero}
\end{lemma}

Each $x^{\cS}_i$ is a fixed point of
\begin{equation}
f^{\cS^{(i)}}(x) = M_{\cS^{(i)}} x + P_{\cS^{(i)}} B \mu \;,
\label{eq:fSi}
\end{equation}
see (\ref{eq:fS2}), where $\cS^{(i)}$ denotes the $i^{\rm th}$ left shift permutation of $\cS$.
By (\ref{eq:MS}), changing $i$ only changes the cyclic order in which $A_L$ and $A_R$
are multiplied to produce $M_{\cS^{(i)}}$.
This is the basis for the following result which
is a minor generalisation of a result proved in \cite{SiMe09,Si10},
and so we omit a proof\removableFootnote{
How does one prove the multiplicities part?
}.

\begin{lemma}
The determinant of $M_{\cS^{(i)}}$, and its eigenvalues and the multiplicities of the eigenvalues,
are independent of $i$\removableFootnote{
It follows that the same result is true for $\rho I - M_{\cS^{(i)}}$
because the eigenvalues of this matrix equal $\rho$ minus the eigenvalues of $M_{\cS^{(i)}}$,
with multiplicities preserved.
In many places I need this result with either $\rho = 0$ or $\rho = 1$.
Only in the proof Lemma \ref{le:detAll} do I need this result for general $\rho$.
}.
\label{le:eigMSindep}
\end{lemma}

In view of (\ref{eq:fSi}), the stability of an $\cS$-cycle is
governed by the eigenvalues of $M_{\cS^{(i)}}$,
and by Lemma \ref{le:eigMSindep} it suffices to consider $i = 0$.
These observations provide us with the following result.

\begin{proposition}
An admissible $\cS$-cycle, with $s^{\cS}_i \ne 0$ for all $i$, is attracting 
if and only if all eigenvalues of $M_{\cS}$
have modulus less than $1$. 
\label{pr:stability}
\end{proposition}

Equation (\ref{eq:fSi}) provides us with an explicit expression for each $x^{\cS}_i$,
as stated in the next result.

\begin{proposition}
The $\cS$-cycle is unique if and only if $I - M_{\cS}$ is nonsingular,
and if $I - M_{\cS}$ is nonsingular then
\begin{equation}
x^{\cS}_i = \left( I - M_{\cS^{(i)}} \right)^{-1} P_{\cS^{(i)}} B \mu \;.
\label{eq:xSiGen}
\end{equation}
\label{pr:existence}
\end{proposition}
\vspace{-5mm}								

Lastly, to obtain a useful explicit expression for each $s^{\cS}_i$
(the first component of $x^{\cS}_i$), we use the row vector
\begin{equation}
\varrho^{\sf T} \defeq e_1^{\sf T} {\rm adj} \left( I - A_L \right)
= e_1^{\sf T} {\rm adj} \left( I - A_R \right) \;,
\label{eq:varrho}
\end{equation}
where the second equality is a consequence of (\ref{eq:continuityCondition}) and (\ref{eq:adjIdentity2}).
The following identity,
\begin{equation}
e_1^{\sf T} {\rm adj} \left( I-M_{\cS^{(i)}} \right) P_{\cS^{(i)}}
= \det \left( P_{\cS^{(i)}} \right) \varrho^{\sf T} \;,
\label{eq:id}
\end{equation}
is a consequence of (\ref{eq:continuityCondition}),
see \cite{Si10,SiMe10} for a derivation.
By combining (\ref{eq:xSiGen}) and (\ref{eq:id}) we obtain
\begin{equation}
\det(I-M_{\cS}) s^{\cS}_i = \det \left( P_{\cS^{(i)}} \right) \varrho^{\sf T} B \mu \;,
\label{eq:id2}
\end{equation}
from which the next result follows immediately.

\begin{proposition}~\\
\vspace{-5mm}								
\begin{enumerate}
\item
If $I - M_{\cS}$ is nonsingular, then
\begin{equation}
s^{\cS}_i = \frac{\det \left( P_{\cS^{(i)}} \right) \varrho^{\sf T} B \mu}
{\det \left( I - M_{\cS} \right)} \;.
\label{eq:sFormula}
\end{equation}
\item
If $I - M_{\cS}$ is singular, $f^{\cS}$ has a fixed point, $\mu \ne 0$, and $\varrho^{\sf T} B \ne 0$,
then $P_{\cS^{(i)}}$ is singular for all $i$.
\end{enumerate}
\label{pr:admissibility}
\end{proposition}
\vspace{-5mm}								

\subsection{Consequences of $\det \left( I - M_{\cS^{\overline{0}}} \right) \ne 0$}
\label{sub:aNonzero}

Our definition of a shrinking point (Definition \ref{df:shrPoint})
includes the assumption $\det \left( I - M_{\cS^{\overline{0}}} \right) \ne 0$.
By Proposition \ref{pr:existence}, this ensures that the $\cS^{\overline{0}}$-cycle is unique.
In this section we provide two important results requiring the assumption
$\det \left( I - M_{\cS^{\overline{0}}} \right) \ne 0$\removableFootnote{
The following result is not used in later sections
(the first part of Lemma \ref{le:varrhoNonzero} is required to show
that the $t_i$ are nonzero, \cite{Si10})
and so I have omitted it from the text.

\begin{lemma}
Suppose $I - M_{\cS^{\overline{0}}}$ is nonsingular.
Then
\begin{enumerate}
\item
$I - A_L$ and $I - A_R$ cannot both be singular;
\item
$\varrho^{\sf T} \ne 0$
\end{enumerate}
\label{le:varrhoNonzero}
\end{lemma}

\begin{proof}
\begin{enumerate}
\item
By (\ref{eq:adjIdentity}) and (\ref{eq:varrho}),
\begin{equation}
\varrho^{\sf T} \left( I - A_J \right) = \det \left( I - A_J \right) e_1^{\sf T} \;,
\label{eq:varrhoIdentity}
\end{equation}
for $J = L,R$.
If $I - A_L$ and $I - A_R$ are both singular,
then by (\ref{eq:varrhoIdentity}) we have
$\varrho^{\sf T} A_L = \varrho^{\sf T} A_R = \varrho^{\sf T}$.
$M_{\cS^{\overline{0}}}$ is a product of instances of $A_L$ and $A_R$,
hence this implies $\varrho^{\sf T} M_{\cS^{\overline{0}}} = \varrho^{\sf T}$,
and thus that $I - M_{\cS^{\overline{0}}}$ is singular.
\item
By part (i), $I - A_J$ is nonsingular for some $J = L,R$.
This implies ${\rm adj} \left( I - A_J \right)$ is nonsingular
(since $\det({\rm adj}(A)) = \det(A)^{N-1}$),
therefore $\varrho = e_1^{\sf T} {\rm adj} \left( I - A_J \right)$ cannot be the zero vector.
\end{enumerate}
\end{proof}
}.

\begin{lemma}
Suppose $I - M_{\cS^{\overline{0}}}$ is nonsingular and $I - M_{\cS}$ is singular.
Then the eigenvalue $1$ of $M_{\cS}$ has algebraic multiplicity $1$,
and the corresponding right eigenspace of $M_{\cS}$ is not orthogonal to $e_1$\removableFootnote{
I don't need a lemma stating that if $P_{\cS}$ is singular
then the algebraic multiplicity of the zero eigenvalue is $1$
(this is Lemma 3(v) of {\sc LinearAlgebra.pdf})
because the expansions of $\det \left( P_{\cS} \right)$ in later sections
implicitly imply that the algebraic multiplicity is $1$.
}.
\label{le:rankNm1}
\end{lemma}

\begin{proof}
By (\ref{eq:continuityCondition}) and (\ref{eq:MS}), we can write
\begin{equation}
M_{\cS^{\overline{0}}} = M_{\cS} + C_{\cS} e_1^{\sf T} \;,
\end{equation}
for some $C_{\cS} \in \mathbb{R}^N$.
By Lemma \ref{le:matrixDeterminant},
\begin{equation}
\det \left( I - M_{\cS^{\overline{0}}} \right) =
\det \left( I - M_{\cS} - C_{\cS} e_1^{\sf T} \right) =
e_1^{\sf T} {\rm adj} \left( I - M_{\cS} \right) C_{\cS} \;,
\end{equation}
because $\det \left( I - M_{\cS} \right) = 0$.
Since $\det \left( I - M_{\cS^{\overline{0}}} \right) \ne 0$, by assumption,
${\rm adj} \left( I - M_{\cS} \right)$ cannot be the zero matrix.
Thus by Lemma \ref{le:adjugateRank},
the eigenvalue $1$ of $M_{\cS}$ has algebraic multiplicity $1$.

Let $v \in \mathbb{R}^N$ be an eigenvector of $M_{\cS}$ corresponding to the eigenvalue $1$.
That is $v \ne 0$ and
\begin{equation}
0 = \left( I - M_{\cS} \right) v =
\left( I - M_{\cS^{\overline{0}}} \right) v + C_{\cS} e_1^{\sf T} v \;.
\end{equation}
But $\left( I - M_{\cS^{\overline{0}}} \right) v \ne 0$,
because $I - M_{\cS^{\overline{0}}}$ is nonsingular,
therefore $e_1^{\sf T} v \ne 0$ as required.
\end{proof}


\begin{lemma}
Suppose $I - M_{\cS^{\overline{0}}}$ is nonsingular,
$I - M_{\cS}$ and $P_{\cS}$ are singular,
$\mu \ne 0$, and $\varrho^{\sf T} B \ne 0$.
Then $P_{\cS^{(i)}}$ is singular for all $i$\removableFootnote{
This result is only used in later sections
to note that at a point in parameter space where $P_{\cS}$ and $I - M_{\cS}$ are singular,
by Lemma \ref{le:PSisingular} we can expect that $P_{\cS^{((\ell-1)d)}}$ is also singular,
and so we have a shrinking point (if admissibility and non-degeneracy conditions are satisfied).
}\removableFootnote{
The result remains true if the assumptions $\mu \ne 0$ and $\varrho^{\sf T} B \ne 0$ are omitted!
However, proving the result without these assumptions (done below) requires more work.
I have chosen to include the redundant assumptions so that
the reader can see the result nicely following on from Proposition \ref{pr:admissibility}(ii),
and because these assumptions are needed in regards to shrinking points.
Moreover, I am not going to mention that the assumptions are redundant in the text,
because I think this is unnecessary detail and a bit technical.
The more general result would perhaps be appropriate in a book chapter.

Lemma \ref{le:PSisingular}, without
the assumptions $\mu \ne 0$ and $\varrho^{\sf T} B \ne 0$,
is an immediate consequence of the following result.

\begin{lemma}
Suppose $P_{\cS}$ is singular and $I - M_{\cS^{\overline{0}}}$ is nonsingular.
Then for all $i$,
\begin{equation}
P_{\cS^{(i)}} = \left( I - M_{\cS^{(i)}} \right)
\left( I - M_{\cS^{\overline{0}(i)}} \right)^{-1} P_{\cS^{\overline{0}(i)}} \;.
\label{eq:PSidentity}
\end{equation}
\label{le:PSidentity}
\end{lemma}
\vspace{-5mm}								

I have a ``direct'' proof of Lemma \ref{le:PSidentity},
i.e.~without referring to a vector $B$ or periodic solutions of (\ref{eq:f}).
This is nice, however the proof is quite long requiring mathematical induction of $i$,
see Lemma 2(ix) of {\sc LinearAlgebra.pdf}.

\begin{proof}[Proof of Lemma \ref{le:PSidentity}]
The given assumptions place no restrictions on $B$ and $\mu$,
so let us choose any $B \in \mathbb{R}^N$ and $\mu \in \mathbb{R}$.

Since $I - M_{\cS^{\overline{0}}}$ is nonsingular,
by Proposition \ref{pr:existence} the map (\ref{eq:f}) has a unique $\cS^{\overline{0}}$-cycle
for which the $i^{\rm th}$ point is
\begin{equation}
x^{\cS^{\overline{0}}}_i = \left( I - M_{\cS^{\overline{0}(i)}} \right)^{-1}
P_{\cS^{\overline{0}(i)}} B \mu \;.
\label{eq:PSidentityProof1}
\end{equation}
Since $P_{\cS^{\overline{0}}} = P_{\cS}$, see (\ref{eq:PS}),
and $P_{\cS}$ is assumed to be singular,
by Proposition \ref{pr:admissibility}(i) we have $s^{\cS^{\overline{0}}}_0 = 0$.
Therefore by Lemma \ref{le:sZero},
the $\cS^{\overline{0}}$-cycle is also an $\cS$-cycle.
That is,
\begin{equation}
\left( I - M_{\cS^{(i)}} \right) x^{\cS^{\overline{0}}}_i = P_{\cS^{(i)}} B \mu \;.
\label{eq:PSidentityProof2}
\end{equation}
By combining (\ref{eq:PSidentityProof1}) and (\ref{eq:PSidentityProof2}), we obtain
\begin{equation}
P_{\cS^{(i)}} B \mu = \left( I - M_{\cS^{(i)}} \right)
\left( I - M_{\cS^{\overline{0}(i)}} \right)^{-1} P_{\cS^{\overline{0}(i)}} B \mu \;.
\label{eq:PSidentityProof3}
\end{equation}
Since (\ref{eq:PSidentityProof3}) holds for any
$B \in \mathbb{R}^N$ and $\mu \in \mathbb{R}$,
we must have (\ref{eq:PSidentity})
(to actually see this, simply put $\mu = 1$ and $B = e_1,\ldots,e_N$).
\end{proof}
}.
\label{le:PSisingular}
\end{lemma}

\begin{proof}
Since $I - M_{\cS^{\overline{0}}}$ is nonsingular,
there exists a unique $\cS^{\overline{0}}$-cycle,
$\left\{ x^{\cS^{\overline{0}}}_i \right\}$.
The matrix $P_{\cS^{\overline{0}}}$ is singular
because $P_{\cS^{\overline{0}}} = P_{\cS}$, (\ref{eq:PS}).
Therefore $s^{\cS^{\overline{0}}}_0 = 0$, (\ref{eq:sFormula}).
By Lemma \ref{le:sZero},
$\left\{ x^{\cS^{\overline{0}}}_i \right\}$ is also an $\cS$-cycle,
and so $x^{\cS^{\overline{0}}}_0$ is a fixed point of $f^{\cS}$.
Therefore, by Proposition \ref{pr:admissibility}(ii),
$P_{\cS^{(i)}}$ is singular for all $i$.
\end{proof}

\section{Shrinking points}
\label{sec:shrPoints}
\setcounter{equation}{0}

Our definition of a shrinking point, Definition \ref{df:shrPoint},
is based on the conditions $s^{\cS^{\overline{0}}}_0 = 0$
and $s^{\cS^{\overline{0}}}_{\ell d} = 0$,
where the $s^{\cS^{\overline{0}}}_i$ are the first coordinates of the points
of the $\cS^{\overline{0}}$-cycle,
and $\cS = \cF[\ell,m,n]$ is a rotational symbol sequence.
Here we begin by providing additional
motivation for our restriction to rotational symbol sequences,
and the choice of the indices $i = 0$ and $i = \ell d$ in Definition \ref{df:shrPoint}.

\subsection{Motivation for Definition \ref{df:shrPoint}}
\label{sub:shrPointMotivation}

Conceptually, a shrinking point is a point in parameter space where (\ref{eq:f}) has
an $\cS$-cycle, $\{ x^{\cS}_i \}$, with two points on the switching manifold.
Without loss of generality we can suppose that one of these points is $x^{\cS}_0$, and that $\cS_0 = L$.
Let $x^{\cS}_\alpha$, where $1 \le \alpha \le n-1$
and $n$ is the period, be the other point of the $\cS$-cycle on the switching manifold.
By a double application of Lemma \ref{le:sZero},
this $\cS$-cycle is also an $\cS^{\overline{0} \, \overline{\alpha}}$-cycle.
If $\cS_\alpha = R$, then
$\cS^{\overline{0} \, \overline{\alpha}}$ has the same number of $L$'s as $\cS$.
In this case it is possible for there to exist an integer $d$ such that
\begin{equation}
\cS^{\overline{0} \, \overline{\alpha} (d)} = \cS \;.
\label{eq:rssMainIdentityAlt}
\end{equation}
That is, if we flip the $0^{\rm th}$ and $\alpha^{\rm th}$ symbols of $\cS$,
then apply the $d^{\rm th}$ left shift permutation, we recover the original symbol sequence.

The next result tells us that if (\ref{eq:rssMainIdentityAlt}) holds,
then $\cS$ must equal $\cF[\ell,m,n]$ for some integers $\ell$ and $m$,
and $\alpha = \ell d {\rm ~mod~} n$.
In other words, our restriction to rotational symbol sequences
and the choice of the indices $i = 0$ and $i = \ell d$ in Definition \ref{df:shrPoint}
can be viewed as a consequence of supposing that the symbol sequence associated with a shrinking point
satisfies (\ref{eq:rssMainIdentityAlt}) for some values of $\alpha$ and $d$.

\begin{proposition}
Let $\cS$ be a periodic symbol sequence of period $n$ with $\cS_0 = L$, and suppose
$\cS^{\overline{0} \, \overline{\alpha} (d)} = \cS$,
for some $1 \le \alpha \le n-1$ and $1 \le d \le n-1$ with ${\rm gcd}(d,n) = 1$\removableFootnote{
Suppose instead ${\rm gcd}(d,n) = g$, for some $g > 1$,
and for simplicity suppose $\cS_0 = L$.
Let $\cU = \cS_0 \cdots \cS_{g-1}$, and consider the map
\begin{equation}
x_{i+1} = \begin{cases}
f^{\cU}(x_i) \;, & s_i \le 0 \\
f^{\cU^{\overline{0}}}(x_i) \;, & s_i \ge 0
\end{cases} \;.
\end{equation}
Let $\tilde{n} = \frac{n}{g}$,
and $\cT = \cS_0 \cS_g \cdots \cS_{(\tilde{n}-1) g}$.
Then it is straight-forward to show that
$\cT$ is rotational and $\cS$-cycles of (\ref{eq:f})
correspond to $\cT$-cycles of the above map.

Note, I haven't developed the notation to compound symbol sequences
as in this scenario, e.g.~$\cS = \cT(\cU)$.
}.
Let $m$ denote the multiplicative inverse of $d$ modulo $n$,
and $\ell = m \alpha {\rm ~mod~} n$.
Then $\cS = \cF[\ell,m,n]$, and $\alpha = \ell d {\rm ~mod~} n$.
\label{pr:rss}
\end{proposition}

\begin{proof}
The formula $\alpha = \ell d {\rm ~mod~} n$
is a trivial consequence of our definitions of $\ell$ and $m$ in the statement of the theorem.
It remains to show that $\cS = \cF[\ell,m,n]$,
which we achieve by verifying (\ref{eq:rssAlt}).

For any $j = 1,\ldots,\ell-1$,
\begin{equation}
\cS_{j d {\rm \,mod\,} n} =
\cS^{\overline{0} \, \overline{\alpha}}_{j d {\rm \,mod\,} n} =
\cS^{\overline{0} \, \overline{\alpha} (d)}_{(j-1) d {\rm \,mod\,} n} \;,
\label{eq:rssProof1}
\end{equation}
because $j \ne 0,\ell$, and where the second equality
follows from the definition of a shift permutation.
Then
\begin{equation}
\cS^{\overline{0} \, \overline{\alpha} (d)}_{(j-1) d {\rm \,mod\,} n} =
\cS_{(j-1) d {\rm \,mod\,} n} \;,
\label{eq:rssProof2}
\end{equation}
because $\cS^{\overline{0} \, \overline{\alpha} (d)} = \cS$, by assumption.
By then starting with $\cS_0 = L$,
and recursively applying (\ref{eq:rssProof1})-(\ref{eq:rssProof2}),
we obtain $\cS_0 = \cS_d = \cdots = \cS_{(\ell-1) d {\rm \,mod\,} n}$,
matching (\ref{eq:rssAlt}).

Equation (\ref{eq:rssProof2}) is true for all $j \in \mathbb{Z}$,
whereas (\ref{eq:rssProof1}) is false for $j = \ell$
(because $\ell d {\rm ~mod~} n = \alpha$).
This implies $\cS_{\ell d {\rm \,mod\,} n} \ne \cS_{(\ell-1)d {\rm \,mod\,} n}$,
and thus $\cS_{\ell d {\rm \,mod\,} n} = R$.
Finally, (\ref{eq:rssProof1}) is true for all $j = \ell+1,\ldots,n-1$
and so a similar recursive argument gives us
$\cS_{\ell d {\rm \,mod\,} n} = \cS_{(\ell d+1) {\rm \,mod\,} n} = \cdots =
\cS_{-d {\rm \,mod\,} n} = R$,
which verifies (\ref{eq:rssAlt}), and hence $\cS = \cF[\ell,m,n]$.
\end{proof}

\subsection{Basic properties of shrinking points}
\label{sub:shrPoints}

\begin{figure}[b!]
\begin{center}
\setlength{\unitlength}{1cm}
\begin{picture}(8,6)
\put(0,0){\includegraphics[height=6cm]{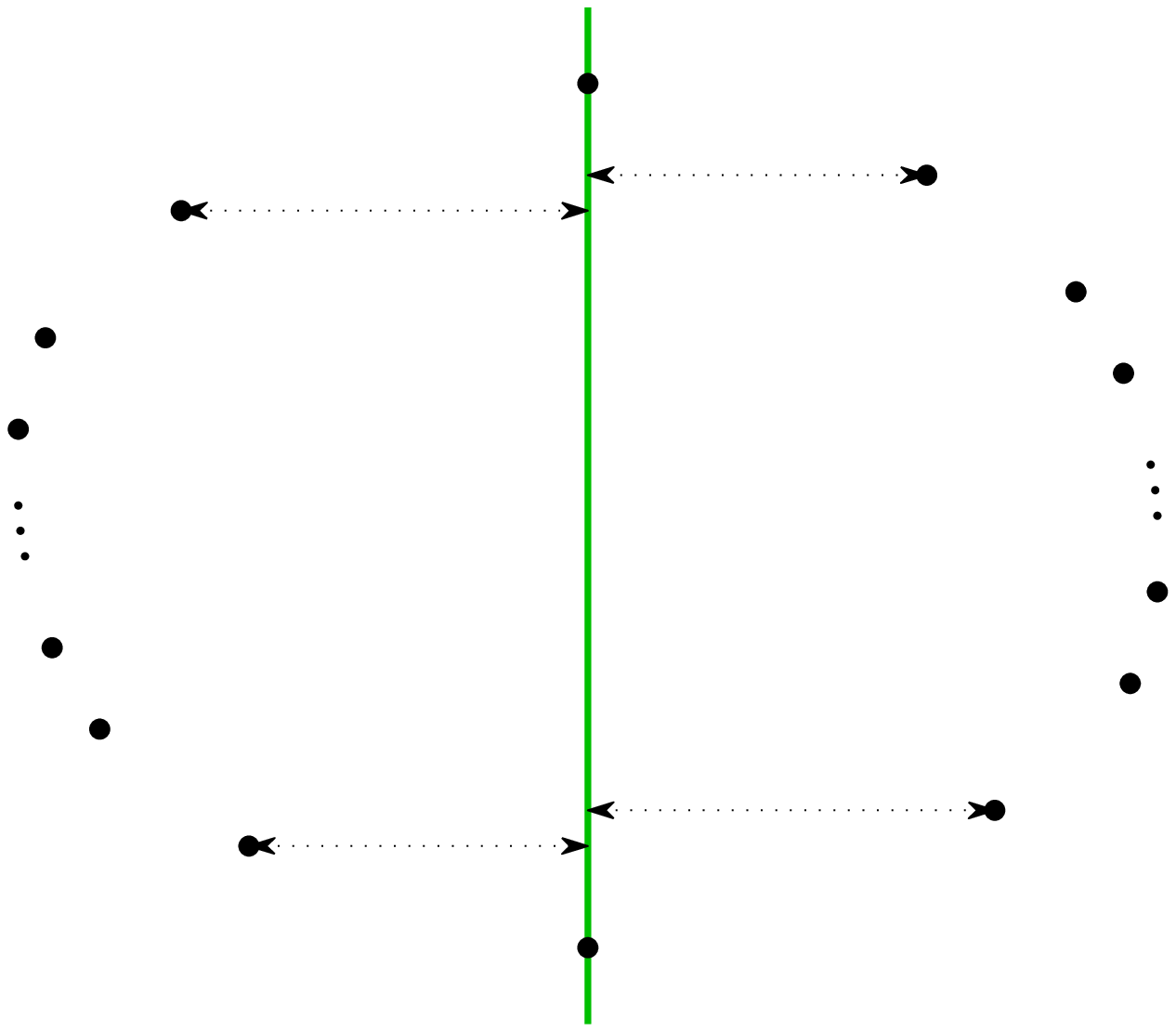}}
\put(.3,4.25){\scriptsize $y_{(\ell-2)d}$}
\put(1.1,5){\scriptsize $y_{(\ell-1)d}$}
\put(3.5,5.6){\scriptsize $y_{\ell d}$}
\put(5.8,5.2){\scriptsize $y_{(\ell+1)d}$}
\put(6.7,4.55){\scriptsize $y_{(\ell+2)d}$}
\put(2.17,4.58){\scriptsize $-t_{(\ell-1)d}$}
\put(4.57,4.78){\scriptsize $t_{(\ell+1)d}$}
\put(.6,1.6){\scriptsize $y_{2 d}$}
\put(1.66,.84){\scriptsize $y_d$}
\put(3.66,.3){\scriptsize $y_0$}
\put(6.15,1.02){\scriptsize $y_{-d}$}
\put(7,1.74){\scriptsize $y_{-2 d}$}
\put(2.66,1.23){\scriptsize $-t_d$}
\put(4.9,1.38){\scriptsize $t_{-d}$}
\put(4.04,3){\scriptsize $s=0$}
\end{picture}
\caption{
A schematic diagram showing the $\cS^{\overline{0}}$-cycle at an $\cS$-shrinking point.
\label{fig:perSolnAtShrPoint}
}
\end{center}
\end{figure}


Recall, at an $\cS$-shrinking point, $y_i$ denotes the $i^{\rm th}$ point of the $\cS^{\overline{0}}$-cycle
and $t_i$ denotes the first coordinate of $y_i$, see \S\ref{sub:definitions}.
By assumption
\begin{equation}
t_0 = 0 \;, \qquad
t_{\ell d} = 0 \;,
\label{eq:t0telld}
\end{equation}
and $t_i \ne 0$, for all $i \ne 0,\ell d$.
The $\cS^{\overline{0}}$-cycle is assumed to be admissible,
and this fixes the signs of the $t_i$.
In particular,
\begin{equation}
t_d < 0 \;, \qquad
t_{(\ell-1)d} < 0 \;, \qquad
t_{(\ell+1)d} > 0 \;, \qquad
t_{-d} > 0 \;,
\label{eq:tSigns}
\end{equation}
see Fig.~\ref{fig:perSolnAtShrPoint}\removableFootnote{
In Definition \ref{df:shrPoint} we include the assumption $t_i \ne 0$, for all $i \ne 0,\ell d$,
because it is required in most of the results below, especially Theorem \ref{th:basicUnfolding}.
Without this assumption (\ref{eq:tSigns}) still holds \cite{Si10}.

In my proof in \cite{Si10} I missed that logical point
to note that $I-A_L$ and $I-A_R$ cannot both be singular.
}.

The next three results provide key properties of shrinking points.

\begin{proposition}
At an $\cS$-shrinking point, $\{ y_i \}$ has period $n$.
\label{pr:periodn}
\end{proposition}

This is proved in \cite{Si10}
and is a consequence of the rotational nature of $\cF[\ell,m,n]$\removableFootnote{
In {\sc LinearAlgebra.pdf} I showed directly that
$\left( I-M_{\cS^{\overline{0}(d)}} \right)^{-1} P_{\cS^{\overline{0}(d)}} \ne
\left( I-M_{\cS^{\overline{0}}} \right)^{-1} P_{\cS^{\overline{0}}}$.
The proof is rather long (it uses mathematical induction and different cases).
}.

\begin{proposition}
At an $\cS$-shrinking point,
\begin{enumerate}
\item
$\{ y_i \}$ is both an $\cS$-cycle and an $\cS^{(-d)}$-cycle,
\item
$I-M_{\cS}$ is singular\removableFootnote{
Note, by Lemma \ref{le:rankNm1},
the algebraic multiplicity of the unit eigenvalue of $M_{\cS}$ is $1$.
},
\item
$P_{\cS^{(i)}}$ is singular for all $i$.
\end{enumerate}
\label{pr:MSPSsingular}
\end{proposition}

\begin{proof}
By definition, $\{ y_i \}$ is an $\cS^{\overline{0}}$-cycle.
Since $t_0 = 0$, by Lemma \ref{le:sZero}, $\{ y_i \}$ is also an $\cS$-cycle.
Since $t_{\ell d} = 0$, $\{ y_i \}$ is similarly also an $\cS^{\overline{0} \, \overline{\ell d}}$-cycle,
and so by (\ref{eq:rssMainIdentity}), $\{ y_i \}$ is also an $\cS^{(-d)}$-cycle,
which proves part (i).

By part (i), $y_0$ and $y_d$ are both fixed points of $f^{\cS}$.
By (\ref{eq:t0telld}) and (\ref{eq:tSigns}) these points are distinct,
hence by Proposition \ref{pr:existence}, $I - M_{\cS}$ is singular.
Finally, $P_{\cS^{(i)}}$ is singular for all $i$ by Proposition \ref{pr:admissibility}(ii).
\end{proof}

\begin{proposition}
For any $\cS$-shrinking point in the phase space of (\ref{eq:f}),
let $\cP$ denote the nonplanar polygon formed by
joining each $y_i$ to $y_{(i+d)}$ by a line segment.
Then each point on $\cP$ belongs to an admissible $\cS$-cycle of (\ref{eq:f}),
and the restriction of (\ref{eq:f}) to $\cP$
is homeomorphic to rigid rotation with rotation number $\frac{m}{n}$.
\label{pr:polygon}
\end{proposition}

Refer to \cite{SiMe09,Si10} for a proof\removableFootnote{
The proof is elementary except a short argument
is required to show that $\cP$ has no self-intersections.
}.

\subsection{Eigenvectors associated with shrinking points}
\label{sub:eigenvectors}

By Proposition \ref{pr:MSPSsingular}(ii),
$M_{\cS}$ has a unit eigenvalue.
Since $\det \left( I - M_{\cS^{\overline{0}}} \right) \ne 0$,
by Lemma \ref{le:rankNm1} the unit eigenvalue is of algebraic multiplicity one.
Furthermore, by Lemma \ref{le:eigMSindep} each $M_{\cS^{(i)}}$
has a unit eigenvalue of algebraic multiplicity one.

Recall, in \S\ref{sub:definitions} we let $u_j^{\sf T}$ and $v_j$
denote the left and right eigenvectors of $M_{\cS^{(j)}}$ corresponding to the unit eigenvalue
normalised by $u_j^{\sf T} v_j = 1$ and $e_1^{\sf T} v_j = 1$,
for $j = 0, (\ell-1)d, \ell d, -d$.
The following result provides explicit expressions for $u_j^{\sf T}$ and $v_j$.

\begin{lemma}
At an $\cS$-shrinking point, for each $j \in \left\{ 0, (\ell-1)d, \ell d, -d \right\}$,
\begin{equation}
u_j^{\sf T} = \frac{e_1^{\sf T} {\rm adj} \left( I - M_{\cS^{(j)}} \right)}{c} \;, \qquad
v_j = \frac{y_{j+d}-y_j}{t_{j+d}-t_j} \;,
\label{eq:uvj}
\end{equation}
where $c$ is given by (\ref{eq:c}).
\label{le:uv}
\end{lemma}

\begin{proof}
By applying Lemma \ref{le:adjugateRank} to the matrix $A = I - M_{\cS^{(j)}}$,
we obtain ${\rm adj} \left( I - M_{\cS^{(j)}} \right) = c v_j u_j^{\sf T}$.
Since also $e_1^{\sf T} v_j = 1$,
this implies $\frac{e_1^{\sf T} {\rm adj} \left( I - M_{\cS^{(j)}} \right)}{c} = u_j^{\sf T}$.

Next let $\hat{v}_j = \frac{y_{j+d}-y_j}{t_{j+d}-t_j}$.
It remains to show that $\hat{v}_j = v_j$.
Trivially $e_1^{\sf T} \hat{v}_j = 1$.
By Proposition \ref{pr:MSPSsingular},
$y_j$ and $y_{j+d}$ are both fixed points of $f^{\cS^{(j)}}$, and thus
\begin{equation}
y_{j+d} - y_j = 
f^{\cS^{(j)}} \left( y_{j+d} \right) - f^{\cS^{(j)}} \left( y_j \right) =
M_{\cS^{(j)}} \left( y_{j+d} - y_j \right) \;.
\end{equation}
Therefore $M_{\cS^{(j)}} \hat{v}_j = \hat{v}_j$.
That is, each $\hat{v}_j$ satisfies the same properties as $v_j$.
But $v_j$ is unique, hence $\hat{v}_j = v_j$, as required.
\end{proof}

\begin{figure}[b!]
\begin{center}
\setlength{\unitlength}{1cm}
\begin{picture}(16,12)
\put(0,0){\includegraphics[height=12cm]{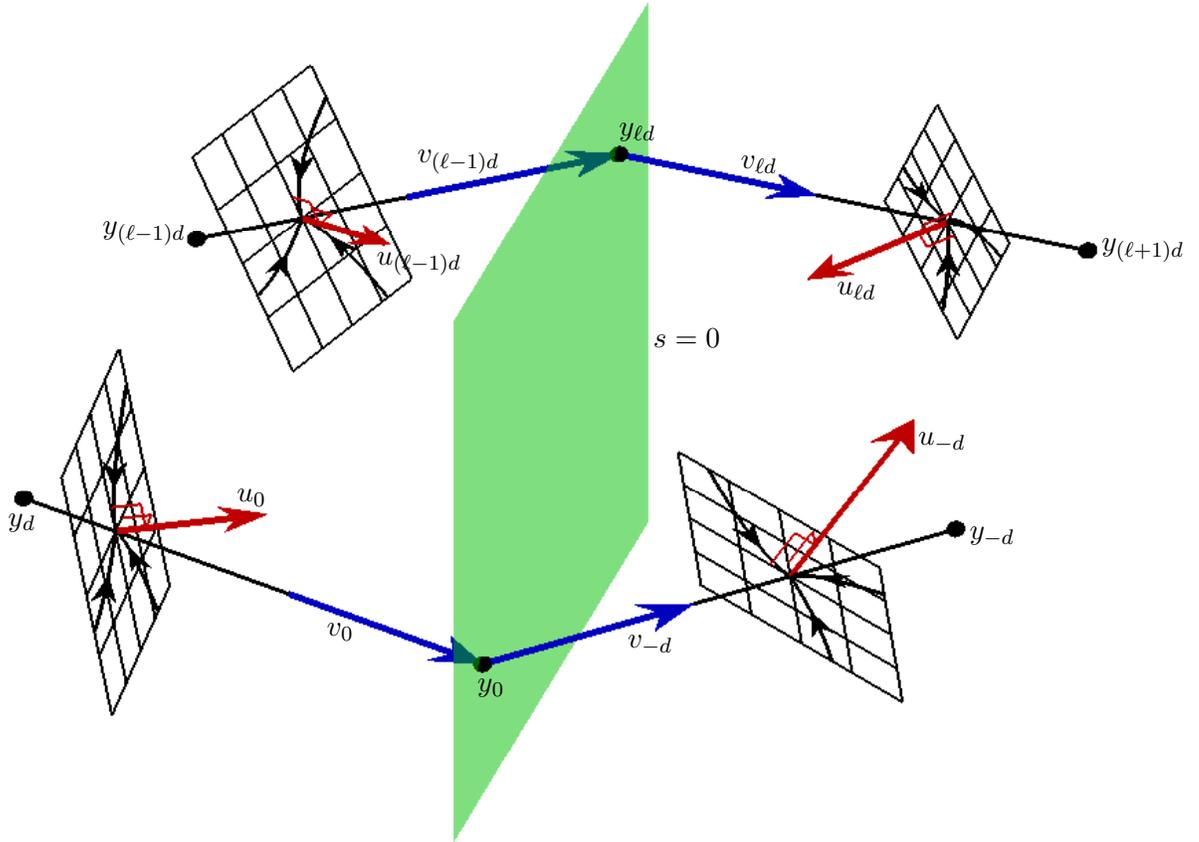}}
\put(9.34,7){\small $s=0$}
\put(7,2.42){\small $y_0$}
\put(.76,4.62){\small $y_d$}
\put(2.02,8.5){\small $y_{(\ell-1)d}$}
\put(8.88,9.82){\small $y_{\ell d}$}
\put(15.3,8.28){\small $y_{(\ell+1)d}$}
\put(13.54,4.44){\small $y_{-d}$}
\put(5,3.16){\small $v_0$}
\put(6.2,9.46){\small $v_{(\ell-1)d}$}
\put(10.5,9.38){\small $v_{\ell d}$}
\put(9,3){\small $v_{-d}$}
\put(3.8,4.94){\small $u_0$}
\put(5.66,8.1){\small $u_{(\ell-1)d}$}
\put(11.78,7.7){\small $u_{\ell d}$}
\put(12.86,5.7){\small $u_{-d}$}
\end{picture}
\caption{
A schematic diagram illustrating the periodic solution $\{ y_i \}$
and the eigenvectors $u_j^{\sf T}$ and $v_j$.
\label{fig:eigSchem}
}
\end{center}
\end{figure}

The eigenvectors $u_j^{\sf T}$ and $v_j$ are sketched in Fig.~\ref{fig:eigSchem}.
By Lemma \ref{le:uv}, the eigenvector $v_0$, for example,
has the same direction as the line segment connecting $y_0$ to $y_d$, and $e_1^{\sf T} v_j = 1$.
Let $z$ be any point on this line segment other than $y_0$ and $y_d$.
By Proposition \ref{pr:polygon}, $z$ is a fixed point of $f^{\cS}$.
Moreover, there exists a neighbourhood of $z$ that follows the sequence $\cS$
under the next $n$ iterations of (\ref{eq:f}).
Within this neighbourhood, the hyperplane that intersects $z$ and is orthogonal to $u_0^{\sf T}$ is invariant.
If all the eigenvalues of $M_{\cS}$, other than the unit eigenvalue,
have modulus less than $1$ (i.e.~$\sigma < 1$, see (\ref{eq:sigma})),
then within this neighbourhood the hyperplane is the stable manifold of $z$ for the map $f^{\cS}$.
In summary, iterates of $f^{\cS}$ approach the line segment connecting $y_0$ to $y_d$
(which has direction $v_0$) on a hyperplane orthogonal to $u_0^{\sf T}$.
The remaining eigenvectors $u_j^{\sf T}$ and $v_j$ can be interpreted similarly.

The next result indicates how the eigenvectors are related to one another algebraically.

\begin{lemma}
We have
\begin{align}
v_{\ell d} &= \frac{t_d}{t_{(\ell+1)d}} M_{\cX} v_0 \;, &
u_{\ell d}^{\sf T} &= \frac{t_{(\ell+1)d}}{t_d} u_0^{\sf T} M_{\cY} \;,
\label{eq:vuelld} \\
v_0 &= \frac{t_{(\ell+1)d}}{t_d} M_{\cY} v_{\ell d} \;, &
u_0^{\sf T} &= \frac{t_d}{t_{(\ell+1)d}} u_{\ell d}^{\sf T} M_{\cX} \;,
\label{eq:vu0} \\
v_{(\ell-1)d} &= \frac{t_{-d}}{t_{(\ell-1)d}} M_{\cX^{\overline{0}}} v_{-d} \;, &
u_{(\ell-1)d}^{\sf T} &= \frac{t_{(\ell-1)d}}{t_{-d}} u_{-d}^{\sf T} M_{\cY^{\overline{0}}} \;,
\label{eq:vuellm1d} \\
v_{-d} &= \frac{t_{(\ell-1)d}}{t_{-d}} M_{\cY^{\overline{0}}} v_{(\ell-1)d} \;, &
u_{-d}^{\sf T} &= \frac{t_{-d}}{t_{(\ell-1)d}} u_{(\ell-1)d}^{\sf T} M_{\cX^{\overline{0}}} \;.
\label{eq:vumd}
\end{align}
Furthermore\removableFootnote{
From this result we can immediately derive formulas such as
\begin{equation}
u_{-d}^{\sf T} M_{\hat{\cX}} v_{-d} = u_0^{\sf T} M_{\hat{\cX}} v_0 \;,
\end{equation}
but I have chosen to include such formulas only within proofs below as needed,
rather than in a separate Lemma.
},
\begin{align}
v_{-d} &= -\frac{t_d}{t_{-d}} M_{\hat{\cX}} v_0 \;, &
u_{-d}^{\sf T} &= -\frac{t_{-d}}{t_d} u_0^{\sf T} M_{\hat{\cY}} \;,
\label{eq:vumd2} \\
v_0 &= -\frac{t_{-d}}{t_d} M_{\hat{\cY}} v_{-d} \;, &
u_0^{\sf T} &= -\frac{t_d}{t_{-d}} u_{-d}^{\sf T} M_{\hat{\cX}} \;,
\label{eq:vu02} \\
v_{(\ell-1)d} &= -\frac{t_{(\ell+1)d}}{t_{(\ell-1)d}} M_{\check{\cX}} v_{\ell d} \;, &
u_{(\ell-1)d}^{\sf T} &= -\frac{t_{(\ell-1)d}}{t_{(\ell+1)d}} u_{\ell d}^{\sf T} M_{\check{\cY}} \;,
\label{eq:vuellm1d2} \\
v_{\ell d} &= -\frac{t_{(\ell-1)d}}{t_{(\ell+1)d}} M_{\check{\cY}} v_{(\ell-1)d} \;, &
u_{\ell d}^{\sf T} &= -\frac{t_{(\ell+1)d}}{t_{(\ell-1)d}} u_{(\ell-1)d}^{\sf T} M_{\check{\cX}} \;.
\label{eq:vuelld2}
\end{align}
\label{le:vuall}
\end{lemma}

\begin{proof}
By Proposition \ref{pr:MSPSsingular},
$f^{\cX}(y_0) = y_{\ell d}$ and
$f^{\cX}(y_d) = y_{(\ell+1)d}$.
Therefore
\begin{equation}
M_{\cX} v_0 = \frac{1}{t_d} M_{\cX} \left( y_d - y_0 \right) =
\frac{1}{t_d} \left( y_{(\ell+1)d} - y_{\ell d} \right) =
\frac{t_{(\ell+1)d}}{t_d} v_{\ell d} \;,
\end{equation}
which verifies the first part of (\ref{eq:vuelld}).
The first parts of the remaining equations can be derived in the same fashion.

By (\ref{eq:partitions1}) and (\ref{eq:partitions2}),
\begin{equation}
M_{\cY} M_{\cS^{(\ell d)}} =
M_{\cY} M_{\cX} M_{\cY} =
M_{\cS} M_{\cY} \;.
\end{equation}
Therefore
\begin{equation}
u_0^{\sf T} M_{\cY} M_{\cS^{(\ell d)}} =
u_0^{\sf T} M_{\cS} M_{\cY} = 
u_0^{\sf T} M_{\cY} \;,
\end{equation}
i.e.~$u_0^{\sf T} M_{\cY}$ is a left eigenvector of $M_{\cS^{(\ell d)}}$
corresponding to the eigenvalue $1$,
and therefore is a multiple of $u_{\ell d}$.
Also, by using the first part of (\ref{eq:vu0}) we obtain
\begin{equation}
\frac{t_{(\ell+1)d}}{t_d} u_0^{\sf T} M_{\cY} v_{\ell d} =
u_0^{\sf T} v_0 = 1 \;.
\end{equation}
Therefore $\frac{t_{(\ell+1)d}}{t_d} u_0^{\sf T} M_{\cY}$
has the same magnitude and direction as $u_{\ell d}$.
This verifies the second part of (\ref{eq:vuelld}),
and second parts of the remaining equations can be demonstrated similarly.
\end{proof}

\subsection{A basic unfolding of shrinking points}
\label{sub:unfolding}

The behaviour of $\cF[\ell,m,n]$-cycles and
$\cF[\ell \pm 1,m,n]$-cycles near an $\cS$-shrinking point,
where $\cS = \cF[\ell,m,n]$, was summarised in \S\ref{sub:definitions}.
In this section we review this behaviour more carefully.

We assume $\xi = (\xi_1,\xi_2) \in \mathbb{R}^2$, for simplicity,
let $\xi^*$ be an $\cS$-shrinking point,
and introduce local $(\eta,\nu)$-coordinates (\ref{eq:etanu}).
The condition $\det(J) \ne 0$, where $J$ is given by (\ref{eq:J}),
ensures that the coordinate change $(\xi_1,\xi_2) \leftrightarrow (\eta,\nu)$ is invertible.

The following result specifies curves of border-collision bifurcations,
$\eta = \psi_1(\nu)$ and $\nu = \psi_2(\eta)$,
along which $\cF[\ell,m,n]$ and $\cF[\ell+1,m,n]$-cycles coincide.
The subsequent result provides a useful expression for $\det \left( I - M_{\cS} \right)$.
Both results are proved in \cite{SiMe09,Si10},
except that expressions for the coefficients in terms of the $t_i$ are derived in \cite{SiMe10}.

\begin{lemma}
Suppose (\ref{eq:f}) with $K \ge 2$ has an $\cS$-shrinking point
at $\xi = \xi^*$ and $\det(J) \ne 0$.
Then, in a neighbourhood of $\xi = \xi^*$,
\begin{enumerate}
\item
there exists a unique $C^K$\removableFootnote{
Take care to note that a different definition of $K$ is used in \cite{SiMe09,Si10}.
}
function $\psi_1 : \mathbb{R} \to \mathbb{R}$, with
\begin{equation}
\psi_1(\nu) = -\frac{t_d}{t_{(\ell-1)d} t_{(\ell+1)d}} \nu^2 + \co \left( \nu^2 \right) \;,
\label{eq:phi1}
\end{equation}
such that $s^{\cS^{\overline{\ell d}}}_{\ell d} = 0$ on the locus $\eta = \psi_1(\nu)$,
\item
there exists a unique $C^K$ function $\psi_2 : \mathbb{R} \to \mathbb{R}$, with
\begin{equation}
\psi_2(\eta) = -\frac{t_{(\ell-1)d}}{t_d t_{-d}} \eta^2 + \co \left( \eta^2 \right) \;,
\label{eq:phi2}
\end{equation}
such that $s^{\cS^{\overline{\ell d}}}_0 = 0$ on the locus $\nu = \psi_2(\eta)$.
\end{enumerate}
\label{le:psi12}
\end{lemma}

\begin{lemma}
Suppose (\ref{eq:f}) with $K \ge 2$ has an $\cS$-shrinking point
at $\xi = \xi^*$ and $\det(J) \ne 0$.
Then
\begin{equation}
\det(I-M_{\cS}) = \frac{a}{t_d} \eta + \frac{a}{t_{(\ell-1)d}} \nu + \cO \left( (\eta,\nu)^2 \right) \;,
\label{eq:detImMS}
\end{equation}
where $a = \det \left( I-M_{\cS^{\overline{0}}} \right) \big|_{\xi = \xi^*}$.
\end{lemma}

The next result identifies regions, $\Psi_1$ and $\Psi_2$,
within which $\cF[\ell-1,m,n]$, $\cF[\ell,m,n]$ and $\cF[\ell+1,m,n]$-cycles reside,
and represents the basic unfolding of a shrinking point.
The reader is referred to \cite{SiMe09,Si10} for a proof\removableFootnote{
Theorem \ref{th:basicUnfolding} is proved in \cite{SiMe09}, although we seem have omitted the $t_i \ne 0$ requirement.
The analogous theorem in \cite{SiMe10} requires that $M_{\cS}$ does not have any other eigenvalues with unit modulus.
This was imposed such that we could prove the existence of saddle-node bifurcations, and is not required here.
}.
Fig.~\ref{fig:shrPointUnfolding} summarises the unfolding.
If $\sigma < 1$, then some of the periodic solutions are stable, see Table \ref{tb:stability},
but note that Theorem \ref{th:basicUnfolding} does not concern stability
and holds for any value of $\sigma$.

\begin{theorem}
Suppose (\ref{eq:f}) with $K \ge 2$ has an $\cS$-shrinking point
at $\xi = \xi^*$ and $\det(J) \ne 0$.
Let $\Psi_1 = \left\{ (\eta,\nu) ~\big|~ \eta,\nu \ge 0 \right\}$
and $\Psi_2 = \left\{ (\eta,\nu) ~\big|~ \eta \le \psi_1(\nu) ,\, \nu \le \psi_2(\eta) \right\}$,
where $\psi_1$ and $\psi_2$ are specified by Lemma \ref{le:psi12}.
Then there exists a neighbourhood $\cN$ of $(\eta,\nu) = (0,0)$,
such that (\ref{eq:f}) has unique admissible
$\cF[\ell,m,n]$ and $\cF[\ell-1,m,n]$ cycles in $\Psi_1 \cap \cN \setminus \{(0,0)\}$
and (\ref{eq:f}) has admissible
$\cF[\ell,m,n]$ and $\cF[\ell+1,m,n]$ cycles in $\Psi_2 \cap \cN \setminus \{(0,0)\}$.
\label{th:basicUnfolding}
\end{theorem}

\begin{figure}[b!]
\begin{center}
\setlength{\unitlength}{1cm}
\begin{picture}(16,12)
\put(0,0){\includegraphics[height=12cm]{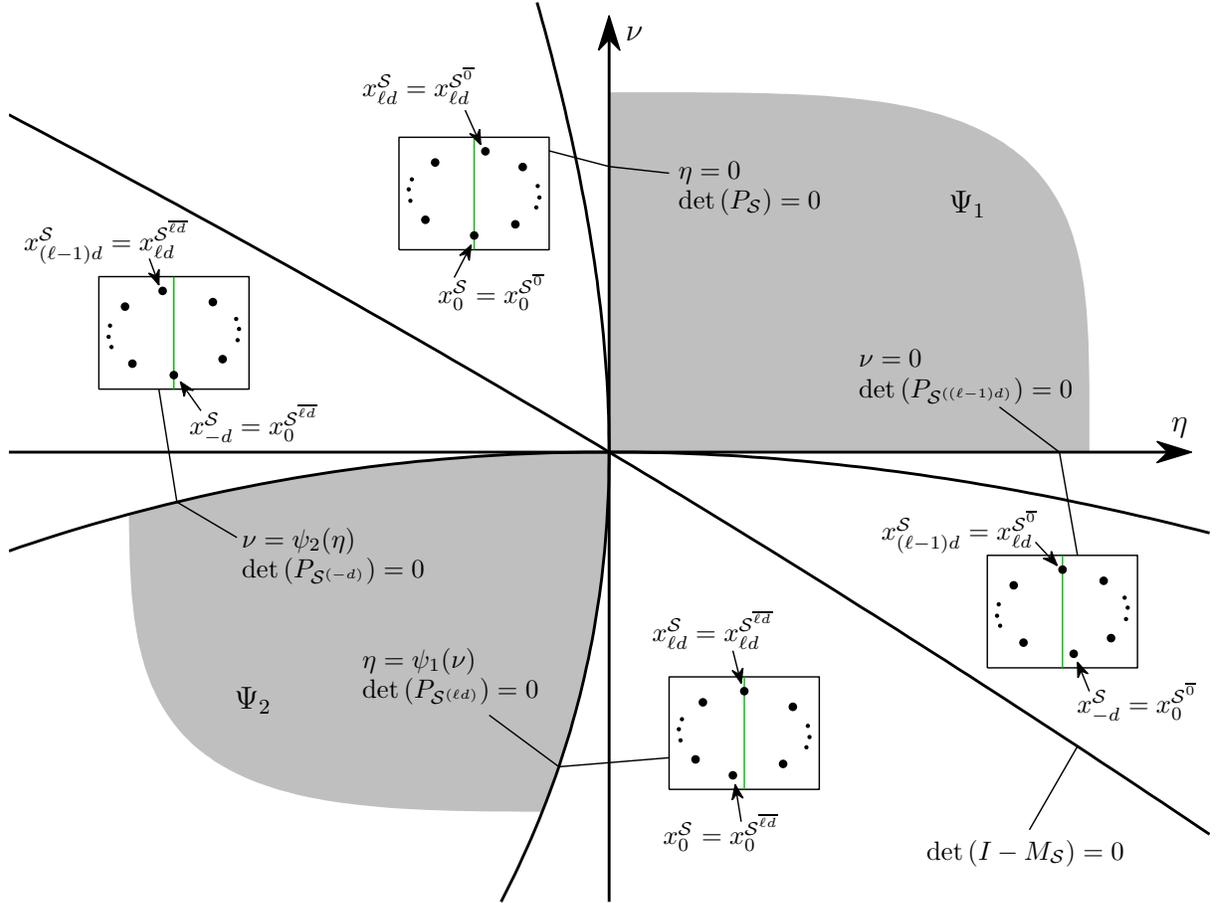}}
\put(15.45,6.3){$\eta$}
\put(8.2,11.5){$\nu$}
\put(12.5,9.2){$\Psi_1$}
\put(3,2.6){$\Psi_2$}
\put(8.9,9.25){\footnotesize $\det \left( P_{\cS} \right) = 0$}
\put(8.9,9.65){\footnotesize $\eta = 0$}
\put(11.3,6.75){\footnotesize $\det \left( P_{\cS^{((\ell-1)d)}} \right) = 0$}
\put(11.3,7.15){\footnotesize $\nu = 0$}
\put(4.7,2.75){\footnotesize $\det \left( P_{\cS^{(\ell d)}} \right) = 0$}
\put(4.7,3.15){\footnotesize $\eta = \psi_1(\nu)$}
\put(3.1,4.35){\footnotesize $\det \left( P_{\cS^{(-d)}} \right) = 0$}
\put(3.1,4.75){\footnotesize $\nu = \psi_2(\eta)$}
\put(12.2,.6){\footnotesize $\det \left( I - M_{\cS} \right) = 0$}
\put(5.7,8){\footnotesize $x^{\cS}_0 = x^{\cS^{\overline{0}}}_0$}
\put(4.7,10.75){\footnotesize $x^{\cS}_{\ell d} = x^{\cS^{\overline{0}}}_{\ell d}$}
\put(14.2,2.55){\footnotesize $x^{\cS}_{-d} = x^{\cS^{\overline{0}}}_0$}
\put(11.6,4.85){\footnotesize $x^{\cS}_{(\ell-1)d} = x^{\cS^{\overline{0}}}_{\ell d}$}
\put(8.7,.8){\footnotesize $x^{\cS}_0 = x^{\cS^{\overline{\ell d}}}_0$}
\put(8.55,3.5){\footnotesize $x^{\cS}_{\ell d} = x^{\cS^{\overline{\ell d}}}_{\ell d}$}
\put(2.4,6.25){\footnotesize $x^{\cS}_{-d} = x^{\cS^{\overline{\ell d}}}_0$}
\put(.2,8.7){\footnotesize $x^{\cS}_{(\ell-1)d} = x^{\cS^{\overline{\ell d}}}_{\ell d}$}
\end{picture}
\caption{
The basic unfolding of a shrinking point as specified by Theorem \ref{th:basicUnfolding}.
In $(\eta,\nu)$-coordinates, the shrinking point is located at the origin,
the positive axes are border-collision bifurcation curves that bound the region $\Psi_1$,
and $\eta = \psi_1(\nu)$ and $\nu = \psi_2(\eta)$
are border-collision bifurcation curves that bound $\Psi_2$.
We have also included sketches of the $\cS$-cycle in relation to the switching manifold
at a typical point on each of the four boundaries.
\label{fig:shrPointUnfolding}
}
\end{center}
\end{figure}

\subsection{Further identities relating to shrinking points}
\label{sub:furtherProperties}

We conclude this section by deriving additional algebraic expressions
regarding shrinking points that are used in later sections.

As implied by (\ref{eq:detImMS}),
$I - M_{\cS}$ is singular along a curve passing through the $\cS$-shrinking point.
At points where $I - M_{\cS}$ is non-singular,
the $\cS$-cycle is unique,
and the following result provides us with an asymptotic expression
for the location of the points of the $\cS$-cycle.

\begin{lemma}
Suppose (\ref{eq:f}) with $K \ge 2$ has an $\cS$-shrinking point
at $\xi = \xi^*$ and $\det(J) \ne 0$.
Then for all $(\eta,\nu)$ for which $\det(I-M_{\cS}) \ne 0$, for all $i$,
\begin{equation}
x^{\cS}_i = \frac{\frac{y_{i+d}}{t_d} \eta + \frac{y_i}{t_{(\ell-1)d}} \nu + \cO \left( (\eta,\nu)^2 \right)}
{\frac{1}{t_d} \eta + \frac{1}{t_{(\ell-1)d}} \nu + \cO \left( (\eta,\nu)^2 \right)} \;.
\label{eq:xSi}
\end{equation}
\label{le:xSi}
\end{lemma}

\begin{proof}
By (\ref{eq:xSiGen}),
$x^{\cS}_i = \frac{{\rm adj} \left( I-M_{\cS^{(i)}} \right) P_{\cS^{(i)}} B \mu}{\det(I-M_{\cS})}$.
Therefore we can write
\begin{equation}
x^{\cS}_i(\eta,\nu) = \frac{C_0 + C_1 \eta + C_2 \nu + \cO \left( (\eta,\nu)^2 \right)}
{\frac{a}{t_d} \eta + \frac{a}{t_{(\ell-1)d}} \nu + \cO \left( (\eta,\nu)^2 \right)} \;,
\label{eq:xSiProof1}
\end{equation}
for some $C_0, C_1, C_2 \in \mathbb{R}^N$.
When $\eta = 0$, $x^{\cS}_i = x^{\cS^{\overline{0}}}_i$,
thus $x^{\cS}_i(0,\nu) = y_i + \cO(\eta,\nu)$,
hence $C_0 = 0$ and $C_2 = \frac{a}{t_{(\ell-1)d}} y_i$.
Similarly when $\nu = 0$, $x^{\cS}_i = x^{\cS^{\overline{0}}}_{i+d}$,
thus $x^{\cS}_i(\eta,0) = y_{i+d} + \cO(\eta,\nu)$,
hence $C_1 = \frac{a}{t_{d}} y_{i+d}$.
By substituting these expressions for $C_1$ and $C_2$ into (\ref{eq:xSiProof1})
and cancelling instances of $a$, we obtain (\ref{eq:xSi}) as required.
\end{proof}

At the shrinking point, $M_{\cS}$ has a unit eigenvalue
and so near the shrinking point $M_{\cS}$ has an eigenvalue near $1$.
Throughout this paper this eigenvalue is denoted by $\lambda$.
Locally $\lambda$ is $C^K$ function of $\eta$ and $\nu$\removableFootnote{
The eigenvalue $\lambda$ is a root of the characteristic polynomial.
The coefficients of the polynomial are $C^K$ because $M_{\cS}$ is $C^K$.
The implicit function theorem gives us $\lambda$ and tells us that it is $C^K$.
}
because the algebraic multiplicity
of the unit eigenvalue at the shrinking point is one, Lemma \ref{le:rankNm1}.

\begin{lemma}
Suppose (\ref{eq:f}) with $K \ge 2$ has an $\cS$-shrinking point
at $\xi = \xi^*$ and $\det(J) \ne 0$.
Then
\begin{equation}
\lambda = 1 - \frac{a}{c t_d} \eta - \frac{a}{c t_{(\ell-1)d}} \nu + \cO \left( (\eta,\nu)^2 \right) \;,
\label{eq:lambda}
\end{equation}
where $c$ is the product of the nonzero eigenvalues of $I-M_{\cS}$ at $\xi = \xi^*$, (\ref{eq:c}).
\label{le:lambda}
\end{lemma}

\begin{proof}
Write $\lambda = 1 + k_1 \eta + k_2 \nu + \cO(2)$, for some $k_1, k_2 \in \mathbb{R}$.
Then 
\begin{align}
\det(I-M_{\cS}) &= (1-\lambda) \left( c + \cO(\eta,\nu) \right) \nonumber \\
&= \left( -k_1 \eta - k_2 \nu + \cO \left( (\eta,\nu)^2 \right) \right)
\left( c + \cO(\eta,\nu) \right) \nonumber \\
&= -k_1 c \eta - k_2 c \nu + \cO \left( (\eta,\nu)^2 \right) \;.
\label{eq:lambdaProof1}
\end{align}
By matching (\ref{eq:detImMS}) and (\ref{eq:lambdaProof1}) we obtain (\ref{eq:lambda}) as required.
\end{proof}

The last two results provide identities that connect
various quantities associated with a shrinking point.
For a proof of Lemma \ref{le:fourtIdentity}, refer to \cite{SiMe10}.
The proof involves expanding
$s^{\cS^{\overline{0}}}_i$ and $s^{\cS^{\overline{\ell d}}}_i$
in terms of $\eta$ and $\nu$ (for certain values of $i$)
and matching coefficients.
This assumes $\det(J) \ne 0$, but we expect that
Lemmas \ref{le:fourtIdentity} and \ref{le:uvIdentities} hold
regardless of how (\ref{eq:f}) varies with $\xi$ as the results 
concern properties of the shrinking point itself\removableFootnote{
I conjecture that by considering all possible $A_L$ and $A_R$,
there always exists a choice $\xi$ such that $\det(J) \ne 0$.
This would solve the problem, but seems extremely difficult to prove.
}.

\begin{lemma}
Suppose (\ref{eq:f}) with $K \ge 2$ has an $\cS$-shrinking point
at $\xi = \xi^*$ and $\det(J) \ne 0$.
Then
\begin{equation}
\frac{a}{b} = -\frac{t_d t_{(\ell-1)d}}{t_{-d} t_{(\ell+1)d}} \;.
\label{eq:fourtIdentity}
\end{equation}
\label{le:fourtIdentity}
\end{lemma}

\begin{lemma}
Suppose (\ref{eq:f}) with $K \ge 2$ has an $\cS$-shrinking point
at $\xi = \xi^*$ and $\det(J) \ne 0$.
Then, repeating (\ref{eq:uvIdentity12early}),
\begin{align}
\frac{u_0^{\sf T} v_{-d}}{a} +
\frac{u_{\ell d}^{\sf T} v_{(\ell-1)d}}{b} &= \frac{1}{c} \;,
\label{eq:uvIdentity1} \\
\frac{u_{(\ell-1)d}^{\sf T} v_{\ell d}}{a} +
\frac{u_{-d}^{\sf T} v_0}{b} &= \frac{1}{c} \;.
\label{eq:uvIdentity2} 
\end{align}
\label{le:uvIdentities}
\end{lemma}

\begin{proof}
Here we derive only (\ref{eq:uvIdentity1}).
Equation (\ref{eq:uvIdentity2}) may be derived similarly.

By (\ref{eq:vu0}), (\ref{eq:vumd}) and (\ref{eq:fourtIdentity}),
\begin{equation}
u_0^{\sf T} v_{-d} =
\left( \frac{t_d}{t_{(\ell+1)d}} u_{\ell d}^{\sf T} M_{\cX} \right)
\left( \frac{t_{(\ell-1)d}}{t_{-d}} M_{\cY^{\overline{0}}} v_{(\ell-1)d} \right) =
-\frac{a u_{\ell d} M_{\cX} M_{\cY^{\overline{0}}} v_{(\ell-1)d}}{b} \;.
\end{equation}
Therefore
\begin{equation}
\frac{u_0^{\sf T} v_{-d}}{a} +
\frac{u_{\ell d}^{\sf T} v_{(\ell-1)d}}{b} =
\frac{u_{\ell d}^{\sf T}
\left( I - M_{\cX} M_{\cY^{\overline{0}}} \right) v_{(\ell-1)d}}{b} \;.
\label{eq:uvIdentityProof2}
\end{equation}
By (\ref{eq:uvj}),
$u_{\ell d}^{\sf T} = \frac{e_1^{\sf T} {\rm adj} \left( I - M_{\cS^{(\ell d)}} \right)}{c}$,
and by (\ref{eq:partitions2}), $M_{\cS^{(\ell d)}} = M_{\cX} M_{\cY}$.
By (\ref{eq:MS}) and Lemma \ref{le:adjIdentity2},
$e_1^{\sf T} {\rm adj} \left( I - M_{\cX} M_{\cY} \right) =
e_1^{\sf T} {\rm adj} \left( I - M_{\cX} M_{\cY^{\overline{0}}} \right)$, thus
\begin{equation}
u_{\ell d}^{\sf T} = \frac{e_1^{\sf T} {\rm adj} \left( I - M_{\cX} M_{\cY^{\overline{0}}} \right)}{c} \;.
\label{eq:uvIdentityProof3}
\end{equation}
By substituting (\ref{eq:uvIdentityProof3}) into (\ref{eq:uvIdentityProof2})
and using (\ref{eq:adjIdentity}) we obtain
\begin{equation}
\frac{u_0^{\sf T} v_{-d}}{a} +
\frac{u_{\ell d}^{\sf T} v_{(\ell-1)d}}{b} =
\frac{e_1^{\sf T} v_{(\ell-1)d} \det \left( I - M_{\cX} M_{\cY^{\overline{0}}} \right)}{b c} \;.
\label{eq:uvIdentityProof4}
\end{equation}
Finally, $\det \left( I - M_{\cX} M_{\cY^{\overline{0}}} \right) = b$,
because $M_{\cX} M_{\cY^{\overline{0}}} =
M_{\cS^{(\ell d) \overline{0}}} =
M_{\cS^{\overline{\ell d} (\ell d)}} =
M_{\cF[\ell+1,m,n]^{(\ell d)}}$, by (\ref{eq:rsselld}).
Thus (\ref{eq:uvIdentityProof4}) reduces to (\ref{eq:uvIdentity1}),
because also $e_1^{\sf T} v_{(\ell-1)d} = 1$, by definition.
\end{proof}

\section{Locating nearby shrinking points}
\label{sec:locations}
\setcounter{equation}{0}

At an $\cS$-shrinking point, for each $j$, the line segment
connecting $y_j$ to $y_{j+d}$ consists of fixed points of $f^{\cS^{(j)}}$, see Fig.~\ref{fig:eigSchem}.
For parameter values near the shrinking point, the line segments persist 
as one-dimensional slow manifolds.
These are described in \S\ref{sub:slow}.
In \S\ref{sub:calculations} we then use these results to
determine the location of nearby $\cG^\pm[k,\chi]$-shrinking points to leading order.

\subsection{Slow manifolds}
\label{sub:slow}

In a neighbourhood of an $\cS$-shrinking point,
for each $j = 0, (\ell-1)d, \ell d, -d$,
we let $\omega_j^{\sf T}$ and $\zeta_j$
denote the left and right eigenvectors of $M_{\cS^{(j)}}$ corresponding to $\lambda$.
More specifically,
\begin{align}
M_{\cS^{(j)}} \zeta_j &= \lambda \zeta_j \;, &
e_1^{\sf T} \zeta_j &= 1 \;, \label{eq:zetaProperties} \\
\omega_j^{\sf T} M_{\cS^{(j)}} &= \lambda \omega_j^{\sf T} \;, &
\omega_j^{\sf T} \zeta_j &= 1 \;. \label{eq:omegaProperties}
\end{align}
Each $\omega_j^{\sf T}$ and $\zeta_j$ is a $C^K$ function\removableFootnote{
The matrix $M_{\cS^{(j)}}$ is $C^K$,
and $\lambda$ is $C^K$ (as noted above),
so the implicit function theorem can be used to obtain the eigenvectors,
subject to a length restriction, and show that they are $C^K$.
},
of $\eta$ and $\nu$.
Recall, $u_j^{\sf T}$ and $v_j$ denote the eigenvectors at the shrinking point,
see \S\ref{sub:definitions}, thus
\begin{equation}
\zeta_j(0,0) = v_j \;, \qquad
\omega_j^{\sf T}(0,0) = u_j^{\sf T} \;.
\label{eq:zetaomega00}
\end{equation}
The following result relates the eigenvectors to one another
based on the partitions of $\cS$ introduced in Definition \ref{df:XYall}.

\begin{lemma}
For any matrix $Q$,
\begin{align}
\omega_0^{\sf T} Q M_{\hat{\cX}} \zeta_0 &=
\omega_{-d}^{\sf T} M_{\hat{\cX}} Q \zeta_{-d} \;, \label{eq:omegaMXhatzeta} \\
\omega_0^{\sf T} M_{\hat{\cY}} Q \zeta_0 &=
\omega_{-d}^{\sf T} Q M_{\hat{\cY}} \zeta_{-d} \;, \label{eq:omegaMYhatzeta} \\
\omega_{\ell d}^{\sf T} Q M_{\check{\cX}} \zeta_{\ell d} &=
\omega_{(\ell-1)d}^{\sf T} M_{\check{\cX}} Q \zeta_{(\ell-1)d} \;, \label{eq:omegaMXcheckzeta} \\
\omega_{\ell d}^{\sf T} M_{\check{\cY}} Q \zeta_{\ell d} &=
\omega_{(\ell-1)d}^{\sf T} Q M_{\check{\cY}} \zeta_{(\ell-1)d} \;. \label{eq:omegaMYcheckzeta}
\end{align}
\label{le:omegaMzeta}
\end{lemma}

\begin{proof}
Here we derive (\ref{eq:omegaMXhatzeta}).
The remaining identities can be derived in the same fashion.

Since $M_{\cS} = M_{\hat{\cY}} M_{\hat{\cX}}$
and $M_{\cS^{(-d)}} = M_{\hat{\cX}} M_{\hat{\cY}}$,
refer to (\ref{eq:MS}), (\ref{eq:partitions1}) and (\ref{eq:partitions3}),
by (\ref{eq:zetaProperties}) we have
\begin{equation}
\zeta_0 = k_1 M_{\hat{\cY}} \zeta_{-d} \;,
\label{eq:omegaMXhatzetaProof1}
\end{equation}
for some $k_1 \in \mathbb{R}$.
Similarly
\begin{equation}
\omega_0^{\sf T} = k_2 \omega_{-d}^{\sf T} M_{\hat{\cX}} \;,
\label{eq:omegaMXhatzetaProof2}
\end{equation}
for some $k_2 \in \mathbb{R}$.
By using (\ref{eq:zetaProperties})-(\ref{eq:omegaProperties}), we then deduce
\begin{equation}
1 = \omega_0^{\sf T} \zeta_0 =
k_1 k_2 \omega_{-d}^{\sf T} M_{\hat{\cX}} M_{\hat{\cY}} \zeta_{-d} =
k_1 k_2 \omega_{-d}^{\sf T} \lambda \zeta_{-d} =
\lambda k_1 k_2 \;.
\label{eq:omegaMXhatzetaProof3}
\end{equation}
Finally, by combining (\ref{eq:omegaMXhatzetaProof1})-(\ref{eq:omegaMXhatzetaProof3}) we obtain
\begin{equation}
\omega_0^{\sf T} Q M_{\hat{\cX}} \zeta_0 =
k_1 k_2 \omega_{-d}^{\sf T} M_{\hat{\cX}} Q M_{\hat{\cX}} M_{\hat{\cY}} \zeta_{-d} =
k_1 k_2 \omega_{-d}^{\sf T} M_{\hat{\cX}} Q \lambda \zeta_{-d} =
\omega_{-d}^{\sf T} M_{\hat{\cX}} Q \zeta_{-d} \;,
\end{equation}
as required.
\end{proof}

Let us now consider the dynamics of $f^{\cS^{(j)}}$
(for any $j = 0, (\ell-1)d, \ell d, -d$).
If $\det \left( I - M_{\cS} \right) \ne 0$,
then $f^{\cS^{(j)}}$ has a unique fixed point, $x^{\cS}_j$.
The line intersecting $x^{\cS}_j$ and of direction $\zeta_j$
is a slow manifold on which the dynamics of $f^{\cS^{(j)}}$ is dictated by the value of $\lambda$.

Since $x^{\cS}_j$ is sensitive to changes in $\eta$ and $\nu$, see (\ref{eq:xSi}),
instead of $x^{\cS}_j$
it more helpful to use the intersection of the slow manifold with the switching manifold,
call it $\varphi_j$, as a reference point about which we can perform calculations.
If $\det \left( I - M_{\cS} \right) \ne 0$, then this intersection point is given by
\begin{equation}
\varphi_j = \left( I - \zeta_j e_1^{\sf T} \right) x^{\cS}_j \;.
\label{eq:varphi}
\end{equation}
The utility of $\varphi_j$ lies in the fact that
it is well-defined even when $\det \left( I - M_{\cS} \right) = 0$, see Lemma \ref{le:slow}.

Any point on the slow manifold can be written as $\varphi_j + h \zeta_j$,
where $h \in \mathbb{R}$ is the first component of this point
(since $e_1^{\sf T} \varphi_j = 0$ and $e_1^{\sf T} \zeta_j = 1$).
If $\det \left( I - M_{\cS} \right) \ne 0$,
then, since $x^{\cS}_j$ is a fixed point of $f^{\cS^{(j)}}$
and $M_{\cS^{(j)}} \zeta_j = \lambda \zeta_j$, we have
\begin{equation}
f^{\cS^{(j)}} \left( \varphi_j + h \zeta_j \right) =
\varphi_j + \left( h \lambda + \gamma_j \right) \zeta_j \;,
\label{eq:slowDyns}
\end{equation}
where,
\begin{equation}
\gamma_j = (1-\lambda) e_1^{\sf T} x^{\cS}_j \;.
\label{eq:gamma}
\end{equation}
Equation (\ref{eq:slowDyns}) describes the dynamics on the slow manifold, see Fig.~\ref{fig:slowManSchem}.
The next result justifies our use of $\varphi_j$ and $\gamma_j$
when $\det \left( I - M_{\cS} \right) = 0$.

\begin{figure}[b!]
\begin{center}
\setlength{\unitlength}{1cm}
\begin{picture}(12,9)
\put(0,0){\includegraphics[height=9cm]{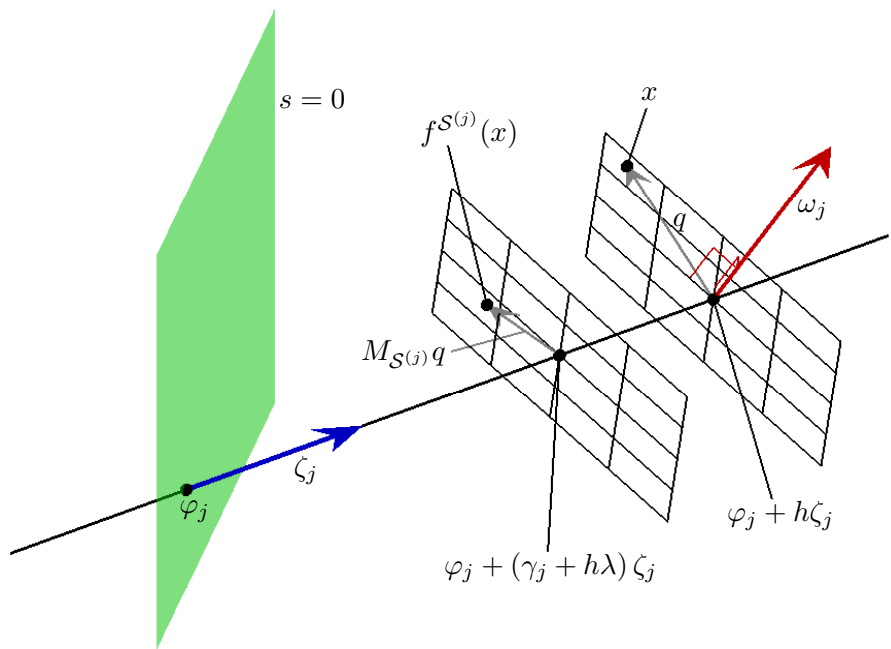}}
\put(3.7,7.6){\small $s=0$}
\put(2.38,2.2){\small $\varphi_j$}
\put(3.9,2.7){\small $\zeta_j$}
\put(5.9,1.4){\small $\varphi_j + \left( \gamma_j + h \lambda \right) \zeta_j$}
\put(9.66,2.1){\small $\varphi_j + h \zeta_j$}
\put(8.52,7.7){\small $x$}
\put(5.58,7.16){\small $f^{\cS^{(j)}}(x)$}
\put(10.6,6.22){\small $\omega_j$}
\put(8.94,6.02){\small $q$}
\put(4.8,4.24){\small $M_{\cS^{(j)}} q$}
\end{picture}
\caption{
An illustration of dynamics near the slow manifold of $f^{\cS^{(j)}}$,
for any $j = 0, (\ell-1)d, \ell d, -d$,
for parameter values near an $\cS$-shrinking point.
The slow manifold is a line of direction $\zeta_j$
(the slow eigenvector (\ref{eq:zetaProperties}))
that intersects the switching manifold at $\varphi_j$
In order to study images of $x$ under $f^{\cS^{(j)}}$,
it is helpful to decompose $x$ using the eigenspaces of $M_{\cS^{(j)}}$,
specifically (\ref{eq:xEigCoords}).
\label{fig:slowManSchem}
}
\end{center}
\end{figure}

\begin{lemma}
Suppose (\ref{eq:f}) with $K \ge 2$ has an $\cS$-shrinking point
at $\xi = \xi^*$ and $\det(J) \ne 0$.
Then there exists a neighbourhood $\cN$ of $(\eta,\nu) = (0,0)$,
such that for all $j \in \left\{ 0, (\ell-1)d, \ell d, -d \right\}$,
there exists unique $C^K$
functions $\varphi_j : \cN \to \mathbb{R}^N$ and $\gamma_j : \cN \to \mathbb{R}$
with $e_1^{\sf T} \varphi_j = 0$, $\gamma_j(0,0) = 0$, and
\begin{equation}
\varphi_0(0,0) = \varphi_{-d}(0,0) = y_0 \;, \qquad
\varphi_{(\ell-1)d}(0,0) = \varphi_{\ell d}(0,0) = y_{\ell d} \;,
\label{eq:varphi00}
\end{equation}
such that (\ref{eq:slowDyns}) is satisfied for all $h \in \mathbb{R}$ and all $(\eta,\nu) \in \cN$.
Moreover, (\ref{eq:varphi}) and (\ref{eq:gamma}) are satisfied
for all $(\eta,\nu) \in \cN$ for which $\det \left( I - M_{\cS} \right) \ne 0$.
\label{le:slow}
\end{lemma}

\begin{proof}
Consider the matrix equation
\begin{equation}
\left( \left( I - M_{\cS^{(j)}} \right) \left( I - e_1 e_1^{\sf T} \right) +
\zeta_j e_1^{\sf T} \right) \phi_j = P_{\cS^{(j)}} B \mu \;,
\label{eq:slowProof1}
\end{equation}
where we wish to solve for the unknown vector $\phi_j$.
Here we show that (\ref{eq:slowProof1}) defines $\phi_j$ uniquely,
and that the desired quantities $\varphi_j$ and $\gamma_j$ are given by
\begin{equation}
\varphi_j = \left( I - e_1 e_1^{\sf T} \right) \phi_j \;, \qquad
\gamma_j = e_1^{\sf T} \phi_j \;.
\label{eq:slowProof11}
\end{equation}
To show that (\ref{eq:slowProof1}) has a unique solution,
we apply Lemma \ref{le:matrixDeterminant} to write
\begin{align}
\det \left( \left( I - M_{\cS^{(j)}} \right) \left( I - e_1 e_1^{\sf T} \right) +
\zeta_j e_1^{\sf T} \right) &=
\det \left( I - M_{\cS^{(j)}} \right) \det \left( I - e_1 e_1^{\sf T} \right) \nonumber \\
&\quad+
e_1^{\sf T} {\rm adj} \left( I - e_1 e_1^{\sf T} \right)
{\rm adj} \left( I - M_{\cS^{(j)}} \right) \zeta_j \;.
\label{eq:slowProof2}
\end{align}
Since $\det \left( I - e_1 e_1^{\sf T} \right) = 0$
and ${\rm adj} \left( I - e_1 e_1^{\sf T} \right) = e_1 e_1^{\sf T}$,
(\ref{eq:slowProof2}) reduces to
\begin{equation}
\det \left( \left( I - M_{\cS^{(j)}} \right) \left( I - e_1 e_1^{\sf T} \right) +
\zeta_j e_1^{\sf T} \right) =
e_1^{\sf T} {\rm adj} \left( I - M_{\cS^{(j)}} \right) \zeta_j \;.
\end{equation}
At $(\eta,\nu) = (0,0)$, $e_1^{\sf T} {\rm adj} \left( I - M_{\cS^{(j)}} \right) = c u_j^{\sf T}$
and $\zeta_j = v_j$, see (\ref{eq:uvj}) and (\ref{eq:zetaomega00}).
Using also $u_j^{\sf T} v_j = 1$ we obtain
\begin{equation}
\det \left( \left( I - M_{\cS^{(j)}} \right) \left( I - e_1 e_1^{\sf T} \right) +
\zeta_j e_1^{\sf T} \right) = c + \cO(\eta,\nu) \;,
\end{equation}
which is nonzero in a neighbourhood of $(0,0)$.
Thus (\ref{eq:slowProof1}) has a unique solution $\phi_j$ in this neighbourhood.

Next we define $\varphi_j$ and $\gamma_j$ by (\ref{eq:slowProof11}).
These are $C^K$ functions of $\eta$ and $\nu$
because the components of (\ref{eq:slowProof1}) are $C^K$.
By (\ref{eq:slowProof11}), $e_1^{\sf T} \varphi_j = 0$. 
Also $\phi_j = \varphi_j + \gamma_j e_1$,
and by substituting this into (\ref{eq:slowProof1}) and simplifying we obtain
\begin{equation}
\left( I - M_{\cS^{(j)}} \right) \varphi_j + \gamma_j \zeta_j = P_{\cS^{(j)}} B \mu \;.
\label{eq:slowProof3}
\end{equation}
Thus for any $h \in \mathbb{R}$
\begin{equation}
M_{\cS^{(j)}} \left( \varphi_j + h \zeta_j \right) + P_{\cS^{(j)}} B \mu =
\varphi_j + \left( h \lambda + \gamma_j \right) \zeta_j \;,
\end{equation}
where we have used $M_{\cS^{(j)}} \zeta_j = \lambda_j \zeta_j$.
and therefore (\ref{eq:slowDyns}).
This shows that (\ref{eq:slowDyns}) is satisfied.

When $(\eta,\nu) = (0,0)$,
(\ref{eq:slowProof3}) is satisfied by $\gamma_j = 0$
and either $\varphi_j = y_0$ or $\varphi_j = y_{\ell d}$, as given in (\ref{eq:varphi00}).
This verifies (\ref{eq:varphi00}) and $\gamma_j(0,0) = 0$ because $\phi_j$ is unique.
Similarly, when $\det \left( I - M_{\cS} \right) \ne 0$,
(\ref{eq:slowProof3}) is satisfied by (\ref{eq:varphi}) and (\ref{eq:gamma}) because
\begin{equation}
\left( I - M_{\cS^{(j)}} \right) \left( I - \zeta_j e_1^{\sf T} \right) x^{\cS}_j + 
(1-\lambda) \zeta_j e_1^{\sf T} x^{\cS}_j =
\left( I - M_{\cS^{(j)}} \right) x^{\cS}_j \;,
\end{equation}
and therefore $\left( I - M_{\cS^{(j)}} \right) x^{\cS}_j = P_{\cS} B \mu$.
Again, because of the uniqueness of $\phi_j$,
(\ref{eq:varphi}) and (\ref{eq:gamma}) hold in a neighbourhood of $(\eta,\nu) = (0,0)$.
\end{proof}

For any $j$, and any $x \in \mathbb{R}^N$, it is helpful to write
\begin{equation}
x = \varphi_j + h \zeta_j + q \;,
\label{eq:xEigCoords}
\end{equation}
where $h \in \mathbb{R}$ and $q \in \mathbb{R}^N$ with $\omega_j^{\sf T} q = 0$.
The vector $h \zeta_j$ represents the component of $x - \varphi_j$ in the $\zeta_j$ direction.
The vector $q$ represents the component of $x - \varphi_j$ in the remaining eigendirections of $M_{\cS^{(j)}}$.
The decomposition (\ref{eq:xEigCoords}) is unique
and enables us to express iterates of $x$ under $f^{\cS^{(j)}}$ succinctly, Lemma \ref{le:slowDynsGen}.
This is illustrated in Fig.~\ref{fig:slowManSchem}
and formalised by the following result.
We omit a proof as it is a straight-forward application of eigenspace decomposition.

\begin{lemma}
For all $j \in \left\{ 0, (\ell-1)d, \ell d, -d \right\}$, and for all $x \in \mathbb{R}^N$,
there exists unique $h \in \mathbb{R}$ and $q \in \mathbb{R}^N$ with
$\omega_j^{\sf T} q = 0$ such that $x = \varphi_j + h \zeta_j + q$.
Moreover
\begin{equation}
h = \omega_j^{\sf T} \left( x - \varphi_j \right) \;, \qquad
q = \left( I - \zeta_j \omega_j^{\sf T} \right) \left( x - \varphi_j \right) \;,
\label{eq:hq}
\end{equation}
and for any $k \in \mathbb{Z}^+$,
\begin{equation}
f^{\left( \cS^{(j)} \right)^k}(x) = \varphi_j +
\left( \gamma_j \sum_{j=0}^{k-1} \lambda^j + h \lambda^k \right) \zeta_j +
M_{\cS^{(j)}}^k q \;.
\label{eq:slowDynsGen}
\end{equation}
\label{le:slowDynsGen}
\vspace{-5mm}								
\end{lemma}

The remaining results of this paper assume $\sigma < 1$,
where $\sigma$ denotes the maximum modulus of the eigenvalues of $M_{\cS}$,
excluding the unit eigenvalue, at the $\cS$-shrinking point, (\ref{eq:sigma}).
This assumption ensures iterates under $f^{\cS{(j)}}$ converge to the slow manifold,
which is central to the validity of our main results.

\begin{lemma}
If $\sigma < 1$, then $c > 0$\removableFootnote{
This relates to an observation that Feigin made for fixed points of (\ref{eq:f}).
Let $\sigma^+_{\cS}$ denote the number of real eigenvalues of $M_{\cS}$ that are greater than $1$.
If $1$ is not an eigenvalue of $M_{\cS}$,
then the sign of the product of the eigenvalues of $I - M_{\cS}$
is equal to $(-1)^{\sigma^+_{\cS}}$.
}.
\label{le:cPositive}
\end{lemma}

\begin{proof}
As in (\ref{eq:sigma}), let $\rho_1,\ldots,\rho_N$ be the eigenvalues of $M_{\cS}$,
counting multiplicity, and $\rho_1 = 1$.
By the definition of $c$ (\ref{eq:c}), $c = \prod_{i=2}^N (1 - \rho_i)$.
The assumption $\sigma < 1$ implies $|\rho_i| < 1$, for each $i \ne 1$.
If $\rho_i \in \mathbb{R}$, then $1 - \rho_i > 0$.
Any $\rho_i \notin \mathbb{R}$ appear in complex conjugate pairs
with $(1-\rho_i)(1-\overline{\rho}_i) > 0$.
Thus $\prod_{i=2}^N (1 - \rho_i)$ can be expressed as a product of positive numbers,
and so $c > 0$ as required.
\end{proof}

As described above, $q$ is a linear combination of the eigendirections of $M_{\cS}$ other than $\zeta_j$.
If $\sigma < 1$, then the corresponding eigenvalues all have modulus less than $1$,
in which case $\left\| M_{\cS^{(j)}}^k q \right\| \to 0$, as $k \to \infty$.
Moreover, we can write $M_{\cS^{(j)}}^k q = \cO \left( \sigma^k \right)$.
This is true for any $q$ of the form (\ref{eq:hq}),
thus we have the following result\removableFootnote{
Here is my proof:

For any $x \in \mathbb{R}^N$,
$\omega_j^{\sf T} \left( I - \zeta_j \omega_j^{\sf T} \right) x = 0$
because $\omega_j^{\sf T} \zeta_j = 1$.
That is, the vector $\left( I - \zeta_j \omega_j^{\sf T} \right) x$ 
is orthogonal to the left eigenvector of $M_{\cS^{(j)}}$ corresponding to the eigenvalue $\lambda$.
Therefore $\left( I - \zeta_j \omega_j^{\sf T} \right) x$
can be written as a linear combination of basis vectors of the right eigenspaces of $M_{\cS}$
not corresponding to $\lambda$.
Hence $M_{\cS^{(j)}}^k \left( I - \zeta_j \omega_j^{\sf T} \right) x$
can be written as this linear combination with terms scaled by the $k^{\rm th}$ power
of the corresponding eigenvalue, and therefore is order $\sigma^k$.
Since $x$ is arbitrary, this completes the proof.
}.

\begin{lemma}
If $\sigma < 1$, then for each $j \in \left\{ 0, (\ell-1)d, \ell d, -d \right\}$,
\begin{equation}
M_{\cS^{(j)}}^k \left( I - \zeta_j \omega_j^{\sf T} \right) = \cO \left( \sigma^k \right) \;.
\label{eq:ordersigmak}
\end{equation}
\label{le:ordersigmak}
\vspace{-5mm}								
\end{lemma}

\subsection{Calculations for nearby shrinking points}
\label{sub:calculations}

In this section we derive formulas for the border-collision bifurcation curves of $\cG^\pm[k,\chi]$-cycles.
The first result provides us with expressions for
$\det \left( \rho I - M_{\cG^\pm[k,\chi]} \right)$,
useful for large values of $k \in \mathbb{Z}^+$.

\begin{lemma}
Suppose (\ref{eq:f}) with $K \ge 2$ has an $\cS$-shrinking point
at $\xi = \xi^*$ and $\det(J) \ne 0$ and $\sigma < 1$.
Choose any $\chi_{\rm max} \in \mathbb{Z}^+$\removableFootnote{
The idea here is that we require $k - |\chi|$ to be order $k$.
},
$k \in \mathbb{Z}^+$, and $\rho \in \mathbb{C}$.
Then, in a neighbourhood of $(\eta,\nu) = (0,0)$\removableFootnote{
We do not evaluate the leading order terms at $(\eta,\nu) = (0,0)$
as this keeps the error term exponentially small,
and, more importantly, it is not clear how to deal with powers of $\lambda$ in such an evaluation.
},
\begin{align}
\det \left( \rho I - M_{\cG^+[k,\chi]} \right) &= \begin{cases}
\rho^N \left( 1 - \frac{\lambda^{k+\chi+1}}{\rho}
\omega_{\ell d}^{\sf T} \left( M_{\cX^{\overline{0}}} M_{\cY} \right)^{-\chi-1}
M_{\check{\cX}} \zeta_{\ell d} \right) +
\cO \left( \sigma^k \right) \;, &
-\chi_{\rm max} \le \chi \le -1 \\
\rho^N \left( 1 - \frac{\lambda^{k-\chi}}{\rho}
\omega_0^{\sf T} \left( M_{\cY^{\overline{0}}} M_{\cX} \right)^{\chi}
M_{\hat{\cX}} \zeta_0 \right) +
\cO \left( \sigma^k \right) \;, &
0 \le \chi \le \chi_{\rm max}
\end{cases}
\;, \label{eq:detImMGplus} \\
\det \left( \rho I - M_{\cG^-[k,\chi]} \right) &= \begin{cases}
\rho^N \left( 1 - \frac{\lambda^{k+\chi}}{\rho}
\omega_{-d}^{\sf T} \left( M_{\cY} M_{\cX^{\overline{0}}} \right)^{-\chi}
M_{\hat{\cY}} \zeta_{-d} \right) +
\cO \left( \sigma^k \right) \;, &
-\chi_{\rm max} \le \chi \le 0 \\
\rho^N \left( 1 - \frac{\lambda^{k-\chi+1}}{\rho}
\omega_{(\ell-1)d}^{\sf T} \left( M_{\cX} M_{\cY^{\overline{0}}} \right)^{\chi-1}
M_{\check{\cY}} \zeta_{(\ell-1) d} \right) +
\cO \left( \sigma^k \right) \;, &
1 \le \chi \le \chi_{\rm max}
\end{cases}
\;. \label{eq:detImMGminus}
\end{align}
\label{le:detAll}
\end{lemma}

\begin{proof}
Here we derive (\ref{eq:detImMGplus}) for $0 \le \chi \le \chi_{\rm max}$.
The remaining formulas can be derived similarly.

By Lemma \ref{le:eigMSindep},
$M_{\cG^+[k,\chi]}$ and $M_{\cG^+[k,\chi]^{(i)}}$ share the same eigenvalues and multiplicities, for any $i$.
The same is true for $\rho I - M_{\cG^+[k,\chi]}$ and $\rho I - M_{\cG^+[k,\chi]^{(i)}}$, thus, for any $i$,
\begin{equation}
\det \left( \rho I - M_{\cG^+[k,\chi]} \right) = 
\det \left( \rho I - M_{\cG^+[k,\chi]^{(i)}} \right) \;.
\label{eq:detImMGplusproof1}
\end{equation}

For any $0 \le \chi \le \chi_{\rm max}$,
$\cG^+[k,\chi] = \left( \cS^{\overline{\ell d}} \right)^{\chi} \cS^{k-\chi} \hat{\cX}$, (\ref{eq:Gplus}).
The word $\hat{\cX}$ has $n-d$ symbols (\ref{df:XYall}), thus
\begin{equation}
\cG^+[k,\chi]^{(d-n)} =
\hat{\cX} \left( \cS^{\overline{\ell d}} \right)^{\chi} \cS^{k-\chi} \;.
\label{eq:detImMGplusproof2}
\end{equation}
$\cG^+[k,\chi]^{(d-n)}$ is a particularly useful permutation of $\cG^+[k,\chi]$ for the purposes of this proof
because, by (\ref{eq:detImMGplusproof2}), it ends in a power involving $k$.
By (\ref{eq:detImMGplusproof2}), 
\begin{equation}
M_{\cG^+[k,\chi]^{(d-n)}} = M_{\cS}^{k-\chi} M_{\cS^{\overline{\ell d}}}^{\chi} M_{\hat{\cX}} \;.
\label{eq:detImMGplusproof3}
\end{equation}
Since $M_{\cS} \zeta_0 = \lambda \zeta_0$, by Lemma \ref{le:ordersigmak} with $j=0$ we can write
\begin{equation}
M_{\cS}^{k-\chi} =
M_{\cS}^{k-\chi} \zeta_0 \omega_0^{\sf T} + M_{\cS}^{k-\chi} \left( I - \zeta_0 \omega_0^{\sf T} \right) =
\lambda^{k-\chi} \zeta_0 \omega_0^{\sf T} + \cO \left( \sigma^k \right) \;,
\label{eq:detImMGplusproof4}
\end{equation}
with which (\ref{eq:detImMGplusproof3}) becomes
\begin{equation}
M_{\cG^+[k,\chi]^{(d-n)}} =
\lambda^{k-\chi} \zeta_0 \omega_0^{\sf T}
\left( M_{\cY^{\overline{0}}} M_{\cX} \right)^{\chi} M_{\hat{\cX}} +
\cO \left( \sigma^k \right) \;,
\label{eq:detImMGplusproof5}
\end{equation}
where we have also substituted $M_{\cS^{\overline{\ell d}}} = M_{\cY^{\overline{0}}} M_{\cX}$.
Finally, by using Lemma \ref{le:matrixDeterminant},
(\ref{eq:detImMGplusproof1}) with $i = n-d$, and (\ref{eq:detImMGplusproof5}),
we arrive at (\ref{eq:detImMGplus}) for $0 \le \chi \le \chi_{\rm max}$ as required.
\end{proof}

To motivate the next result,
recall that boundaries of $\cG^\pm_k$-mode-locking regions
are points where $s^{\cG^\pm[k,\chi]}_i = 0$, for certain values of $i$.
Each $s^{\cG^\pm[k,\chi]}_i$ can be evaluated using (\ref{eq:sFormulaEarly}).
The following result provides asymptotic expressions for the numerator of
(\ref{eq:sFormulaEarly}), applied to $s^{\cG^\pm[k,\chi]}_i$.
In view of Proposition \ref{pr:rssIdentities} and Lemma \ref{le:sZero},
different values of $\chi$ and $i$ can give the same boundary $s^{\cG^\pm[k,\chi]}_i = 0$.
This gives us some choice as to the values of $\chi$ and $i$ that we can use to describe a given boundary.
In Lemma \ref{le:sGall2} we choose the index $i$ that provides the simplest algebraic expression,
leading to four different cases, as indicated\removableFootnote{
Also
\begin{align}
\cG^+[k,\chi]^{\left( \tilde{\ell} d_k^+ \right)} &=
\left( \cY \cX^{\overline{0}} \right)^{-\chi - 1}
\check{\cX} \left( \cS^{((\ell-1)d)} \right)^{k+\chi+1} \;, \qquad
-k+1 \le \chi \le -1 \;, \\
\cG^+[k,\chi] &=
\left( \cX \cY^{\overline{0}} \right)^{\chi}
\hat{\cX} \left( \cS^{(-d)} \right)^{k - \chi} \;, \qquad
0 \le \chi \le k-1 \;, \\
\cG^-[k,\chi]^{(-d_k^-)} &=
\left( \cX^{\overline{0}} \cY \right)^{-\chi} \hat{\cY} \cS^{k+\chi} \;, \qquad
-k+1 \le \chi \le 0 \;, \\
\cG^-[k,\chi]^{\left( \left( \tilde{\ell}-1 \right) d_k^- \right)} &=
\left( \cY^{\overline{0}} \cX \right)^{\chi - 1} \check{\cY} \left( \cS^{(\ell d)} \right)^{k-\chi+1} \;, \qquad
1 \le \chi \le k-1 \;.
\end{align}
}.



\begin{lemma}
Suppose (\ref{eq:f}) with $K \ge 2$ has an $\cS$-shrinking point
at $\xi = \xi^*$ and $\det(J) \ne 0$ and $\sigma < 1$.
Choose any $\chi_{\rm max} \in \mathbb{Z}^+$ and $k \in \mathbb{Z}^+$.
Then, in a neighbourhood of $(\eta,\nu) = (0,0)$,
\begin{align}
\det \left( P_{\cG^+[k,\chi]^{\left( \tilde{\ell} d_k^+ \right)}} \right) \varrho^{\sf T} B \mu &=
\gamma_{(\ell-1)d} \sum_{j=0}^{k+\chi} \lambda^j +
\omega_{(\ell-1)d}^{\sf T} \left( f^{\left( \cY \cX^{\overline{0}} \right)^{-\chi-1} \check{\cX}}
\left( \varphi_{(\ell-1)d} \right) - \varphi_{(\ell-1)d} \right) \lambda^{k + \chi + 1} \nonumber \\
&\quad+ \cO \left( \sigma^k \right) \;, {\rm ~for~all~} -\chi_{\rm max} \le \chi \le -1 \;,
\label{eq:sGplusminus} \\
\det \left( P_{\cG^+[k,\chi]} \right) \varrho^{\sf T} B \mu &=
\gamma_{-d} \sum_{j=0}^{k-\chi-1} \lambda^j +
\omega_{-d}^{\sf T} \left( f^{\left( \cX \cY^{\overline{0}} \right)^{\chi} \hat{\cX}}
\left( \varphi_{-d} \right) - \varphi_{-d} \right) \lambda^{k - \chi} +
\cO \left( \sigma^k \right) \;, \nonumber \\
& \quad {\rm ~for~all~} 0 \le \chi \le \chi_{\rm max} \;,
\label{eq:sGplusplus} \\
\det \left( P_{\cG^-[k,\chi]^{\left( -d_k^- \right)}} \right) \varrho^{\sf T} B \mu &=
\gamma_0 \sum_{j=0}^{k-\chi-1} \lambda^j +
\omega_0^{\sf T} \left( f^{\left( \cX^{\overline{0}} \cY \right)^{-\chi} \hat{\cY}}
\left( \varphi_0 \right) - \varphi_0 \right) \lambda^{k + \chi} +
\cO \left( \sigma^k \right) \;, \nonumber \\
& \quad {\rm ~for~all~} -\chi_{\rm max} \le \chi \le 0 \;,
\label{eq:sGminusminus} \\
\det \left( P_{\cG^-[k,\chi]^{\left( \left( \tilde{\ell}-1 \right) d_k^- \right)}} \right) \varrho^{\sf T} B \mu &=
\gamma_{\ell d} \sum_{j=0}^{k-\chi} \lambda^j +
\omega_{\ell d}^{\sf T} \left( f^{\left( \cY^{\overline{0}} \cX \right)^{\chi - 1} \check{\cY}}
\left( \varphi_{\ell d} \right) - \varphi_{\ell d} \right) \lambda^{k - \chi + 1} +
\cO \left( \sigma^k \right) \;, \nonumber \\
& \quad {\rm ~for~all~} 1 \le \chi \le \chi_{\rm max} \;.
\label{eq:sGminusplus}
\end{align}
\label{le:sGall2}
\end{lemma}

\begin{proof}
Here we derive (\ref{eq:sGplusplus}).
The remaining expressions can be derived in the same fashion.

By (\ref{eq:partitions1}), (\ref{eq:partitions3}), (\ref{eq:XXcheck}) and (\ref{eq:YXhat}),
\begin{equation}
\cS \hat{\cX} = \cX \cY \hat{\cX} = \cX \check{\cX} \cY^{\overline{0}} =
\hat{\cX} \cX^{\overline{0}} \cY^{\overline{0}} = \hat{\cX} \cS^{(-d)} \;.
\label{eq:Gplusplus1}
\end{equation}
Thus, for $0 \le \chi \le \chi_{\rm max}$, (\ref{eq:Gplus}) can be rewritten as
\begin{equation}
\cG^+[k,\chi] = \left( \cS^{\overline{\ell d}} \right)^{\chi}
\hat{\cX} \left( \cS^{(-d)} \right)^{k - \chi} \;.
\label{eq:Gplusplus2}
\end{equation}
Then, by Lemma \ref{le:slowDynsGen} with $j = -d$, for any $x \in \mathbb{R}^N$,
\begin{align}
f^{\cG^+[k,\chi]}(x) &=
f^{\left( \cS^{(-d)} \right)^{k-\chi}}
\left( f^{\left( \cS^{\overline{\ell d}} \right)^{\chi} \hat{\cX}}(x) \right) \nonumber \\
&= f^{\left( \cS^{(-d)} \right)^{k-\chi}}
\left( M_{\left( \cS^{\overline{\ell d}} \right)^{\chi} \hat{\cX}} x +
P_{\left( \cS^{\overline{\ell d}} \right)^{\chi} \hat{\cX}} B \mu \right) \nonumber \\
&= \varphi_{-d} +
\left( \gamma_{-d} \sum_{j=0}^{k-\chi-1} \lambda^j +
\omega_{-d}^{\sf T} \left( M_{\left( \cS^{\overline{\ell d}} \right)^{\chi} \hat{\cX}} x +
P_{\left( \cS^{\overline{\ell d}} \right)^{\chi} \hat{\cX}} B \mu -
\varphi_{-d} \right) \lambda^{k-\chi} \right)
\zeta_{-d} \nonumber \\
&\quad+M_{\cS^{(-d)}}^{k-\chi} \left( I - \zeta_{-d} \omega_{-d}^{\sf T} \right)
\left( M_{\left( \cS^{\overline{\ell d}} \right)^{\chi} \hat{\cX}} x +
P_{\left( \cS^{\overline{\ell d}} \right)^{\chi} \hat{\cX}} B \mu -
\varphi_{-d} \right) \;.
\label{eq:sGplusplusproof1}
\end{align}

Now suppose $\det \left( I - M_{\cG^+[k,\chi]} \right) \ne 0$.
Then $x^{\cG^+[k,\chi]}_0$ is well-defined and is the unique fixed point of $f^{\cG^+[k,\chi]}$.
By Lemma \ref{le:slowDynsGen} we can uniquely write
\begin{equation}
x^{\cG^+[k,\chi]}_0 = \varphi_{-d} + h \zeta_{-d} + q \;,
\label{eq:xGplusForm}
\end{equation}
where $h \in \mathbb{R}$ and $\omega_{-d}^{\sf T} q = 0$.
By substituting (\ref{eq:xGplusForm}) for $x$ and $f^{\cG^+[k,\chi]}(x)$ in (\ref{eq:sGplusplusproof1}),
multiplying both sides of (\ref{eq:sGplusplusproof1}) by $e_1^{\sf T}$ on the left,
and applying Lemma \ref{le:ordersigmak}, we obtain\removableFootnote{
In Notebook I I solved a similar equation for $h$ and $q$ exactly
by using formulas for partitioned matrices which is possible because the equation is linear in $h$ and $q$.
This required a lot of effort but only culminated in formulas
that can be derived much more simply!
We have
\begin{equation}
x^{\cG^+[k,0]}_0 = \left( I - M_{\cG^+[k,0]} \right)^{-1} P_{\cG^+[k,0]} B \mu \;.
\end{equation}
From $\cG^+[k,0] = \hat{\cX} \left( \cS^{(-d)} \right)^k$, it follows that
\begin{equation}
P_{\cG^+[k,0]} = \left( I - M_{\cS^{(-d)}}^k \right)
\left( I - M_{\cS^{(-d)}} \right)^{-1} P_{\cS^{(-d)}} +
M_{\cS^{(-d)}}^k P_{\hat{\cX}} \;.
\end{equation}
Substituting this and
$x^{\cS}_{-d} = \left( I - M_{\cS^{(-d)}} \right)^{-1} P_{\cS^{(-d)}} B \mu$ yields
\begin{align}
x^{\cG^+[k,0]}_0 &= \left( I - M_{\cG^+[k,0]} \right)^{-1}
\left( \left( I - M_{\cS^{(-d)}}^k \right) x^{\cS}_{-d} +
M_{\cS^{(-d)}}^k P_{\hat{\cX}} B \mu \right) \nonumber \\
&= \left( I - M_{\cG^+[k,0]} \right)^{-1}
\left( x^{\cS}_{-d} - M_{\cS^{(-d)}}^k \left( x^{\cS}_{-d} - P_{\hat{\cX}} B \mu \right) \right)
\nonumber \\
&= \left( I - M_{\cG^+[k,0]} \right)^{-1}
\left( x^{\cS}_{-d} - M_{\cS^{(-d)}}^k M_{\hat{\cX}} x^{\cS}_0 \right) \nonumber \\
&= \left( I - M_{\cG^+[k,0]} \right)^{-1}
\left( x^{\cS}_{-d} - M_{\cG^+[k,0]} x^{\cS}_0 \right) \nonumber \\
&= x^{\cS}_0 + \left( I - M_{\cG^+[k,0]} \right)^{-1}
\left( x^{\cS}_{-d} - x^{\cS}_0 \right) \;,
\end{align}
where we have used
$x^{\cS}_{-d} = M_{\hat{\cX}} x^{\cS}_0 + P_{\hat{\cX}} B \mu$.
This formula does not seem useful because it involves $\left( I - M_{\cG^+[k,0]} \right)^{-1}$.
By then appropriately adding and subtracting $\zeta_{-d} \omega_{-d}^{\sf T}$
one can obtain formulas for quantities akin to $h$ and $q$,
but these formulas involve $\left( I - M_{\cG^+[k,0]} \right)^{-1}$.
}
\begin{equation}
h = \gamma_{-d} \sum_{j=0}^{k-\chi-1} \lambda^j +
\omega_{-d}^{\sf T} \left( M_{\left( \cS^{\overline{\ell d}} \right)^{\chi} \hat{\cX}}
\left( \varphi_{-d} + h \zeta_{-d} \right) +
P_{\left( \cS^{\overline{\ell d}} \right)^{\chi} \hat{\cX}} B \mu -
\varphi_{-d} \right) \lambda^{k-\chi} + \cO \left( \sigma^k \right) \;,
\label{eq:sGplusplusproof2}
\end{equation}
and
\begin{equation}
q = \cO \left( \sigma^k \right) \;.
\label{eq:sGplusplusproof10}
\end{equation}
By (\ref{eq:xGplusForm}) and (\ref{eq:sGplusplusproof10}),
$s^{\cG^+[k,\chi]}_0 = h + \cO \left( \sigma^k \right)$,
and by solving for $h$ in (\ref{eq:sGplusplusproof2}), we arrive at
\begin{equation}
s^{\cG^+[k,\chi]}_0 = \frac{\gamma_{-d} \sum_{j=0}^{k-\chi-1} \lambda^j +
\omega_{-d}^{\sf T} \left( f^{\left( \cS^{\overline{\ell d}} \right)^{\chi} \hat{\cX}}
\left( \varphi_{-d} \right) - \varphi_{-d} \right) \lambda^{k-\chi}}
{1 - \lambda^{k-\chi} \omega_{-d}^{\sf T} M_{\hat{\cX}}
\left( M_{\cS^{\overline{\ell d}}} \right)^{\chi} \zeta_{-d}} +
\cO \left( \sigma^k \right) \;,
\label{eq:sGplus0proof3}
\end{equation}

Finally, by (\ref{eq:omegaMXhatzeta}) with $Q = M_{\cS^{\overline{\ell d}}} = M_{\cY^{\overline{0}}} M_{\cX}$,
and (\ref{eq:detImMGplus}) with $0 \le \chi \le \chi_{\rm max}$ and $\rho = 1$, we obtain
\begin{equation}
\det \left( I - M_{\cG^+[k,\chi]} \right) =
1 - \lambda^{k-\chi} \omega_{-d}^{\sf T} M_{\hat{\cX}}
\left( M_{\cS^{\overline{\ell d}}} \right)^{\chi} \zeta_{-d} + \cO \left( \sigma^k \right) \;.
\label{eq:detImMGplusplusAlt}
\end{equation}
Therefore, by (\ref{eq:sFormulaEarly}), (\ref{eq:sGplus0proof3}) and (\ref{eq:detImMGplusplusAlt}),
we produce (\ref{eq:sGplusplus}), as required.
Above we assumed \linebreak													
$\det \left( I - M_{\cG^+[k,\chi]} \right) \ne 0$.
By continuity, (\ref{eq:sGplusplus}) also holds at points near $(\eta,\nu) = (0,0)$
for which $\det \left( I - M_{\cG^+[k,\chi]} \right) = 0$. 
\end{proof}


We have now developed the tools necessary to prove Theorem \ref{th:main1}.
The proof is given Appendix \ref{app:proofs}.
In this proof we obtain additional asymptotic expressions for the boundary curves that are useful below.
Specifically, on the curves $\det \left( P_{\cG^+[k,\chi]} \right) = 0$
and $\det \left( P_{\cG^+[k,\chi]^{\left( \left( \tilde{\ell}-1 \right) d_k^+ \right)}} \right) = 0$, we have
\begin{equation}
\frac{\lambda^k}{t_d} \eta + \frac{1}{t_{(\ell-1)d}} \nu +
\cO \left( \left( \eta,\nu \right)^2 \right) = 0 \;,
\label{eq:plusBoundaries}
\end{equation}
and on the curves $\det \left( P_{\cG^-[k,\chi]} \right) = 0$
and $\det \left( P_{\cG^-[k,\chi]^{\left( \left( \tilde{\ell}-1 \right) d_k^+ \right)}} \right) = 0$, we have
\begin{equation}
\frac{1}{t_d} \eta + \frac{\lambda^k}{t_{(\ell-1)d}} \nu +
\cO \left( \left( \eta,\nu \right)^2 \right) = 0 \;.
\label{eq:minusBoundaries}
\end{equation}

The next result concerns the eigenvalues of $M_{\cG^\pm[k,\chi]}$
and is used below to prove Theorem \ref{th:main2}.
Theorem \ref{th:main3} follows from this and is proved in Appendix \ref{app:proofs}.

\begin{lemma}
Suppose (\ref{eq:f}) with $K \ge 2$ has an $\cS$-shrinking point
at $\xi = \xi^*$ and $\det(J) \ne 0$ and $\sigma < 1$.
Choose any $\chi_{\rm max} \in \mathbb{Z}^+$ and $k \in \mathbb{Z}^+$.
Then, in a neighbourhood of $(\eta,\nu) = (0,0)$,
for any $|\chi| \le \chi_{\rm max}$,
$M_{\cG^\pm[k,\chi]}$ has an eigenvalue $\rho_{\cG^+[k,\chi]}$ satisfying
\begin{align}
\rho_{\cG^+[k,\chi]} &= \begin{cases}
\lambda^{k + \chi + 1} \omega_{\ell d}^{\sf T}
\left( M_{\cX^{\overline{0}}} M_{\cY} \right)^{-\chi-1}
M_{\check{\cX}} \zeta_{\ell d} + \cO \left( \sigma^k \right) \;, &
-\chi_{\rm max} \le \chi \le -1 \\
\lambda^{k - \chi} \omega_0^{\sf T}
\left( M_{\cY^{\overline{0}}} M_{\cX} \right)^{\chi}
M_{\hat{\cX}} \zeta_0 + \cO \left( \sigma^k \right) \;, &
0 \le \chi \le \chi_{\rm max}
\end{cases} \;, \label{eq:rhoPlus} \\
\rho_{\cG^-[k,\chi]} &= \begin{cases}
\lambda^{k + \chi} \omega_{-d}^{\sf T}
\left( M_{\cY} M_{\cX^{\overline{0}}} \right)^{-\chi}
M_{\hat{\cY}} \zeta_{-d} + \cO \left( \sigma^k \right) \;, &
-\chi_{\rm max} \le \chi \le 0 \\
\lambda^{k - \chi + 1} \omega_{(\ell-1)d}^{\sf T}
\left( M_{\cX} M_{\cY^{\overline{0}}} \right)^{\chi - 1}
M_{\check{\cY}} \zeta_{(\ell-1) d} + \cO \left( \sigma^k \right) \;, &
1 \le \chi \le \chi_{\rm max}
\end{cases} \;, \label{eq:rhoMinus}
\end{align}
and all other eigenvalues of $M_{\cG^\pm[k,\chi]}$ are $\cO \left( \sigma^k \right)$.
\label{le:eigenvalues}
\end{lemma}

\begin{proof}
Here we prove the result for $\cG^+[k,\chi]$ with $0 \le \chi \le \chi_{\rm max}$.
The remaining parts can be proved similarly.

By Lemma \ref{le:eigMSindep}, the eigenvalues of $M_{\cG^+[k,\chi]}$
are the same as those of $M_{\cG^+[k,\chi]^{(d-n)}}$,
where the latter matrix is given by (\ref{eq:detImMGplusproof5}).
Therefore the eigenvalues of $M_{\cG^+[k,\chi]}$
differ from those of the rank-one matrix $\lambda^{k-\chi} \zeta_0 \omega_0^{\sf T}
\left( M_{\cY^{\overline{0}}} M_{\cX} \right)^{\chi} M_{\hat{\cX}}$, by $\cO \left( \sigma^k \right)$,
and the only non-zero eigenvalue of this matrix is
$\lambda^{k - \chi} \omega_0^{\sf T}
\left( M_{\cY^{\overline{0}}} M_{\cX} \right)^{\chi} M_{\hat{\cX}} \zeta_0$.
\end{proof}

\begin{proof}[Proof of Theorem \ref{th:main2}]
Here we prove the result for $\cG^+[k,\chi]$, with $0 \le \chi \le \chi_{\rm max}$.
The remaining parts of Theorem \ref{th:main2} can be proved in a similar fashion.

We look for intersections of $\det \left( P_{\cG^+[k,\chi]} \right) = 0$,
with $\det \left( \rho I - M_{\cG^+[k,\chi]} \right) = 0$, for $\rho = \pm 1$.
Along $\det \left( P_{\cG^+[k,\chi]} \right) = 0$,
the parameters $\eta$ and $\nu$ are $\cO \left( \frac{1}{k} \right)$, and satisfy (\ref{eq:plusBoundaries}).
In polar coordinates (\ref{eq:polarCoords}),
$\tan(\theta) = -\frac{t_d \nu}{t_{(\ell-1)d} \eta}$.
Thus by (\ref{eq:plusBoundaries}), if $\det \left( P_{\cG^+[k,\chi]} \right) = 0$, then
\begin{equation}
\tan(\theta) = -\lambda^k + \cO \left( \frac{1}{k} \right) \;.
\label{eq:main2proof1}
\end{equation}

Notice $\det \left( \rho I - M_{\cG^+[k,\chi]} \right) = 0$
when $\rho = \rho_{\cG^+[k,\chi]}$, as given by (\ref{eq:rhoPlus}).
Thus by rearranging (\ref{eq:rhoPlus}) with $0 \le \chi \le \chi_{\rm max}$ and putting $\rho_{\cG^+[k,\chi]} = \pm 1$,
we see that $\det \left( \rho I - M_{\cG^+[k,\chi]} \right) = 0$ when
\begin{equation}
\lambda^k = \left. \frac{\pm 1}{u_0^{\sf T} \left( M_{\cY^{\overline{0}}} M_{\cX} \right)^{\chi}
M_{\hat{\cX}} v_0} \right|_{(0,0)} + \cO \left( \frac{1}{k} \right) \;,
\label{eq:main2proof2}
\end{equation}
where we have used the fact that
$\omega_0^{\sf T}$ and $\zeta_0$ are equal to $u_0^{\sf T}$ and $v_0$, to leading order (\ref{eq:zetaomega00}).
Since $M_{\cS^{\overline{\ell d}}} = M_{\cY^{\overline{0}}} M_{\cX}$
and $v_{-d} = -\frac{t_d}{t_{-d}} M_{\hat{\cX}} v_0$ (\ref{eq:vumd2}),
by (\ref{eq:kappaPlus}), (\ref{eq:main2proof2}) reduces to
\begin{equation}
\lambda^k = \frac{\mp t_d}{t_{-d} \kappa^+_\chi} + \cO \left( \frac{1}{k} \right) \;,
\label{eq:main2proof3}
\end{equation}
where $\kappa^+_{\chi}$ is given by (\ref{eq:kappaPlus}),
and we assume $\kappa^+_{\chi} \ne 0$.

Therefore, by (\ref{eq:main2proof1}) and (\ref{eq:main2proof3}), 
any intersection of $\det \left( P_{\cG^+[k,\chi]} \right) = 0$ and 
$\det \left( I - M_{\cG^+[k,\chi]} \right) = 0$ must satisfy
$\tan(\theta) = \frac{t_d}{t_{-d} \kappa^+_\chi} + \cO \left( \frac{1}{k} \right)$.
Since $\tan(\theta) < 0$, $t_d < 0$ and $t_{-d} > 0$, 
the intersection point exists if and only if $\kappa^+_{\chi} > 0$.

Similarly any intersection of $\det \left( P_{\cG^+[k,\chi]} \right) = 0$ and 
$\det \left( -I - M_{\cG^+[k,\chi]} \right) = 0$ must satisfy
$\tan(\theta) = \frac{-t_d}{t_{-d} \kappa^+_\chi} + \cO \left( \frac{1}{k} \right)$.
In this case the intersection point exists if and only if $\kappa^+_{\chi} < 0$.
\end{proof}

\section{Properties of nearby shrinking points}
\label{sec:properties}
\setcounter{equation}{0}

In this section we work towards a proof of Theorem \ref{th:main4}.

Recall, $\eta = s^{\cS^{\overline{0}}}_0$ and $\nu = s^{\cS^{\overline{0}}}_{\ell d}$ (\ref{eq:etanu})
provide a convenient local coordinate system in which to study
the dynamics of (\ref{eq:f}) near an $\cS$-shrinking point.
As in \S\ref{sub:definitions}, let $\tilde{\eta} = s^{\cT^{\overline{0}}}_0$ and
$\tilde{\nu} = s^{\cT^{\overline{0}}}_{\tilde{\ell} d_k^\pm}$ denote the analogous 
coordinates for a nearby $\cT = \cG^\pm[k,\chi]$-shrinking point,
where the $\cT$-shrinking point is located at $(\eta,\nu) = \left( \eta_{\cT}, \nu_{\cT} \right)$,
Theorem \ref{th:main2}.

In order to relate $\tilde{\eta}$ and $\tilde{\nu}$ to $\eta$ and $\nu$,
we first represent points $(\eta,\nu)$
as perturbations from $\left( \eta_{\cT}, \nu_{\cT} \right)$, by writing
\begin{equation}
(\eta,\nu) = \left( \eta_{\cT} + \Delta \eta, \nu_{\cT} + \Delta \nu \right) \;.
\label{eq:Deltaetanu}
\end{equation}
At the $\cT$-shrinking point, $\Delta \eta = \Delta \nu = 0$
and $\tilde{\eta} = \tilde{\nu} = 0$.
From Lemmas \ref{le:detAll} and \ref{le:sGall2} and (\ref{eq:sFormulaEarly}),
it can be seen that $\frac{\partial s^{\cT^{\overline{0}}}_0}{\partial \eta}$ and
$\frac{\partial s^{\cT^{\overline{0}}}_0}{\partial \nu}$ are $\cO(k)$,
as are the first derivatives of $s^{\cT^{\overline{0}}}_{\tilde{\ell} d_k^\pm}$.
It follows that $\tilde{\eta}$ and $\tilde{\nu}$ admit the following expansion:
\begin{align}
\tilde{\eta} &= \left( p_1 k + p_3 + \cO \left( \frac{1}{k} \right) \right) \Delta \eta +
\left( p_2 k + p_4 + \cO \left( \frac{1}{k} \right) \right) \Delta \nu +
\cO \left( \left( \Delta \eta, \Delta \nu \right)^2 \right) \;, \label{eq:tildeeta} \\
\tilde{\nu} &= \left( q_1 k + q_3 + \cO \left( \frac{1}{k} \right) \right) \Delta \eta +
\left( q_2 k + q_4 + \cO \left( \frac{1}{k} \right) \right) \Delta \nu +
\cO \left( \left( \Delta \eta, \Delta \nu \right)^2 \right) \;, \label{eq:tildenu}
\end{align}
for some constants $p_i, q_i \in \mathbb{R}$.

The condition $\det(J) \ne 0$ ensures that the change of coordinates
from $(\xi_1,\xi_2) \leftrightarrow (\eta,\nu)$ is locally invertible.
Similarly, the condition $\det(\tilde{J}) \ne 0$, where $\tilde{J}$ is
given by (\ref{eq:JtildeEarly}), ensures that the change of coordinates
$(\eta,\nu) \leftrightarrow \left( \tilde{\eta}, \tilde{\nu} \right)$ is locally invertible.
In view of (\ref{eq:tildeeta}) and (\ref{eq:tildenu}), we can write
\begin{equation}
\tilde{J} \defeq \left. \begin{bmatrix}
\frac{\partial \tilde{\eta}}{\partial \eta} &
\frac{\partial \tilde{\eta}}{\partial \nu} \\
\frac{\partial \tilde{\nu}}{\partial \eta} &
\frac{\partial \tilde{\nu}}{\partial \nu}
\end{bmatrix} \right|_{\left( \eta_{\cT}, \nu_{\cT} \right)} =
\begin{bmatrix}
p_1 k + p_3 &
p_2 k + p_4 \\
q_1 k + q_3 &
q_2 k + q_4
\end{bmatrix}
+ \cO \left( \frac{1}{k} \right) \;.
\label{eq:Jtilde}
\end{equation}

Below in Lemma \ref{le:pq12} we derive identities involving the constants $p_i, q_i \in \mathbb{R}$.
First we show that in the analogous expansion for $\det \left( I - M_{\cT} \right)$
the leading order coefficients take a simple form and are independent to the choice of $\cT$.

\begin{lemma}
Suppose (\ref{eq:f}) with $K \ge 2$ has an $\cS$-shrinking point
at $\xi = \xi^*$ and $\det(J) \ne 0$ and $\sigma < 1$.
Then for any $\cT = \cG^\pm[k,\chi]$,
\begin{equation}
\det \left( I - M_{\cT} \right) =
\left( \frac{a}{c t_d} k + \cO(1) \right) \Delta \eta +
\left( \frac{a}{c t_{(\ell-1)d}} k + \cO(1) \right) \Delta \nu +
\cO \left( \left( \Delta \eta, \Delta \nu \right)^2 \right) \;.
\label{eq:detImMTnearby}
\end{equation}
\label{le:detImMTnearby}
\end{lemma}

\begin{proof}
By (\ref{eq:lambda}),
\begin{equation}
\lambda^k = \left( 1 - \frac{a}{c t_d} \eta -
\frac{a}{c t_{(\ell-1)d}} \nu + \cO \left( \left( \eta,\nu \right)^2 \right) \right)^k \;.
\label{eq:lambdak2}
\end{equation}
By (\ref{eq:lambda}),
$\frac{\partial \lambda}{\partial \eta} \left( \eta_{\cT}, \nu_{\cT} \right) =
\frac{-a}{c t_d} + \cO \left( \frac{1}{k} \right)$, and
$\frac{\partial \lambda}{\partial \nu} \left( \eta_{\cT}, \nu_{\cT} \right) =
\frac{-a}{c t_{(\ell-1)d}} + \cO \left( \frac{1}{k} \right)$.
Therefore by expanding (\ref{eq:lambdak2}) about
$(\eta,\nu) = \left( \eta_{\cT}, \nu_{\cT} \right)$, we obtain\removableFootnote{
We are evaluating
\begin{equation}
\lambda^k(\eta,\nu) = \lambda^k \left( \eta_{\cT}, \nu_{\cT} \right) +
k \lambda^{k-1} \left( \eta_{\cT}, \nu_{\cT} \right)
\frac{\partial \lambda}{\partial \eta} \left( \eta_{\cT}, \nu_{\cT} \right) \Delta \eta +
k \lambda^{k-1} \left( \eta_{\cT}, \nu_{\cT} \right)
\frac{\partial \lambda}{\partial \nu} \left( \eta_{\cT}, \nu_{\cT} \right) \Delta \nu +
\cO \left( \left( \Delta \eta, \Delta \nu \right)^2 \right) \;.
\end{equation}
}
\begin{equation}
\lambda^k(\eta,\nu) = \lambda^k \left( \eta_{\cT}, \nu_{\cT} \right)
\left( 1 - \left( \frac{a}{c t_d} k + \cO(1) \right) \Delta \eta -
\left( \frac{a}{c t_{(\ell-1)d}} k + \cO(1) \right) \Delta \nu +
\cO \left( \left( \Delta \eta, \Delta \nu \right)^2 \right) \right) \;,
\label{eq:detImMTnearbyproof1}
\end{equation}
where we have also substituted $\lambda^{k-1} \left( \eta_{\cT}, \nu_{\cT} \right) =
\lambda^k \left( \eta_{\cT}, \nu_{\cT} \right) + \cO \left( \frac{1}{k} \right)$.

By Lemma \ref{le:detAll} we can write
\begin{equation}
\det \left( I - M_{\cT} \right) =
1 - c_1 \lambda^k + \cO \left( \frac{1}{k} \right) \;,
\label{eq:detImMTnearbyproof2}
\end{equation}
where $c_1 \in \mathbb{R}$ depends on $\cT$ but is independent of $k$.
We have $\det \left( I - M_{\cT} \right) = 0$
at $(\eta,\nu) = \left( \eta_{\cT}, \nu_{\cT} \right)$, thus
\begin{equation}
\lambda^k \left( \eta_{\cT}, \nu_{\cT} \right) =
\frac{1}{c_1 \left( \eta_{\cT}, \nu_{\cT} \right)} + \cO \left( \frac{1}{k} \right) \;.
\label{eq:detImMTnearbyproof3}
\end{equation}
With (\ref{eq:detImMTnearbyproof1}) and (\ref{eq:detImMTnearbyproof3}),
(\ref{eq:detImMTnearbyproof2}) reduces to (\ref{eq:detImMTnearby}) as required.
\end{proof}

\begin{lemma}
For any $\cT = \cG^\pm[k,\chi]$,
the coefficients of (\ref{eq:tildeeta}) and (\ref{eq:tildenu}) satisfy
\begin{align}
\frac{1}{t_{(\ell-1)d}} \left( 1 \mp
\frac{{\rm sgn}(a)}{\Gamma \left( \theta^+_{\chi} \right) \sin \left( \theta^+_{\chi} \right)} \right) p_1 -
\frac{1}{t_d} \left( 1 \pm
\frac{{\rm sgn}(a)}{\Gamma \left( \theta^+_{\chi} \right) \cos \left( \theta^+_{\chi} \right)} \right) p_2 &= 0 \;,
\label{eq:p1p2} \\
\frac{1}{t_{(\ell-1)d}} \left( 1 \mp
\frac{{\rm sgn}(a)}{\Gamma \left( \theta^+_{\chi} \right) \sin \left( \theta^+_{\chi} \right)} \right) q_1 -
\frac{1}{t_d} \left( 1 \pm
\frac{{\rm sgn}(a)}{\Gamma \left( \theta^+_{\chi} \right) \cos \left( \theta^+_{\chi} \right)} \right) q_2 &= 0 \;,
\label{eq:q1q2}
\end{align}
where $\Gamma$ is given by (\ref{eq:Gamma})-(\ref{eq:GammaExtended}) and
the $\theta^\pm_{\chi}$ are given by (\ref{eq:thetaPlus})-(\ref{eq:thetaMinus}).
\label{le:pq12}
\end{lemma}

\begin{proof}
Here we prove the result for $\cT = \cG^+[k,\chi]$.
The result for $\cT = \cG^-[k,\chi]$ can be proved in the same fashion.

The curve $\tilde{\eta} = 0$
is a boundary of the $\cG^+_k$-mode-locking region emanating from the $\cT = G^+[k,\chi]$-shrinking point.
Note that $\det \left( P_{\cG^+[k,\chi]} \right) = 0$,
which is approximated by (\ref{eq:plusBoundaries}), defines the same curve.
By evaluating (\ref{eq:plusBoundaries}) at the $\cT$-shrinking point, we obtain
\begin{equation}
\lambda^k \left( \eta_{\cT}, \nu_{\cT} \right) = 
\frac{-t_d \nu_{\cT}}{t_{(\ell-1)d} \eta_{\cT}} + \cO \left( \frac{1}{k} \right) \;.
\label{eq:pq12proof1}
\end{equation}
By further expanding (\ref{eq:plusBoundaries}) about the $\cT$-shrinking point
with (\ref{eq:Deltaetanu}), we determine that $\tilde{\eta} = 0$ is described by
\begin{equation}
\frac{1}{t_d} \left( 1 - \frac{a k \eta_{\cT}}{c t_d} +
\cO \left( \frac{1}{k} \right) \right) \Delta \eta +
\frac{1}{t_{(\ell-1)d}} \left( \frac{1}{\lambda^k \left( \eta_{\cT}, \nu_{\cT} \right)} -
\frac{a k \eta_{\cT}}{c t_d} +
\cO \left( \frac{1}{k} \right) \right) \Delta \nu +
\cO \left( \left( \Delta \eta, \Delta \nu \right)^2 \right) = 0 \;.
\label{eq:pq12proof2}
\end{equation}
By multiplying both sides of (\ref{eq:pq12proof2}) by
$\frac{-c t_d}{a \eta_{\cT} k}$ and substituting (\ref{eq:pq12proof1}), we obtain
\begin{equation}
\left( \frac{1}{t_d} -
\frac{c}{a \eta_{\cT} k} +
\cO \left( \frac{1}{k} \right) \right) \Delta \eta +
\left( \frac{1}{t_{(\ell-1)d}} +
\frac{c}{a \nu_{\cT} k} +
\cO \left( \frac{1}{k} \right) \right) \Delta \nu +
\cO \left( \left( \Delta \eta, \Delta \nu \right)^2 \right) \;.
\label{eq:pq12proof3}
\end{equation}
By then evaluating $\eta_{\cT}$ and $\nu_{\cT}$ with (\ref{eq:polarCoords}) and (\ref{eq:nearbyCurve}),
and taking care to accommodate different cases depending on the sign of $a$,
we determine that $\tilde{\eta} = 0$ is described by
\begin{align}
& \frac{1}{t_d} \left( 1 + \frac{{\rm sgn}(a)}{\Gamma \left( \theta^+_{\chi} \right) \cos \left( \theta^+_{\chi} \right)} +
\cO \left( \frac{1}{k} \right) \right) \Delta \eta \nonumber \\
&+ \frac{1}{t_{(\ell-1)d}} \left( 1 - \frac{{\rm sgn}(a)}{\Gamma \left( \theta^+_{\chi} \right) \sin \left( \theta^+_{\chi} \right)} +
\cO \left( \frac{1}{k} \right) \right) \Delta \nu +
\cO \left( \left( \Delta \eta, \Delta \nu \right)^2 \right) \;.
\label{eq:pq12proof4}
\end{align}
By matching (\ref{eq:tildeeta}) and (\ref{eq:pq12proof4}) we obtain (\ref{eq:p1p2}) for $\cT = \cG^+[k,\chi]$.
The curve $\tilde{\nu} = 0$ is also given by (\ref{eq:plusBoundaries}),
hence the same result holds for $q_1$ and $q_2$, i.e.~(\ref{eq:q1q2}).
\end{proof}

We complete this section by deriving a novel identity for the leading order term of
$\det \left( \tilde{J} \right)$.
This is used to prove Theorem \ref{th:main4} in Appendix \ref{app:proofs}.
First note that by (\ref{eq:detImMS}) we can write
\begin{equation}
\det \left( I - M_{\cT} \right) =
\frac{\tilde{a}}{\tilde{t}_{d_k^\pm}} \tilde{\eta} + 
\frac{\tilde{a}}{\tilde{t}_{\left( \tilde{\ell} - 1 \right) d_k^\pm}} \tilde{\nu} +
\cO \left( \left( \tilde{\eta}, \tilde{\nu} \right)^2 \right) \;.
\label{eq:tildedetImMT2}
\end{equation}
Moreover, from Lemmas \ref{le:detAll} and \ref{le:sGall2}
it can be seen that $\tilde{t}_{d_k^\pm}$ and
$\tilde{t}_{\left( \tilde{\ell} - 1 \right) d_k^\pm}$ are $\cO \left( \frac{1}{k} \right)$,
and $\tilde{a}$ is $\cO(1)$,
and therefore we can write
\begin{equation}
\det \left( I - M_{\cT} \right) = \left( r_1 k + r_3 + \cO \left( \frac{1}{k} \right) \right) \tilde{\eta} +
\left( r_2 k + r_4 + \cO \left( \frac{1}{k} \right) \right) \tilde{\nu} +
\cO \left( \left( \tilde{\eta}, \tilde{\nu} \right)^2 \right) \;,
\label{eq:tildedetImMT}
\end{equation}
for some constants $r_i \in \mathbb{R}$.

\begin{lemma}
For any $\cT = \cG^\pm[k,\chi]$,
the coefficients of (\ref{eq:tildedetImMT}) satisfy
\begin{equation}
p_1 r_1 + q_1 r_2 = 0 \;,
\label{eq:consistencyCondition}
\end{equation}
and $\tilde{J}$ (\ref{eq:Jtilde}) satisfies\removableFootnote{
Here is an equivalent formula:
\begin{equation}
\det \left( \tilde{J} \right) =
-\frac{a}{c r_2} \left( \frac{p_2}{t_d} - \frac{p_1}{t_{(\ell-1)d}} \right) k + \cO(1) \;.
\end{equation}
}
\begin{equation}
\det \left( \tilde{J} \right) =
\frac{a}{c r_1} \left( \frac{q_2}{t_d} - \frac{q_1}{t_{(\ell-1)d}} \right) k + \cO(1) \;.
\label{eq:detJtilde}
\end{equation}
\label{le:detJtilde}
\end{lemma}

\begin{proof}
By (\ref{eq:Jtilde}),
\begin{equation}
\det \left( \tilde{J} \right) =
\left( p_1 q_4 + p_3 q_2 - p_2 q_3 - p_4 q_1 \right) k + \cO(1) \;,
\label{eq:detJtildeproof1}
\end{equation}
where the $k^2$-term has vanished because (\ref{eq:p1p2}) and (\ref{eq:q1q2}) imply
\begin{equation}
p_1 q_2 - p_2 q_1 = 0 \;.
\label{eq:pq12}
\end{equation}
By substituting (\ref{eq:tildeeta}) and (\ref{eq:tildenu}) into (\ref{eq:tildedetImMT}) we obtain
\begin{align}
\det \left( I - M_{\cS} \right) &=
\left( \left( p_1 r_1 + q_1 r_2 \right) k^2 +
\left( p_1 r_3 + p_3 r_1 + q_1 r_4 + q_3 r_2 \right) k + \cO(1) \right) \Delta \eta \nonumber \\
&\quad+\left( \left( p_2 r_1 + q_2 r_2 \right) k^2 +
\left( p_2 r_3 + p_4 r_1 + q_2 r_4 + q_4 r_2 \right) k + \cO(1) \right) \Delta \nu +
\cO \left( \left( \Delta \eta, \Delta \nu \right)^2 \right) \;.
\label{eq:detJtildeproof2}
\end{align}
By matching the $k^2$-terms of (\ref{eq:detImMTnearby}) and (\ref{eq:detJtildeproof2}),
we deduce that $p_1 r_1 + q_1 r_2 = 0$
(verifying (\ref{eq:consistencyCondition})) and $p_2 r_1 + q_2 r_2 = 0$.
Note that these equations are equivalent in view of (\ref{eq:pq12}).
By then matching the $k$-terms of (\ref{eq:detImMTnearby}) and (\ref{eq:detJtildeproof2}), we obtain
\begin{equation}
p_1 r_3 + p_3 r_1 + q_1 r_4 + q_3 r_2 = \frac{a}{c t_d} \;, \qquad
p_2 r_3 + p_4 r_1 + q_2 r_4 + q_4 r_2 = \frac{a}{c t_{(\ell-1)d}} \;. \label{eq:detJtildeproof3}
\end{equation}
By combining (\ref{eq:consistencyCondition}), (\ref{eq:pq12}) and (\ref{eq:detJtildeproof3}), we obtain
\begin{equation}
\frac{a}{c} \left( \frac{q_2}{t_d} - \frac{q_1}{t_{(\ell-1)d}} \right) =
r_1 \left( p_1 q_4 + p_3 q_2 - p_2 q_3 - p_4 q_1 \right) \;,
\end{equation}
which by (\ref{eq:detJtildeproof1}) yields (\ref{eq:detJtilde}), as required.
\end{proof}

\section{Summary}
\label{sec:conc}
\setcounter{equation}{0}

Shrinking points are codimension-two points in the parameter space of
piecewise-linear continuous maps at which mode-locking regions have zero width.
In this paper we have studied the $N$-dimensional map (\ref{eq:f}),
which has a single switching manifold, $s = 0$.
We have considered mode-locking regions that, in a symbolic sense,
can be assigned a rotation number, $\frac{m}{n}$.
At any shrinking point in such a mode-locking region
there exists an invariant polygon in the phase space of (\ref{eq:f}).
All orbits on the polygon have period $n$,
rotation number $\frac{m}{n}$, and, say, $\ell$ points to the left of the switching manifold per period
(except a special periodic orbit, labelled $\{ y_i \}$,
that has two points on the switching manifold).

This paper provides the first rigorous study into the dynamics near an arbitrary shrinking point,
other than the period-$n$ dynamics within the mode-locking region itself which was examined in \cite{SiMe09}.
We refer to the shrinking point as an $\cS$-shrinking point,
where $\cS = \cF[\ell,m,n]$ is the symbol sequence associated with orbits on the invariant polygon.
On each side of the mode-locking region connected to an $\cS$-shrinking point,
there is a sequence of mode-locking regions.
On one side the mode-locking regions have associated rotation numbers $\frac{k m + m^-}{k n + n^-}$,
and on the other side the mode-locking regions have associated rotation numbers $\frac{k m + m^+}{k n + n^+}$,
where $k \in \mathbb{Z}^+$ and $\frac{m^-}{n^-}$ and $\frac{m^+}{n^+}$
are the left and right Farey roots of $\frac{m}{n}$.
The local curvature and relative spacing of these mode-locking regions
was described using polar coordinates and
the nonlinear function $\Gamma$ (\ref{eq:Gamma})-(\ref{eq:GammaExtended}),
as indicated in Theorem \ref{th:main1}.

The two sequences of mode-locking regions themselves have shrinking points.
Thus sequences of shrinking points converge to the $\cS$-shrinking point.
We have characterised these shrinking points with symbol sequences, $\cG^\pm[k,\chi]$.
But the $\cG^\pm[k,\chi]$-shrinking points only exist for particular values of $\chi \in \mathbb{Z}$.
We proved, subject to certain non-degeneracy conditions, see Theorem \ref{th:main2},
that there exists a sequence of potential $\cG^\pm[k,\chi]$-shrinking points,
that converge to the $\cS$-shrinking point as $k \to \infty$,
if and only if $\kappa^\pm_{\chi} > 0$,
where $\kappa^\pm_{\chi}$ are scalar constants associated with the $\cS$-shrinking point.
The angular coordinates of the potential $\cG^\pm[k,\chi]$-shrinking points are
given, to leading order, by $\theta^\pm_{\chi}$.
Numerical investigations reveal that these points are commonly valid shrinking points,
but may not be due a lack of admissibility of the orbits on the associated invariant polygon.
Theorem \ref{th:main3} and equation (\ref{eq:uvIdentity12early})
show that there are some restrictions on the combinations of signs possible for the $\kappa^\pm_{\chi}$.
Theorem \ref{th:main4} tells us that nearby $\cG^\pm[k,\chi]$-shrinking points
are non-degenerate and have the same orientation as the $\cS$-shrinking point.

It remains to describe other dynamics near shrinking points,
such as periodic, quasiperiodic and chaotic dynamics at points in parameter space between
the nearby mode-locking regions that we have identified,
and consider more general classes of piecewise-smooth maps.
Such maps arise in diverse applications, and
if there is only weak nonlinearity in the smooth pieces of the map
(or if the relevant orbits are only traversing parts of phase space that involve weak nonlinearity),
then the mode-locking regions can exhibit a sausage-string structure
involving points of near-zero width \cite{Ti02,SzOs09,ZhMo08b}.
Border-collision bifurcations are described by piecewise-smooth continuous maps,
and the influence of the nonlinearity in the pieces of the map
increases with the distance in parameter space from the border-collision bifurcation.
This influence on mode-locking region boundaries emanating
from shrinking points was explained in \cite{SiMe10},
but it remains to understand the effect of such nonlinearities on other local dynamics.

\appendix

\section{Additional proofs}
\label{app:proofs}
\setcounter{equation}{0}

\begin{proof}[Proof of Lemma \ref{le:adjIdentity2}]
First, suppose that $A$ and $A + v u^{\sf T}$ are nonsingular.
The identity
\begin{equation}
\left( A + v u^{\sf T} \right)^{-1} =
A^{-1} - \frac{A^{-1} v u^{\sf T} A^{-1}}{1 + u^{\sf T} A^{-1} v} \;,
\label{eq:ShermanMorrison}
\end{equation}
is known as the Sherman-Morrison formula and be can verified directly.
We use (\ref{eq:adjIdentity}) to rewrite (\ref{eq:ShermanMorrison}) as
\begin{equation}
\frac{{\rm adj} \left( A + v u^{\sf T} \right)}{\det \left( A + v u^{\sf T} \right)} =
\frac{{\rm adj}(A)}{\det(A)} -
\frac{{\rm adj}(A) v u^{\sf T} {\rm adj}(A)}{\det(A) \det \left( A + v u^{\sf T} \right)} \;,
\end{equation}
and therefore
\begin{equation}
{\rm adj} \left( A + v u^{\sf T} \right) = {\rm adj}(A) +
{\rm adj}(A) u^{\sf T} {\rm adj}(A) v -
{\rm adj}(A) v u^{\sf T} {\rm adj}(A) \;.
\label{eq:ShermanMorrison3}
\end{equation}
Upon multiplying (\ref{eq:ShermanMorrison3}) by $u^{\sf T}$ on the left,
the last two terms cancel leaving us with (\ref{eq:adjIdentity2}).

The subset of triples $(A,u,v)$ for which both $A$ and $A + v u^{\sf T}$ are nonsingular
is dense in the set of all triples $(A,u,v)$.
Therefore since both sides of (\ref{eq:adjIdentity2}) are continuous functions of $A$, $u$ and $v$,
(\ref{eq:adjIdentity2}) holds in general.
\end{proof}

\begin{proof}[Proof of Lemma \ref{le:adjugateRank}]
First, suppose ${\rm rank}(A) = N-1$.
Then $0$ is an eigenvalue of $A$, and so there exist
$u, v \in \mathbb{R}^N$ such that $u^{\sf T} A = 0$,
$A v = 0$, and $u^{\sf T} v = 1$.
By (\ref{eq:adjIdentity}), ${\rm adj}(A)$ must be of the form
\begin{equation}
{\rm adj}(A) = \hat{c} v u^{\sf T} \;,
\label{eq:adjIdentity3proof1}
\end{equation}
for some $\hat{c} \in \mathbb{R}$.
To demonstrate (\ref{eq:adjIdentity3}) it remains to show that $\hat{c} = c$.

Let $\ee \in \mathbb{R}$.
Then by (\ref{eq:adjIdentity}) and $A v = 0$, we have
\begin{equation}
\det(A + \ee I) v =
{\rm adj}(A + \ee I)(A + \ee I) v =
{\rm adj}(A + \ee I) v \ee =
{\rm adj}(A) v \ee + \cO \left( \ee^2 \right) \;.
\label{eq:adjIdentity3proof2}
\end{equation}
By substituting (\ref{eq:adjIdentity3proof1}) into (\ref{eq:adjIdentity3proof2})
and using $u^{\sf T} v = 1$, we obtain
\begin{equation}
\det(A + \ee I) v = \hat{c} v \ee + \cO \left( \ee^2 \right) \;.
\label{eq:adjIdentity3proof3}
\end{equation}

Notice, $\ee$ is an eigenvalue of $A + \ee I$.
Let $\lambda_j(\ee)$, for $j = 2,\ldots,N$,
denote the remaining eigenvalues of $A + \ee I$, counting multiplicity.
By definition, $c = \prod_{j=2}^N \lambda_j(0)$, and $\det(A + \ee I)$
is the product of all eigenvalues of $A + \ee I$, thus
\begin{equation}
\det(A + \ee I) =
\ee \prod_{j=2}^N \lambda_j(\ee) =
c \ee + \cO \left( \ee^2 \right) \;.
\label{eq:adjIdentity3proof4}
\end{equation}
By matching (\ref{eq:adjIdentity3proof3}) and (\ref{eq:adjIdentity3proof4}),
we deduce that $\hat{c} = c$, and therefore (\ref{eq:adjIdentity3}) as required.

Second, if ${\rm rank}(A) < N-1$
then for any $i$ and $j$, the $(N-1) \times (N-1)$ matrix formed 
by removing the $i^{\rm th}$ row and $j^{\rm th}$ column from $A$ also has rank less than $N-1$.
Thus $m_{ij} = 0$ for all $i$ and $j$ and so ${\rm adj}(A)$ is the zero matrix.
\end{proof}

\begin{proof}[Proof of Theorem \ref{th:main1}]
Here we construct $C^K$ curves along which
$\det \left( P_{\cG^+[k,\chi]} \right) = 0$, for $0 \le \chi \le \chi_{\rm max}$,
and $\det \left( P_{\cG^+[k,\chi]^{\left( \left( \tilde{\ell}-1 \right) d_k^+ \right)}} \right) = 0$,
for $1 \le \chi \le \chi_{\rm max}$,
and verify (\ref{eq:nearbyCurve})-(\ref{eq:thetaSigns}) for these curves\removableFootnote{
These represent a quarter of all curves, given $\chi_{\rm max}$.
}.
The result for the remaining curves can be obtained in the same fashion.

To solve $\det \left( P_{\cG^+[k,\chi]} \right) = 0$ we use (\ref{eq:sGplusplus}).
In (\ref{eq:sGplusplus}) we substitute
$\varphi_{-d} = y_0 + \cO(\eta,\nu)$,
$f^{\cS^{\overline{\ell d}}} \left( y_0 \right) = y_0 + \cO(\eta,\nu)$,
and $f^{\hat{\cX}} \left( y_0 \right) = y_{-d} + \cO(\eta,\nu)$, to obtain
\begin{align}
\omega_{-d}^{\sf T} \left( f^{\left( \cS^{\overline{\ell d}} \right)^{\chi} \hat{\cX}}
\left( \varphi_{-d} \right) - \varphi_{-d} \right) &=
u_{-d}^{\sf T} \left( y_{-d} - y_0 \right) + \cO(\eta,\nu) \nonumber \\
&= t_{-d} + \cO(\eta,\nu) \;.
\end{align}
Thus by (\ref{eq:sGplusplus}), if $\det \left( I - M_{\cS} \right) \ne 0$
(in which case $\lambda \ne 1$), we can write
\begin{equation}
\det \left( P_{\cG^+[k,\chi]} \right) \varrho^{\sf T} B \mu =
\frac{\gamma_{-d} \left( 1 - \lambda^{k-\chi} \right)}{1-\lambda} +
\lambda^{k-\chi} \left( t_{-d} + \cO(\eta,\nu) \right) + \cO \left( \sigma^k \right) \;.
\label{eq:main1proof2}
\end{equation}

By (\ref{eq:lambda}),
\begin{equation}
1-\lambda = \frac{a}{c t_d} \eta + \frac{a}{c t_{(\ell-1)d}} \nu +
\cO \left( \left( \eta,\nu \right)^2 \right) \;.
\label{eq:lambda2}
\end{equation}
To evaluate $\gamma_{-d}$, we substitute (\ref{eq:xSi}) and (\ref{eq:lambda2}) into
(\ref{eq:gamma}) to obtain, after simplification,
\begin{equation}
\gamma_{-d} = \frac{a t_{-d}}{c t_{(\ell-1)d}} \nu +
\cO \left( \left( \eta,\nu \right)^2 \right) \;.
\label{eq:gammamd}
\end{equation}
Also, since $\lambda = 1 + \cO(\eta,\nu)$ and $\chi$ is a constant (independent of $k$), we can write
$\lambda^{k-\chi} = \lambda^k \left( 1 + \cO(\eta,\nu) \right)$.
By substituting these expressions into (\ref{eq:main1proof2}) we arrive at
\begin{equation}
\det \left( P_{\cG^+[k,\chi]} \right) \varrho^{\sf T} B \mu =
\frac{a t_{-d}}{c (1-\lambda)}
\left( \frac{\lambda^k}{t_d} \eta + \frac{1}{t_{(\ell-1)d}} \nu +
\cO \left( \left( \eta,\nu \right)^2 \right) \right) \;.
\label{eq:main1proof6}
\end{equation}
Note, the apparent singularity $\lambda = 1$ in (\ref{eq:main1proof6}) is spurious
because $\det \left( P_{\cG^+[k,\chi]} \right) \varrho^{\sf T} B \mu$
is $C^K$ in a neighbourhood of $(\eta,\nu) = (0,0)$.

The only instance of $k$ in the leading order term of (\ref{eq:main1proof6})
occurs in the quantity $\lambda^k$, where by (\ref{eq:lambda2})
\begin{equation}
\lambda^k = \left( 1 - \frac{a}{c t_d} \eta -
\frac{a}{c t_{(\ell-1)d}} \nu + \cO \left( \left( \eta,\nu \right)^2 \right) \right)^k \;.
\label{eq:lambdak}
\end{equation}
Therefore, in the limit $k \to \infty$,
$\lambda^k$ only can take $\cO(1)$ values other than $0$ and $1$
if $\eta, \nu = \cO \left( \frac{1}{k} \right)$, as the limit is taken.
For this reason it is appropriate to write
\begin{equation}
\eta = \frac{\hat{\eta}}{k} \;, \qquad
\nu = \frac{\hat{\nu}}{k} \;.
\label{eq:hatetanu}
\end{equation}
and treat $\hat{\eta}$ and $\hat{\nu}$ as $\cO(1)$ constants.
By substituting (\ref{eq:hatetanu}) into (\ref{eq:lambdak}) we obtain
\begin{equation}
\lambda^k = {\rm e}^{-\frac{a}{c}
\left( \frac{\hat{\eta}}{t_d} + \frac{\hat{\nu}}{t_{(\ell-1)d}} \right)} +
\cO \left( \frac{1}{k} \right) \;.
\label{eq:lambdakasexp}
\end{equation}
Then by substituting (\ref{eq:lambda2}) and (\ref{eq:lambdakasexp}) into (\ref{eq:main1proof6}) we obtain
\begin{equation}
\det \left( P_{\cG^+[k,\chi]} \right) \varrho^{\sf T} B \mu =
\frac{t_{-d}}{{\rm e}^{\frac{a \hat{\nu}}{c t_{(\ell-1)d}}}}
\frac{\frac{\hat{\eta}}{t_d} \,{\rm e}^{-\frac{a \hat{\eta}}{c t_d}} +
\frac{\hat{\nu}}{t_{(\ell-1)d}} \,{\rm e}^{\frac{a \hat{\nu}}{c t_{(\ell-1)d}}}}
{\frac{\hat{\eta}}{t_d} + \frac{\hat{\nu}}{t_{(\ell-1)d}}} + \cO \left( \frac{1}{k} \right) \;.
\label{eq:main1proof8}
\end{equation}

Next we work in polar coordinates (\ref{eq:polarCoords}).
For clarity, we consider only the case $a < 0$.
The result for $a > 0$ can be obtained by switching signs in the expressions that follow appropriately.

Since $t_d < 0$, $t_{(\ell-1)d} < 0$, see (\ref{eq:tSigns}), and $c > 0$ (Lemma \ref{le:cPositive}),
with $a < 0$,
\begin{equation}
\hat{\eta} = \frac{c t_d}{a} \hat{r} \cos(\theta) \;, \qquad
\hat{\nu} = \frac{c t_{(\ell-1)d}}{a} \hat{r} \sin(\theta) \;,
\label{eq:polarCoordshataneg}
\end{equation}
where we let
\begin{equation}
\hat{r} = \frac{r}{k} \;.
\label{eq:hatr}
\end{equation}
Then by (\ref{eq:polarCoordshataneg}) and (\ref{eq:main1proof8}) we can write
\begin{equation}
\det \left( P_{\cG^+[k,\chi]} \right) \varrho^{\sf T} B \mu =
\frac{t_{-d}}{{\rm e}^{\hat{r} \sin(\theta)}}
H_1(\hat{r},\theta) \;,
\label{eq:main1proof9}
\end{equation}
where $H_1$ is a $C^K$ function and
\begin{equation}
H_1(\hat{r},\theta) = H_2(\hat{r},\theta) + \cO \left( \frac{1}{k} \right) \;,
\label{eq:H1}
\end{equation}
where
\begin{equation}
H_2(\hat{r},\theta) = \frac{\cos(\theta) \,{\rm e}^{-\hat{r} \cos(\theta)} + \sin(\theta) \,{\rm e}^{\hat{r} \sin(\theta)}}
{\cos(\theta) - \sin(\theta)} \;.
\label{eq:H2}
\end{equation}
It is a straight-forward exercise to show that
$H_2 \left( \Gamma(\theta), \theta \right) = 0$,
where $\Gamma$ is given by (\ref{eq:Gamma}) and $\theta \in \left( \frac{3 \pi}{2}, 2 \pi \right)$.

Next we employ the implicit function theorem to
find where the right hand-side of (\ref{eq:main1proof9}) is zero.
We define,
\begin{equation}
H_3(\hat{r},\theta,\ee) = k \ee H_1(\hat{r},\theta) + (1 - k \ee) H_2(\hat{r},\theta) \;,
\label{eq:H3}
\end{equation}
and we are interested in small values of $\ee \in \mathbb{R}$.
Notice $H_3$ is $C^K$,
$H_3 \left( \Gamma(\theta), \theta, 0 \right) = 0$, and 
$H_3(\hat{r},\theta,\ee) = H_2(\hat{r},\theta) + \cO(\ee)$.
Therefore, for any $\theta \in \left( \frac{3 \pi}{2}, 2 \pi \right)$,
there exists a neighbourhood of $(\hat{r},\ee) = (\Gamma(\theta),0)$
in which we can apply the implicit function theorem.
That is, there exists a unique $C^K$ function $\tilde{\Gamma}(\theta,\ee)$,
such that $H_3 \left( \tilde{\Gamma}(\theta,\ee), \theta, \ee \right) = 0$,
inside the neighbourhood, and $\tilde{\Gamma}(\theta,0) = \Gamma(\theta)$.
Then, assuming $k$ is sufficiently large,
$H_1 \left( \tilde{\Gamma} \left( \theta, \frac{1}{k} \right), \theta \right) \equiv 0$.
By (\ref{eq:main1proof9}), this shows that $\det \left( P_{\cG^+[k,\chi]} \right) = 0$
along a unique $C^K$ curve satisfying (\ref{eq:nearbyCurve}) and (\ref{eq:thetaSigns}).

To obtain the same result for
$\det \left( P_{\cG^+[k,\chi]^{\left( \left( \tilde{\ell}-1 \right) d_k^+ \right)}} \right)$,
we begin by using (\ref{eq:lkdkplusminus}) to write
\begin{equation}
\left( \tilde{\ell} - 1 \right) d_k^+ {\rm ~mod~} n_k^+ =
\ell d {\rm ~mod~} n + (\chi-1) n \;,
\end{equation}
where ``${\rm mod~} n_k^+$'' is not needed on the right hand-side
by assuming $1 \le \chi \le \chi_{\rm max}$.
Thus by (\ref{eq:Gplusplus2}), the first $\left( \tilde{\ell} - 1 \right) d_k^+ {\rm ~mod~} n_k^+$
symbols of $\cG^+[k,\chi]$ are given by $\left( \cS^{\overline{\ell d}} \right)^{\chi-1} \cX$,
and so we can write
\begin{equation}
x^{\cG^+[k,\chi]}_{\left( \tilde{\ell} - 1 \right) d_k^+} =
f^{\left( \cS^{\overline{\ell d}} \right)^{\chi-1} \cX}
\left( x^{\cG^+[k,\chi]}_0 \right) \;.
\label{eq:main1proof11}
\end{equation}
Substituting $x^{\cG^+[k,\chi]}_0 = \varphi_{-d} + h \zeta_{-d} + q$ (\ref{eq:xGplusForm})
into (\ref{eq:main1proof11}) gives
\begin{equation}
x^{\cG^+[k,\chi]}_{\left( \tilde{\ell} - 1 \right) d_k^+} =
f^{\left( \cS^{\overline{\ell d}} \right)^{\chi-1} \cX}
\left( \varphi_{-d} \right) + M_{\cX} M_{\cS^{\overline{\ell d}}}^{\chi-1}
\left( h \zeta_{-d} + q \right) \;.
\label{eq:main1proof12}
\end{equation}
By then substituting
$h = s^{\cG^+[k,\chi]}_0 + \cO \left( \sigma^k \right)$,
and $q = \cO \left( \sigma^k \right)$ (refer to the proof of Lemma \ref{le:sGall2})
into (\ref{eq:main1proof12}),
and multiplying both sides of (\ref{eq:main1proof12}) by
$e_1^{\sf T} \det \left( I - \cG^+[k,\chi] \right)$ on the left and using (\ref{eq:sFormulaEarly}),
we produce
\begin{align}
\det \left( P_{\cG^+[k,\chi]^{\left( \left( \tilde{\ell} - 1 \right) d_k^+ \right)}}
\right) \varrho^{\sf T} B \mu &=
e_1^{\sf T} f^{\left( \cS^{\overline{\ell d}} \right)^{\chi-1} \cX} \left( \varphi_{-d} \right)
\det \left( I - M_{\cG^+[k,\chi]} \right) \nonumber \\
&\quad+ e_1^{\sf T} M_{\cX} M_{\cS^{\overline{\ell d}}}^{\chi-1} \zeta_{-d}
\det \left( P_{\cG^+[k,\chi]} \right) \varrho^{\sf T} B \mu +
\cO \left( \sigma^k \right) \;.
\label{eq:main1proof13}
\end{align}
The solution to 
$\det \left( P_{\cG^+[k,\chi]^{\left( \left( \tilde{\ell} - 1 \right) d_k^+ \right)}} \right) = 0$
is the same, to leading order, as the solution to
$\det \left( P_{\cG^+[k,\chi]} \right) = 0$,
because $f^{\left( \cS^{\overline{\ell d}} \right)^{\chi-1} \cX} \left( \varphi_{-d} \right) =
y_{\ell d} + \cO \left( \frac{1}{k} \right)$
and so the first term in (\ref{eq:main1proof13}) is higher order
than the term involving $\det \left( P_{\cG^+[k,\chi]} \right)$\removableFootnote{
Here I might need to use the fact that
$e_1^{\sf T} M_{\cX} M_{\cS^{\overline{\ell d}}}^{\chi-1} \zeta_{-d} \ne 0$.
This should be proved in my subsequent skew sawtooth paper.
}.
Therefore near $(\eta,\nu) = (0,0)$ there exists a unique $C^K$ curve satisfying
(\ref{eq:nearbyCurve}) and (\ref{eq:thetaSigns}) along which 
$\det \left( P_{\cG^+[k,\chi]^{\left( \left( \tilde{\ell} - 1 \right) d_k^+ \right)}} \right) = 0$.
\end{proof}

\begin{proof}[Proof of Theorem \ref{th:main3}]
For brevity we restrict our attention to $\cG^+[k,\chi]$ with $0 \le \chi \le \chi_{\rm max}$.

By Proposition \ref{pr:MSPSsingular}(i), $\{ \tilde{y}_i \}$ is a $\cG^+[k,\chi]$-cycle.
Thus $\tilde{y}_0$ maps to $\tilde{y}_{\tilde{\ell} d_k^+}$
under $f^L$ and $f^R$ in the order specified by the first
$\tilde{\ell} d_k^+ {\rm ~mod~} n_k^+$ symbols of $\cG^+[k,\chi]$.
Since $0 \le \chi \le \chi_{\rm max}$, by (\ref{eq:lkdkplusminus}),
$\tilde{\ell} d_k^+ {\rm ~mod~} n_k^+ = \ell d {\rm ~mod~} n + \chi n$.
Thus, by (\ref{eq:Gplus}), the first $\tilde{\ell} d_k^+ {\rm ~mod~} n_k^+$ symbols of $\cG^+[k,\chi]$
are $\left( \cS^{\overline{\ell d}} \right)^{\chi} \cX$, and therefore
\begin{equation}
\tilde{y}_{\tilde{\ell} d_k^+} =
f^{\left( \cS^{\overline{\ell d}} \right)^{\chi} \cX} \left( \tilde{y}_0 \right) \;.
\label{eq:main3proof2}
\end{equation}
Also, by Proposition \ref{pr:MSPSsingular}(i),
$\{ \tilde{y}_i \}$ is a $\cG^+[k,\chi]^{\left( -d_k^+ \right)}$-cycle.
Thus $\tilde{y}_{d_k^+}$ maps to $\tilde{y}_{\left( \tilde{\ell}+1 \right) d_k^+}$
following the first $\tilde{\ell} d_k^+ {\rm ~mod~} n_k^+$ symbols of $\cG^+[k,\chi]$.
That is,
\begin{equation}
\tilde{y}_{\left( \tilde{\ell}+1 \right) d_k^+} =
f^{\left( \cS^{\overline{\ell d}} \right)^{\chi} \cX} \left( \tilde{y}_{d_k^+} \right) \;.
\label{eq:main3proof3}
\end{equation}
In addition
\begin{equation}
\tilde{y}_{d_k^+} = f^{\cS^{(-d)}} \left( \tilde{y}_0 \right) \;,
\label{eq:main3proof4}
\end{equation}
because $d_k^+ = n$ (see Lemma \ref{le:mndkpm})
and the first $n$ symbols of $\cG^+[k,\chi]^{\overline{0}}$ are
$\cS^{\overline{\ell d} \,\overline{0}} = \cX^{\overline{0}} \cY^{\overline{0}} = \cS^{(-d)}$,
by (\ref{eq:Gplus}) and (\ref{eq:partitions3}).

In the form $\tilde{y}_0 = \varphi_{-d} + h \zeta_{-d} + q$ (\ref{eq:xEigCoords}),
we have $q = \cO \left( \sigma^k \right)$
(for the same reasons as for $x^{\cG^+[k,\chi]}_0$ in the proof of Lemma \ref{le:sGall2}), thus
\begin{equation}
\tilde{y}_0 = \varphi_{-d} + \cO \left( \sigma^k \right) \;,
\label{eq:main3proof5}
\end{equation}
because also $e_1^{\sf T} \tilde{y}_0 = 0$,
$e_1^{\sf T} \varphi_{-d} = 0$, and
$e_1^{\sf T} \zeta_{-d} = 1$.
By (\ref{eq:slowDynsGen}) and (\ref{eq:main3proof4}),
\begin{equation}
\tilde{y}_{d_k^+} = \varphi_{-d} + \gamma_{-d} \zeta_{-d} + \cO \left( \sigma^k \right) \;,
\label{eq:main3proof6}
\end{equation}
and by (\ref{eq:main3proof3}), (\ref{eq:main3proof5}) and (\ref{eq:main3proof6}),
\begin{align}
\tilde{y}_{\left( \tilde{\ell}+1 \right) d_k^+} &=
f^{\left( \cS^{\overline{\ell d}} \right)^{\chi} \cX}
\left( \tilde{y}_0 + \gamma_{-d} \zeta_{-d} + \cO \left( \sigma^k \right) \right) \nonumber \\
&= f^{\left( \cS^{\overline{\ell d}} \right)^{\chi} \cX} \left( \tilde{y}_0 \right) +
\gamma_{-d} M_{\cX} M_{\cS^{\overline{\ell d}}}^{\chi} \zeta_{-d} +
\cO \left( \sigma^k \right) \;.
\label{eq:main3proof7}
\end{align}
By then multiplying (\ref{eq:main3proof7}) by $e_1^{\sf T}$ on the left
and using (\ref{eq:main3proof2}) and $e_1^{\sf T} \tilde{y}_{\tilde{\ell} d_k^+} = 0$, we obtain
\begin{equation}
\tilde{t}_{\left( \tilde{\ell}+1 \right) d_k^+} =
\gamma_{-d} e_1^{\sf T} M_{\cX}(0,0) M_{\cS^{\overline{\ell d}}}^{\chi}(0,0) v_{-d} +
\cO \left( \frac{1}{k^2} \right) \;,
\label{eq:main3proof8}
\end{equation}
where we have also used $\zeta_{-d}(0,0) = v_{-d}$.

Our next step is to derive the following identity\removableFootnote{
More generally we have
\begin{align}
u_0^{\sf T} \left( I - M_{\cY^{\overline{0}}}(0,0) M_{\cX}(0,0) \right) &=
\frac{b t_d e_1^{\sf T} M_{\cX}(0,0)}{c t_{(\ell+1)d}} \;,  \\
u_{\ell d}^{\sf T} \left( I - M_{\cX^{\overline{0}}}(0,0) M_{\cY}(0,0) \right) &=
\frac{a t_{(\ell+1)d} e_1^{\sf T} M_{\cY}(0,0)}{c t_d} \;, \\
u_{-d}^{\sf T} \left( I -  M_{\cY}(0,0) M_{\cX^{\overline{0}}}(0,0) \right) &=
\frac{a t_{-d} e_1^{\sf T} M_{\cX^{\overline{0}}}(0,0)}{c t_{(\ell-1)d}} \;, \\
u_{(\ell-1)d}^{\sf T} \left( I - M_{\cX}(0,0) M_{\cY^{\overline{0}}}(0,0) \right) &=
\frac{b t_{(\ell-1)d} e_1^{\sf T} M_{\cY^{\overline{0}}}(0,0)}{c t_{-d}} \;.
\end{align}	
}
\begin{equation}
u_0^{\sf T} \left( I - M_{\cS^{\overline{\ell d}}}(0,0) \right) =
\frac{b t_d e_1^{\sf T} M_{\cX}(0,0)}{c t_{(\ell+1)d}} \;.
\label{eq:u0Identity}
\end{equation}
By substituting (\ref{eq:uvj}) for $u_{\ell d}^{\sf T}$
into (\ref{eq:vu0}) we obtain
\begin{equation}
u_0^{\sf T} = \frac{t_d e_1^{\sf T} {\rm adj} \left( I - M_{\cS^{(\ell d)}}(0,0) \right) M_{\cX}(0,0)}
{c t_{(\ell+1)d}} \;.
\end{equation}
We have $M_{\cS^{(\ell d)}} = M_{\cX} M_{\cY}$ (\ref{eq:partitions1}),
and in view of (\ref{eq:MS}) and (\ref{eq:adjIdentity2}) we can
substitute $M_{\cY}$ for $M_{\cY^{\overline{0}}}$ to obtain
\begin{equation}
u_0^{\sf T} = \frac{t_d e_1^{\sf T} {\rm adj}
\left( I - M_{\cX}(0,0) M_{\cY^{\overline{0}}}(0,0) \right) M_{\cX}(0,0)}
{c t_{(\ell+1)d}} \;.
\label{eq:u0IdentityProof1}
\end{equation}
Note, $\det \left( I - M_{\cX}(0,0) M_{\cY^{\overline{0}}}(0,0) \right) = b$, because
$M_{\cX} M_{\cY^{\overline{0}}} =
M_{\cS^{(\ell d) \overline{0}}} =
M_{\cS^{\overline{\ell d} (\ell d)}}$,
refer to (\ref{eq:ab}), (\ref{eq:partitions2}) and Lemma (\ref{le:eigMSindep}).
Thus by (\ref{eq:adjIdentity}) and (\ref{eq:u0IdentityProof1}),
\begin{equation}
u_0^{\sf T} = \frac{b t_d e_1^{\sf T}
\left( I - M_{\cX}(0,0) M_{\cY^{\overline{0}}}(0,0) \right)^{-1} M_{\cX}(0,0)}
{c t_{(\ell+1)d}} \;.
\label{eq:u0IdentityProof2}
\end{equation}
By substituting
$\left( I - M_{\cX} M_{\cY^{\overline{0}}} \right)^{-1} M_{\cX} =
M_{\cX} \left( I - M_{\cY^{\overline{0}}} M_{\cX} \right)^{-1}$
and multiply both sides of (\ref{eq:u0IdentityProof2})
by $I - M_{\cY^{\overline{0}}} M_{\cX}$ on the right we arrive at (\ref{eq:u0Identity}).

By substituting (\ref{eq:u0Identity}) into (\ref{eq:main3proof8}) we obtain
\begin{equation}
\tilde{t}_{\left( \tilde{\ell}+1 \right) d_k^+} =
\frac{\gamma_{-d} c t_{(\ell+1)d} u_0^{\sf T}
\left( I - M_{\cS^{\overline{\ell d}}}(0,0) \right)
M_{\cS^{\overline{\ell d}}}^{\chi}(0,0) v_{-d}}{b t_d} +
\cO \left( \frac{1}{k^2} \right) \;.
\label{eq:main3proof10}
\end{equation}
By using (\ref{eq:fourtIdentity}) and
$\gamma_{-d} = \frac{a t_{-d} \nu_{\cG^+[k,\chi]}}{c t_{(\ell-1)d}} +
\cO \left( \frac{1}{k^2} \right)$ (\ref{eq:gammamd}),
(\ref{eq:main3proof10}) simplifies to
\begin{equation}
\tilde{t}_{\left( \tilde{\ell}+1 \right) d_k^+} =
\nu_{\cG^+[k,\chi]} \left( \kappa^+_{\chi+1} - \kappa^+_{\chi} \right) +
\cO \left( \frac{1}{k^2} \right) \;,
\label{eq:main3proof11}
\end{equation}
where we have substituted (\ref{eq:kappaPlus}).

If $a > 0$, then $\theta^+_{\chi} \in \left( \frac{\pi}{2}, \pi \right)$
(see Table \ref{tb:thetaDomains}),
thus $\sin \left( \theta^+_{\chi} \right) > 0$,
and so by (\ref{eq:polarCoords})
$\nu_{\cG^+[k,\chi]} > 0$ for arbitrarily large values of $k$.
Since $\kappa^+_{\chi} > 0$ and $\tilde{t}_{\left( \tilde{\ell}+1 \right) d_k^+} > 0$
(by the assumption that $\left( \eta_{\cG^+[k,\chi]}, \nu_{\cG^+[k,\chi]} \right)$ is
a $\cG^+[k,\chi]$-shrinking point),
by (\ref{eq:main3proof11}) we must have $\kappa^+_{\chi+1} > 0$, as claimed.

If $\chi \ge 1$, we can similarly show that
\begin{equation}
\tilde{t}_{\left( \tilde{\ell}-1 \right) d_k^+} =
\nu_{\cG^+[k,\chi]} \left( \kappa^+_{\chi-1} - \kappa^+_{\chi} \right) +
\cO \left( \frac{1}{k^2} \right) \;,
\label{eq:main3proof12}
\end{equation}
based on the knowledge that $\tilde{y}_0$ and $\tilde{y}_{d_k^+}$
map to $\tilde{y}_{\left( \tilde{\ell}-1 \right) d_k^+}$ and $\tilde{y}_{\tilde{\ell} d_k^+}$,
respectively, under $f^{\left( \cS^{\overline{\ell d}} \right)^{\chi-1} \cX}$.
Then if $a < 0$, $\nu_{\cG^+[k,\chi]} < 0$, for arbitrarily large values of $k$,
and thus since $\kappa^+_{\chi} > 0$ and $\tilde{t}_{\left( \tilde{\ell}-1 \right) d_k^+} < 0$,
by (\ref{eq:main3proof12}) we must have $\kappa^+_{\chi-1} > 0$, as claimed.

The remaining cases can be proved in the same fashion.
\end{proof}

\begin{proof}[Proof of Theorem \ref{th:main4}]
Part (ii) of the theorem is an immediate consequence of Lemma \ref{le:eigenvalues}.
We prove part (i) for $\cG^+[k,\chi]$ with $0 \le \chi \le \chi_{\rm max}$.
Other cases may be proved in a similar fashion.

We first show that ${\rm sgn}(\tilde{a}) = {\rm sgn}(a)$.
At $(\eta,\nu) = \left( \eta_{\cT},\nu_{\cT} \right)$,
$\det \left( I - M_{\cG^+[k,\chi]} \right) = 0$ and
$\det \left( I - M_{\cG^+[k,\chi+1]} \right) = \tilde{b}$\removableFootnote{
We do computations for $\tilde{b}$ rather than $\tilde{a}$
(which would be more natural) because
(i) above we provided the full derivation of (\ref{eq:main3proof11}) (which relates to $\tilde{b}$)
but not of (\ref{eq:main3proof12}) (which relates to $\tilde{a}$),
(ii) then we do not need to consider $\chi = 0$ and $\chi \ge 1$ separately
when we apply Lemma \ref{le:detAll}.
}.
Thus, by (\ref{eq:detImMGplus}) with $\rho = 1$,
\begin{align}
1 - \left. \left( \lambda^{k-\chi} \omega_0^{\sf T} M_{\cS^{\overline{\ell d}}}^{\chi}
M_{\hat{\cX}} \zeta_0 \right) \right|_{\left( \eta_{\cT}, \nu_{\cT} \right)} +
\cO \left( \sigma^k \right) &= 0 \;, \label{eq:sgnatildeproof1} \\
1 - \left. \left( \lambda^{k-\chi-1} \omega_0^{\sf T} M_{\cS^{\overline{\ell d}}}^{\chi+1}
M_{\hat{\cX}} \zeta_0 \right) \right|_{\left( \eta_{\cT}, \nu_{\cT} \right)} +
\cO \left( \sigma^k \right) &= \tilde{b} \;. \label{eq:sgnatildeproof2}
\end{align}
By combining (\ref{eq:sgnatildeproof1}) and (\ref{eq:sgnatildeproof2})
and using (\ref{eq:vumd2}) we obtain
\begin{equation}
\tilde{b} = 1 - \frac{u_0^{\sf T} M_{\cS^{\overline{\ell d}}}^{\chi+1}(0,0) v_{-d}}
{u_0^{\sf T} M_{\cS^{\overline{\ell d}}}^{\chi}(0,0) v_{-d}} +
\cO \left( \frac{1}{k} \right) \;.
\end{equation}
Then by (\ref{eq:kappaPlus}),
\begin{equation}
\tilde{b} = 1 - \frac{\kappa^+_{\chi+1}}{\kappa^+_{\chi}} +
\cO \left( \frac{1}{k} \right) \;,
\end{equation}
and by (\ref{eq:main3proof11}),
\begin{equation}
\tilde{b} = \frac{-\tilde{t}_{\left( \tilde{\ell}+1 \right) d_k^+}}
{\kappa^+_{\chi} \nu_{\cG^+[k,\chi]}} +
\cO \left( \frac{1}{k} \right) \;.
\label{eq:sgnatildeproof5}
\end{equation}
Since $\kappa^+_{\chi} > 0$ and $\tilde{t}_{\left( \tilde{\ell}+1 \right) d_k^+} > 0$,
(\ref{eq:sgnatildeproof5}) tells us that
${\rm sgn}(\tilde{b}) = -{\rm sgn} \left( \nu_{\cG^+[k,\chi]} \right)$.
Hence by (\ref{eq:polarCoords}) and Table \ref{tb:thetaDomains},
${\rm sgn}(\tilde{b}) = -{\rm sgn}(a)$.
Since ${\rm sgn}(a) = -{\rm sgn}(b)$, for any shrinking point (\ref{eq:fourtIdentity}),
we have ${\rm sgn}(\tilde{a}) = -{\rm sgn}(\tilde{b})$, and therefore
\begin{equation}
{\rm sgn}(\tilde{a}) = {\rm sgn}(a) \;,
\label{eq:sgna}
\end{equation}
as required.

Next we derive an explicit expression for $p_1$ (a coefficient of a
leading order term in (\ref{eq:tildeeta})) from which we can ascertain the sign of $p_1$.
The desired result ($\det \left( \tilde{J} \right) > 0$)
then follows from Lemma \ref{le:detJtilde} and some additional identities.

By (\ref{eq:rss0}), $\cG^+[k,\chi]^{\overline{0}} = \cG^+[k,\chi-1]^{(-d_k^+)}$.
Therefore by (\ref{eq:Deltaetanu}) and (\ref{eq:tildeeta}) we can write
\begin{equation}
p_1 = \lim_{k \to \infty} \frac{1}{k} \frac{\partial s^{\cG^+[k,\chi - 1]}_{-d_k^+}}{\partial \eta}
\bigg|_{\left( \eta_{\cG^+[k,\chi]}, \nu_{\cG^+[k,\chi]} \right)} \;.
\label{eq:p12}
\end{equation}
Here we use Lemma \ref{le:sGall2} to evaluate $s^{\cG^+[k,\chi - 1]}_{-d_k^+}$.
This requires separate calculations for the cases $\chi = 0$ and $\chi \ge 1$\removableFootnote{
In view of the way we have defined $\tilde{\eta}$ and $\tilde{\nu}$,
and the range of $\chi$ values in the formulas of Lemma \ref{le:sGall2},
we have to treat the cases $\chi = 0$ and $\chi \ge 1$ separately,
whereas the case $\chi \le -1$ need not be separated.
}.
If $\chi \ge 1$, we express $s^{\cG^+[k,\chi-1]}_{-d_k^+}$ in terms of
$s^{\cG^+[k,\chi-1]}_0$ so that we can apply (\ref{eq:sGplusplus}).
If $\chi = 0$, we express $s^{\cG^+[k,-1]}_{-d_k^+}$ in terms of
$s^{\cG^+[k,-1]}_{\left( \ell_k^+-1 \right) d_k^+}$ so that we can apply (\ref{eq:sGplusminus}).
For brevity here we provide details only for the case $\chi \ge 1$\removableFootnote{
For the case $\chi = 0$, let us merely note that it may be shown that
\begin{equation}
s^{\cG^+[k,-1]}_{-d_k^+} = \frac{s^{\cG^+[k,-1]}_{\left( \ell_k^+-1 \right) d_k^+} -
e_1^{\sf T} f^{\cX^{\overline{0}}} \left( \varphi_{-d} \right)}
{e_1^{\sf T} M_{\cX^{\overline{0}}} \zeta_{-d}} + \cO \left( \sigma^k \right) \;,
\end{equation}
which, by (\ref{eq:vuellm1d}) and (\ref{eq:p12}), gives
\begin{equation}
p_1 = \lim_{k \to \infty} \frac{1}{k} \frac{t_{-d}}{t_{(\ell-1)d}}
\frac{\partial s^{\cG^+[k,-1]}_{\left( \ell_k^+-1 \right) d_k^+}}{\partial \eta}
\bigg|_{\left( \eta_{\cG^+[k,0]}, \nu_{\cG^+[k,0]} \right)} \;.
\end{equation}
By then using (\ref{eq:sGplusminus}) we are led to (\ref{eq:p1}),
with $\tilde{a} = \frac{a}{c \kappa^+_0}$.
}.

In the proof of Lemma \ref{le:sGall2}, it was shown that $x^{\cG^+[k,\chi-1]}_0$
lies within $\cO \left( \sigma^k \right)$ of the slow manifold with $j = -d$,
see (\ref{eq:xGplusForm}) and (\ref{eq:sGplusplusproof10}).
The same is true for $x^{\cG^+[k,\chi-1]}_{-d_k^+}$, because
$x^{\cG^+[k,\chi-1]}_0 = f^{\cS^{(-d)}} \left( x^{\cG^+[k,\chi-1]}_{-d_k^+} \right)$.
That is,
\begin{equation}
x^{\cG^+[k,\chi-1]}_i = \varphi_{-d} +
s^{\cG^+[k,\chi-1]}_i \zeta_{-d} + \cO \left( \sigma^k \right) \;,
\label{eq:p1proof2}
\end{equation}
for $i = 0$ and $i = -d_k^+$.
By (\ref{eq:slowDynsGen}), we obtain
\begin{equation}
s^{\cG^+[k,\chi-1]}_{-d_k^+} = \frac{s^{\cG^+[k,\chi-1]}_0 - \gamma_{-d}}
{\lambda} + \cO \left( \sigma^k \right) \;.
\label{eq:p1proof3}
\end{equation}
Since $\lambda = 1 + \cO \left( \frac{1}{k} \right)$,
and $\gamma_{-d}$ is independent of $k$, by (\ref{eq:p12})
\begin{equation}
p_1 = \lim_{k \to \infty} \frac{1}{k} \frac{\partial s^{\cG^+[k,\chi - 1]}_0}{\partial \eta}
\bigg|_{\left( \eta_{\cG^+[k,\chi]}, \nu_{\cG^+[k,\chi]} \right)} \;.
\label{eq:p13}
\end{equation}

To evaluate (\ref{eq:p13}) we use (\ref{eq:sFormula}) to write
\begin{equation}
s^{\cG^+[k,\chi-1]}_0 = \frac{\det \left( P_{\cG^+[k,\chi-1]} \right) \varrho^{\sf T} B \mu}
{\det \left( I - M_{\cG^+[k,\chi-1]} \right)} \;.
\label{eq:p1proof4}
\end{equation}
Here it is sufficient to write the denominator of (\ref{eq:p1proof4}) as
\begin{equation}
\det \left( I - M_{\cG^+[k,\chi-1]} \right) =
\tilde{a} + \cO \left( \Delta \eta, \Delta \nu \right) \;.
\label{eq:denominator}
\end{equation}
Since we assuming $\chi \ge 1$, by (\ref{eq:sGplusplus})
the numerator of (\ref{eq:p1proof4}) is
\begin{equation}
\det \left( P_{\cG^+[k,\chi-1]} \right) \varrho^{\sf T} B \mu =
\gamma_{-d} \sum_{j=0}^{k-\chi} \lambda^j +
\omega_{-d}^{\sf T} \left( f^{\left( \cS^{\overline{\ell d}} \right)^{\chi-1} \hat{\cX}}
\left( \varphi_{-d} \right) - \varphi_{-d} \right) \lambda^{k - \chi+1} +
\cO \left( \sigma^k \right) \;.
\label{eq:numerator}
\end{equation}
We now evaluate the components of (\ref{eq:numerator}).
We have $\varphi_{-d} = y_0 + \cO \left( \frac{1}{k} \right)$,
and so
$f^{\left( \cS^{\overline{\ell d}} \right)^{\chi-1} \hat{\cX}} \left( \varphi_{-d} \right) =
f^{\left( \cS^{\overline{\ell d}} \right)^{\chi-1} \hat{\cX}} \left( y_0 \right) +
\cO \left( \sigma^k \right) =
f^{\hat{\cX}} \left( y_0 \right) +
\cO \left( \sigma^k \right) =
y_{-d} + \cO \left( \frac{1}{k} \right)$.
Thus
\begin{equation}
\omega_{-d}^{\sf T} \left( f^{\left( \cS^{\overline{\ell d}} \right)^{\chi-1} \hat{\cX}}
\left( \varphi_{-d} \right) - \varphi_{-d} \right) = t_{-d} + \cO \left( \frac{1}{k} \right) \;.
\label{eq:numeratorTerm1}
\end{equation}
By (\ref{eq:gamma}), and the formula for the sum of a truncated geometric series,
when $\lambda \ne 1$,
\begin{equation}
\gamma_{-d} \sum_{j=0}^{k-\chi} \lambda^j =
e_1^{\sf T} x^{\cS}_{-d} \left( 1 - \lambda^{k-\chi+1} \right) \;.
\label{eq:numeratorTerm2}
\end{equation}
Since $\lambda = 1 + \cO \left( \frac{1}{k} \right)$,
and $\chi$ is independent of $k$, we can write
$\lambda^{k-\chi+1} = \lambda^k + \cO \left( \frac{1}{k} \right)$.
By (\ref{eq:detImMTnearbyproof1}) and (\ref{eq:pq12proof1}),
\begin{equation}
\lambda^k = \frac{-t_d \nu_{\cT}}{t_{(\ell-1)d} \eta_{\cT}}
\left( 1 - \left( \frac{a}{c t_d} k + \cO(1) \right) \Delta \eta -
\left( \frac{a}{c t_{(\ell-1)d}} k + \cO(1) \right) \Delta \nu +
\cO \left( \left( \Delta \eta, \Delta \nu \right)^2 \right) \right) \;,
\label{eq:lambdak3}
\end{equation}
Also by (\ref{eq:xSi}),
\begin{align}
e_1^{\sf T} x^{\cS}_{-d} &=
\frac{\frac{t_{-d}}{t_{(\ell-1)d}} \nu + \cO \left( \left( \eta, \nu \right)^2 \right)}
{\frac{1}{t_d} \eta + \frac{1}{t_{(\ell-1)d}} \nu + \cO \left( \left( \eta, \nu \right)^2 \right)} \nonumber \\
&= \frac{t_{-d} \nu_{\cT}}{t_{(\ell-1)d} \eta_{\cT}
\left( \frac{1}{t_d} + \frac{\nu_{\cT}}{t_{(\ell-1)d} \eta_{\cT}} \right)} + \cO \left( \frac{1}{k} \right) +
\left( -\frac{t_{-d} \nu_{\cT}}{t_d t_{(\ell-1)d} \eta_{\cT}^2
\left( \frac{1}{t_d} + \frac{\nu_{\cT}}{t_{(\ell-1)d} \eta_{\cT}} \right)^2} + \cO(1) \right) \Delta \eta \nonumber \\
&\quad+
\left( \frac{t_{-d}}{t_d t_{(\ell-1)d} \eta_{\cT}
\left( \frac{1}{t_d} + \frac{\nu_{\cT}}{t_{(\ell-1)d} \eta_{\cT}} \right)^2} + \cO(1) \right) \Delta \nu +
\cO \left( \left( \Delta \eta, \Delta \nu \right)^2 \right) \;.
\label{eq:numeratorTerm5}
\end{align}
Finally by (\ref{eq:polarCoords}) and (\ref{eq:thetaPlus}), since $\kappa^+_{\chi} > 0$,
\begin{equation}
\frac{\nu_{\cT}}{\eta_{\cT}} = \frac{t_{(\ell-1)d} \tan \left( \theta^+_{\chi} \right)}{t_d} +
\cO \left( \frac{1}{k} \right) \;.
\label{eq:p1proof5}
\end{equation}
By substituting (\ref{eq:numeratorTerm1})-(\ref{eq:p1proof5}) into (\ref{eq:numerator}) we obtain
\begin{align}
\det \left( P_{\cG^+[k,\chi-1]} \right) \varrho^{\sf T} B \mu &=
\left( \frac{\frac{a t_{-d} k}{c t_d} - \frac{t_{-d}}{\eta_{\cT}}}
{1 + \frac{1}{\tan \left( \theta^+_{\chi} \right)}} + \cO(1) \right) \Delta \eta \nonumber \\
&+\quad
\left( \frac{\frac{a t_{-d} k}{c t_{(\ell-1)d}} +
\frac{t_d t_{-d}}{t_{(\ell-1)d} \eta_{\cT} \tan \left( \theta^+_{\chi} \right)}}
{1 + \frac{1}{\tan \left( \theta^+_{\chi} \right)}} + \cO(1) \right) \Delta \nu +
\cO \left( \left( \Delta \eta, \Delta \nu \right)^2 \right) \;. 
\label{eq:numerator2}
\end{align}
By substituting (\ref{eq:denominator}) and (\ref{eq:numerator2}) into (\ref{eq:p1proof4}),
and then evaluating (\ref{eq:p13}), we arrive at
\begin{equation}
p_1 = \frac{a t_{-d} \left( 1 + \frac{{\rm sgn}(a)}
{\Gamma \left( \theta^+_{\chi} \right) \cos \left( \theta^+_{\chi} \right)} \right)}
{\tilde{a} c t_d \left( 1 + \frac{1}{\tan \left( \theta^+_{\chi} \right)} \right)} \;,
\label{eq:p1}
\end{equation}
where we have used
$\lim_{k \to \infty} k \eta_{\cT} = \left| \frac{c t_d}{a} \right|
\Gamma \left( \theta^+_{\chi} \right) \cos \left( \theta^+_{\chi} \right)$,
which is due to (\ref{eq:polarCoords}) and (\ref{eq:nearbyCurve}).

We now use (\ref{eq:p1}) to show that $p_1 > 0$.
From the definition of $\Gamma$ (\ref{eq:Gamma})-(\ref{eq:GammaExtended}),
and a bit of care with the different cases of the sign of $a$, see Table \ref{tb:thetaDomains},
we obtain\removableFootnote{
Let $\check{\theta} = \theta {\rm ~mod~} \frac{\pi}{2}$.
For either sign of $a$,
$\tan(\theta) = \frac{-1}{\tan(\check{\theta})}$
If $a < 0$ then $\cos(\theta) = \sin(\check{\theta})$,
whereas if $a > 0$ then $\cos(\theta) = -\sin(\check{\theta})$.
Then from the definition of $\Gamma$ we can obtain a formula
for the left hand-side (\ref{eq:tanQuotient}) in terms of $\tan(\check{\theta})$.
This formula is independent of the sign of $a$, and leads to the right hand-side of (\ref{eq:tanQuotient}).
}
\begin{equation}
\frac{1 + \frac{{\rm sgn}(a)}
{\Gamma \left( \theta^+_{\chi} \right) \cos \left( \theta^+_{\chi} \right)}}
{1 + \frac{1}{\tan \left( \theta^+_{\chi} \right)}} =
\frac{\tan \left( \theta^+_{\chi} \right)}{1 + \tan \left( \theta^+_{\chi} \right)} +
\frac{\tan \left( \theta^+_{\chi} \right)}{\ln \left( -\tan \left( \theta^+_{\chi} \right) \right)} \;,
\label{eq:tanQuotient}
\end{equation}
which has a negative value.
Since $c > 0$ (Lemma \ref{le:cPositive}),
$t_d < 0$, $t_{-d} > 0$ (\ref{eq:tSigns}),
and ${\rm sgn}(\tilde{a}) = {\rm sgn}(a)$ (\ref{eq:sgna}),
from (\ref{eq:p1}) and (\ref{eq:tanQuotient}) we conclude that $p_1 > 0$.

Finally, since $\tilde{t}_{d_k} < 0$ and $\tilde{t}_{\left( \tilde{\ell}-1 \right) d_k} < 0$,
by (\ref{eq:tildedetImMT2}) and (\ref{eq:tildedetImMT}),
\begin{equation}
{\rm sgn}(r_1) = {\rm sgn}(r_2) = -{\rm sgn}(\tilde{a}) = -{\rm sgn}(a) \;.
\label{eq:sgnr1sgna}
\end{equation}
Since $p_1 > 0$, by (\ref{eq:consistencyCondition}) and (\ref{eq:sgnr1sgna}), $q_1 < 0$.
From (\ref{eq:q1q2}) we obtain\removableFootnote{
If we wanted to we could write this as
\begin{equation}
q_2 = \frac{t_d}{t_{(\ell-1)d}}
\frac{1}{\tan \left( \theta^+_{\chi} \right)}
\left( \frac{1}{\frac{1}{\ln \left( -\tan \left( \theta^+_{\chi} \right) \right)} +
\frac{1}{\tan \left( \theta^+_{\chi} \right) + 1}} - 1 \right) q_1 \;,
\end{equation}
and then we could relate it to (\ref{eq:tanQuotient}).
}
\begin{equation}
q_2 = \frac{t_d}{t_{(\ell-1)d}}
\frac{1 - \frac{\tan \left( \theta^+_{\chi} \right) + 1}
{\tan \left( \theta^+_{\chi} \right) \ln \left( -\tan \left( \theta^+_{\chi} \right) \right)}}
{1 + \frac{\tan \left( \theta^+_{\chi} \right) + 1}
{\ln \left( -\tan \left( \theta^+_{\chi} \right) \right)}} q_1 \;.
\label{eq:q2}
\end{equation}
Since $t_d < 0$, $t_{(\ell-1)d} < 0$ and $q_1 < 0$, by (\ref{eq:q2}) we must have $q_2 > 0$.
Therefore $\frac{q_2}{t_d} - \frac{q_1}{t_{(\ell-1)d}} < 0$,
and so by (\ref{eq:detJtilde}) and (\ref{eq:sgnr1sgna}),
we have $\det \left( \tilde{J} \right) > 0$, as required.
\end{proof}

\end{document}